\documentclass[10pt,reqno]{amsart}
\usepackage{amsfonts}
\usepackage{amsmath}
\usepackage{CJK}
\usepackage{amssymb}
\usepackage{latexsym,bm}
\usepackage[RGB]{xcolor}
\usepackage{mathrsfs}
\usepackage{amsthm}
\usepackage{hyperref}
\usepackage[normalem]{ulem}
\allowdisplaybreaks[4]
\numberwithin{equation}{section}
\newtheorem{theorem}{Theorem}[section]
\newtheorem{assumption}[theorem]{Assumption}
\newtheorem{lemma}[theorem]{Lemma}
\newtheorem{remark}{Remark}[section]
\newtheorem{corollary}[theorem]{Corollary}
\newtheorem{proposition}[theorem]{Proposition}

\begin{document}
	\title[Spherically symmetric solutions for relaxed CNS]{Global Spherically Symmetric Solutions and Relaxation Limit for the Relaxed Compressible Navier-Stokes Equations}
	\author{Yuxi Hu and Mengran Yuan}
	\thanks{\noindent Yuxi Hu, Department of Mathematics, China University of Mining and Technology, Beijing, 100083 and Key Laboratory of Scientific and Engineering Computing (Ministry of Education), P.R. China, yxhu86@163.com\\
		\indent Mengran Yuan, Department of Mathematics, China University of Mining and Technology, Beijing, 100083, P.R. China,  ymengran2388@126.com
	}
	\thanks{\noindent  }
	\begin{abstract}
		This paper studies an initial boundary value problem for the multidimensional hyperbolized compressible Navier-Stokes equations, in which the classical Newtonian law is replaced by the Maxwell law. We seek spherically symmetric solutions to the studied system in an exterior domain of a ball in $\mathbb R^3$, which are a system possessing a uniform characteristic boundary.
		First, we construct an approximate system featuring a non-characteristic boundary and establish its local well-posedness. Subsequently, by defining a suitable weighted energy functional and carefully handling boundary terms, we derive uniform a priori estimates, enabling the proof of uniform global existence. Leveraging these uniform estimates alongside standard compactness arguments, we establish the global well-posedness of the original system. Additionally, we rigorously justify the global relaxation limit.  \\
		{\bf Keywords}:  Initial boundary value problem; compressible Navier-Stokes equations; Maxwell law; global solution; relaxation limit; \\
		{\bf AMS classification code}: 35A01, 35L50, 35Q30 
	\end{abstract}
	\maketitle

	\section{Introduction}\label{sec1}
	The compressible isentropic Navier-Stokes equations in the domain \(\mathbb{R}^n \times \mathbb{R}^{+}\) are derived from the fundamental principles of mass conservation and momentum balance. These equations are given by:
	
	\begin{equation} \label{1}
		\begin{cases}
			\partial_t \rho + \operatorname{div}(\rho u) = 0, \\[1mm]
			\partial_t (\rho u) + \operatorname{div}(\rho u \otimes u) + \nabla P = \operatorname{div} S,
		\end{cases}
	\end{equation}
	where \((\rho, u, P, S)\) represent the fluid density, velocity, pressure, and stress tensor, respectively. The pressure \(P\) is assumed to follow the standard \(\gamma\)-law, \(P(\rho) = A \rho^\gamma\) with \(A =1\) without loss of generality. To close the system \eqref{1}, a constitutive equation for the stress tensor \(S\) must be specified. For simple fluids, \(S\) is determined by the Newtonian law:
	
	\begin{equation}\label{h1}
		S = \mu \left( \nabla u + \nabla u^T - \frac{2}{n} \operatorname{div} u I_n \right) + \lambda \operatorname{div} u I_n.
	\end{equation}
	
	The system described by equations \eqref{1} and \eqref{h1} is commonly referred to as the classical compressible isentropic Navier-Stokes equations. However, for complex fluids, such as macromolecular or polymeric fluids, the Newtonian law of viscosity, as given by equation \eqref{h1}, is no longer applicable. In particular, for viscoelastic fluids, Maxwell \cite{maxwell1867iv} introduced a constitutive equation that combines Newton's law of viscosity with Hooke's law of elasticity:
	\begin{equation}\label{h2}
		\tau \dot{S} + S = \mu \left( \nabla u + \nabla u^T - \frac{2}{n} \operatorname{div} u I_n \right) + \lambda \operatorname{div} u I_n,
	\end{equation}
	where \(\dot{S} = \partial_t S + u \cdot \nabla S\) denotes the material derivative. The positive parameter $\tau$ is the relaxation time describing the time lag in the response of the stress tensor to velocity gradient. A fluid obeying equation \eqref{h2} is called Maxwell flow, see \cite{maxwell1867iv}. Replacing the material derivative $\dot S$ by objective derivative, we obtain an objective constitutive equation of compressible Oldroyd-B  type models:
	\begin{align}\label{h3}
		\tau \mathring {S}+S=\mu (\nabla u+\nabla u^T-\frac{2}{n} \mathrm{div} u I_n)+\lambda \operatorname{div} u I_n
	\end{align}
	where $\mathring S:=\dot S+S W(u)-W(u)S-a(S D(u)+D(u)S)$ with $W(u)=\frac{1}{2}(\nabla u-\nabla u^T)$, $D(u)=\frac{1}{2} (\nabla u+\nabla u^T)$ and $-1\le a \le 1$ be a constant. The parameters $a=1, -1, 0$  corresponds to the upper-convective, lower-convective and corotational Maxwell models, see \cite{maxwell1867iv}.
	We should mention that,  even for a simple fluid, the effect of relaxation can not always be neglected, see \cite{pelton2013viscoelastic} with the experiments of high-frequency (20GHZ) vibration of nanoscale mechanical devices immersed in water-glycerol mixtures.

	Inspired by the work of Freistühler \cite{freistuhler2020galilei}, which is based on Yong's model \cite{yong2014newtonian}, we investigate the following constitutive equations for the stress tensor \( S = S_1 + S_2 I_n \), where \( S_1 \) and \( S_2 \) satisfy:
	\begin{align}
		\tau_1 \rho \left( \partial_t S_1 + u \cdot \nabla S_1 + S_1 W(u) - W(u) S_1 \right) + S_1 &= \mu \left( \nabla u + \nabla u^T - \frac{2}{n} \mathrm{div} u \, I_n \right), \label{h6} \\
		\tau_2 \rho \left( \partial_t S_2 + u \cdot \nabla S_2 \right) + S_2 &= \lambda \mathrm{div} u. \label{h7}
	\end{align}
	Here, \( \tau_1 \) and \( \tau_2 \) represent the shear and bulk relaxation times, respectively. The objective constitutive equations \eqref{h6}-\eqref{h7} are referred to as the revised Maxwell's law. In its linearized form, this model has been demonstrated to outperform various other models in capturing the dynamic behavior of linear viscoelastic flows, as shown in \cite{chakraborty2015constitutive}. Specifically, for compressible viscoelastic fluids, Chakraborty and Sader \cite{chakraborty2015constitutive} have illustrated that this model provides a comprehensive framework for characterizing the fluid-structure interactions of nanoscale mechanical devices vibrating in simple liquids.
	
	Now, we present some known results for system\eqref{1}, \eqref{h6}-\eqref{h7}. By neglecting the quadratic nonlinear terms  and the factor $\rho$  in \eqref{h6}–\eqref{h7}, Yong \cite{yong2014newtonian} established the local well-posedness of strong solutions and the strong relaxation limit for well-prepared data. These results were subsequently extended to the non-isentropic case by Hu and Racke \cite{hu2017compressible} and to the global relaxation limit with general initial data by Peng \cite{peng2021relaxed}. Additionally, Hu and Wang \cite{HWAML} investigated the finite-time blowup of smooth solutions for certain large initial data. It is important to note that all the aforementioned results focus on the Cauchy problem. In contrast, the initial boundary value problem is more intricate due to its characteristic boundary structure. To the best of our knowledge, the only available results in this context are in the one-dimensional setting, where Hu and Li \cite{HL24} and Hu and Zhao \cite{HZ25} derived global strong solutions for the isentropic and non-isentropic systems, respectively.
	
	In this paper, we are concerned with the initial boundary value problem for the system  \eqref{1}, \eqref{h6}, and \eqref{h7}, within a multidimensional domain that exhibits spherical symmetry. Specifically, we consider the exterior domain of a ball in \(\mathbb{R}^n\) with $n=3$. Denote \(\Omega := \{x \in \mathbb{R}^3 \mid |x| > a\}\) with \(a > 0\). Without loss of generality, we set $a = 1$.
	Our objective is to seek spherically symmetric solutions of the form
	\begin{equation}
		\begin{gathered}
			u(t, x) = v(t, r) \frac{x}{r}, \quad \rho(t, x) = \rho(t, r), \\
			S_1(t, x) = [S^{ij}] = \tilde{S}_1(t, r) \left[\frac{x_i x_j}{r^2} - \frac{1}{3} \delta_{ij}\right], \quad S_2(t, x) = \tilde{S}_2(t, r),
		\end{gathered}
	\end{equation}
	where \(x = (x_1, x_2, x_3) \in \Omega\) and \(r := |x|\). By using this symmetry, we effectively reduce the original system (\ref{1}) to the following set of equations:
	\begin{equation} \label{1.4}
		\begin{cases}
			\partial_t \rho + \partial_r (\rho v) + \dfrac{2}{r} \rho v = 0, \\
			\rho \partial_t v + \rho v \partial_r v + \partial_r P = \dfrac{2}{3} \partial_r \tilde{S}_1 + \dfrac{2}{r} \tilde{S}_1 + \partial_r \tilde{S}_2, \\
			\tau_1 \rho \left(\partial_t \tilde{S}_1 + v \partial_r \tilde{S}_1 \right) + \tilde{S}_1 = 2 \mu \left(\partial_r v - \dfrac{v}{r}\right), \\
			\tau_2 \rho \left(\partial_t \tilde{S}_2 + v \partial_r \tilde{S}_2 \right) + \tilde{S}_2 = \lambda \left(\partial_r v + \dfrac{2}{r} v\right).
		\end{cases}
	\end{equation}
	The system \eqref{1.4} is supplemented with the following initial conditions
	\begin{equation}\label{1.5}
		(\rho(t, r), v(t, r), \tilde{S}_1(t, r), \tilde{S}_2(t, r))|_{t=0} = (\rho_0(r), v_0(r), \tilde{S}_{10}(r), \tilde{S}_{20}(r)), \quad r \in [1, \infty),
	\end{equation}
	and the boundary condition
	\begin{equation}\label{1.6}
		v(t, r)|_{r=1} = 0, \quad t > 0.
	\end{equation} 
	
	Note that when $\tau_1=\tau_2=0$, system (\ref{1.4}) reduces to classical isentropic Navier-Stokes equations with spherical symmetry, for which there are lots of studies on well-posedness theory and large time behavior of both strong and weak solutions,  see \cite{jiang1996global, ding2012global, Cho-Kim}. 
	In particular, for large data that are away from vacuum, Jiang \cite{jiang1996global} proved the global existence and large time behavior of spherically symmetric smooth solutions  in the exterior of a ball in $\mathbb{R}^n$ (\(n = 2\) or 3). When vacuum is taken into account, local and global smooth solutions have also been obtained under certain compatibility conditions, as demonstrated in \cite{ding2012global, Cho-Kim}.  For results concerning weak solutions, we refer to \cite{feireisl2004dynamics,hoff1992spherically,jiang2001spherically,lions1998mathematical,makino1994initial} and references cited therein.
	
	For \(\tau_1 > 0\) and \(\tau_2 > 0\), the fundamental change in the system's structure—from a hyperbolic-parabolic coupled system to a purely hyperbolic system—renders the global existence of classical solutions with large initial data highly unlikely. For instance, Hu and Wang \cite{HWAML} demonstrated that smooth solutions to the relaxed isentropic compressible Navier-Stokes system must blow up in finite time for certain large initial data. Moreover, it can be readily demonstrated that the hyperbolic system \eqref{1.4}–\eqref{1.6} features a characteristic boundary, which introduces significant complexity to the well-posedness theory, even in the local context, see \cite{Chen2007,schochet1986compressible}. 
	
	Our objective is to establish the global well-posedness of strong solutions to system \eqref{1.4}–\eqref{1.6} with small initial data. The strategy we employ is as follows: First, we construct an approximate system with a non-characteristic boundary. By verifying that the boundary is maximally nonnegative, we establish the local well-posedness of strong solutions for this approximated system. Next, we derive uniform a priori estimates and demonstrate the existence of a unique global solution for this approximated system. In this part, the estimates for second-order derivative with respect to $x$ play crucial role. To derive these estimates, we leveraged the structure of the system to construct multiple multipliers. Additionally, boundary terms emerging from integration by parts posed another  challenge. By carefully exploiting the special structure of the system and various multiplier energy methods, the boundary term can be controlled.  Finally, by taking the limit and applying standard compactness arguments, we obtain the desired results.
	
	Now, we introduce some notations. \( W^{m,p} = W^{m,p}(\Omega),\ 0 \leq m \leq \infty,\ 1 \leq p \leq \infty \), denotes the usual Sobolev space with norm \( \|\cdot\|_{W^{m,p}} \). \( H^m \) and \( L^p \) stand for \( W^{m,2} \) resp. \( W^{0,p} \). \( D^\alpha = \partial_t^{\alpha_1} \partial_r^{\alpha_2},\ \alpha = (\alpha_1, \alpha_2),\ |\alpha| = \alpha_1 + \alpha_2 \).

	The following assumptions are needed throughout the paper:
	\begin{assumption}\label{Assump1}
		\begin{enumerate}
			\item[(1)] \(\tau_1\) is the same order as \(\tau_2\), i.e., \(\tau_1 = O(\tau_2)\) as \(\tau_2 \to 0\). Without loss of generality, let \(\tau:= \tau_1 = \tau_2\).
			
			\item[(2)] The initial and boundary data satisfy compatibility condition
			\begin{equation}\label{compatibility}
				\partial_t^k v(0, r)\big|_{r=1} = 0, \quad k = 0, 1. 
			\end{equation}
			\item[(3)] The initial data is well-prepared in the following sense:
			\[
			\|r(\tilde{S}_1(0) - 2\mu(\partial_{r}v_0-\frac{v_0}{r}))\|_{H^1}
			= O(\sqrt{\tau}), 
			\quad \| r(\tilde{S}_2(0) - \lambda(\partial_{r}v_0+\frac{2v_0}{r}))\|_{H^1}
			= O(\sqrt{\tau}), \quad \text{as } \tau \to 0. 
			\]
			
			\item[(4)] \(rV_0^k \in H^{2-k}\) for \(k = 0, 1, 2\) uniformly with respect to $\tau$, where
			\[
			V_0^k:=  \partial_t^k \left(\rho - 1, v, \sqrt{\tau} \tilde{S}_1, \sqrt{\tau} \tilde{S}_2\right) (t = 0, \cdot), \quad k = 0, 1
			\]
			\[
			V_0^2:= \tau \partial_t^2 \left(\rho - 1, v, \sqrt{\tau} \tilde{S}_1, \sqrt{\tau} \tilde{S}_2\right) (t = 0, \cdot)
			\]
			%where \(V_0^1\), \(V_0^2\) are defined recursively by equations (1.10).
		\end{enumerate}  
	\end{assumption}

	The main results of this paper are as follows.
	\begin{theorem}\label{thm1}
		Let Assumption \eqref{Assump1} hold. Then, there exists a small constant $\epsilon_0>0$ such that if
		\begin{equation}
			E_0:= \sum_{k=0}^2\|rV_0^k\|_{H^{2-k}}^2<\epsilon_0,
		\end{equation}
		there exists a globally defined solution $(\rho,v,\tilde{S}_1,\tilde{S}_2)(r,t)$ to system \eqref{1.4}-\eqref{1.6} satisfying
		$$(\rho-1,v,\tilde{S}_1,\tilde{S}_2)(t,r)\in C([0,\infty),H^{2-\delta_0}(\Omega))\cap C^1([0,\infty),H^{1-\delta_0}(\Omega)) $$
		for any $\delta_0>0$ and 
		\begin{equation}\label{1.8}
			\begin{aligned}
				&\sup _{0 \leq t < \infty}\bigg( \sum_{k=0}^1\left\|r\partial_{t}^k(\rho-1, v, \sqrt{\tau} \tilde{S}_1, \sqrt{\tau} \tilde{S}_2)\right\|_{H^{2-k}}^2
				+\tau^2\left\|r\partial_{tt}(\rho-1, v, \sqrt{\tau} \tilde{S}_1, \sqrt{\tau} \tilde{S}_2)\right\|_{L^{2}}^2\bigg)\\
				&+\int_{0}^{\infty}\bigg(\sum_{\alpha=1}^2\left\|r D^\alpha(\rho, v)\right\|_{L^2}^2
				+\sum_{k=0}^1\left\|r\partial_{t}^k(\tilde{S}_1, \tilde{S}_2)\right\|_{H^{2-k}}^2
				+\tau^2\left\|r(\partial_{tt}\tilde{S}_1,\partial_{tt}\tilde{S}_2)\right\|_{L^{2}}^2\bigg)\mathrm{~d} t
				\leq C E_0
			\end{aligned}
		\end{equation}
		where C is a universal constant independent of $\tau$.
	\end{theorem}
	
	\begin{remark} 
		Observe that the norm 
		\(\|V_0^k\|_{H^{2-k}} = O(1)\) as \(\tau \to 0\), which follows from the assumption of well-prepared initial data. To verify this, we derive the following
		\[
		\left\|r\sqrt{\tau}\partial_{t}\tilde{S}_1(t=0,\cdot)\right\|_{H^1} 
		= \left\|r\left(\frac{2\mu(\partial_{r}v-\frac{v}{r})-\tilde{S}_1}{\sqrt{\tau}}-\sqrt{\tau}v\partial_{r}\tilde{S}_1\right)(t=0,\cdot)\right\|_{H^1} 
		= O(1), \quad \text{as } \tau \to 0,
		\]
		and
		\[
		\left\|r\tau^{\frac{3}{2}}\partial_{tt}\tilde{S}_1(t=0,\cdot)\right\|_{L^2} 
		= \left\|r\left(\sqrt{\tau}2\mu(\partial_{r}v-\frac{v}{r})-\sqrt{\tau}\tilde{S}_1-\tau^{\frac{3}{2}}v\partial_{r}\tilde{S}_1\right)_t(t=0,\cdot)\right\|_{L^2} 
		=O(1), \quad \text{as } \tau \to 0.
		\]
		The other terms can be handled similarly.
	\end{remark}
	
	Based on the uniform estimates of solutions, we have the following convergence theorem.
	\begin{theorem}\label{thm2}
		(Global Weak Convergence).
		Let $(\rho^{\tau},v^{\tau},\tilde{S}_1^{\tau},\tilde{S}_2^{\tau})$ denote the family of global solutions constructed in Theorem \ref{thm1}. Then there exist limit functions
		$$
		(\rho^{0},v^{0}) \in L^{\infty}\left(\mathbb{R}^{+} ; H^2(\Omega)\right), \quad(\tilde{S}_1^{0},\tilde{S}_2^{0}) \in L^2\left(\mathbb{R}^{+} ; H^2(\Omega)\right),
		$$
		such that along a subsequence as $\tau\rightarrow 0$:
		\begin{equation}
			\begin{aligned}
				& \left(\rho^{\tau},v^{\tau}\right) \rightharpoonup^*\left(\rho^{0},v^{0}\right) \quad \text { weakly-* in } L^{\infty}\left(\mathbb{R}^{+} ; H^2(\Omega)\right), \\
				& \left(\tilde{S}_1^{\tau},\tilde{S}_2^{\tau}\right) \rightharpoonup\left(\tilde{S}_1^{0},\tilde{S}_2^{0}\right) \quad \text { weakly in } L^2\left(\mathbb{R}^{+} ; H^2(\Omega)\right) .
			\end{aligned}
		\end{equation}
		The limit $\left(\rho^{0},v^{0}\right)$ constitutes a classical solution to the  spherical symmetric  isentropic compressible Navier-Stokes system  with boundary condition (\ref{1.6}) and initial data $\left(\rho^{0},v^{0}\right)(t=0,x)=(\rho_0,v_0)(x)$. Moreover,
		$$\tilde{S}_1^0=2\mu(\partial_{r}v^0-\frac{v^0}{r}), \quad \tilde{S}_2^0=\lambda(\partial_{r}v^0+\frac{2v^0}{r}).$$
	\end{theorem}

	The paper is organized as follows. In Section 2, we construct an approximate system which possesses a non-characteristic boundary and prove its local well-posedness. In Section 3, we establish  uniform a priori estimates and  obtain a global solution for this approximated system. Based on the uniform estimates derived in Section 3 and usual compactness argument, we justify the limit and prove our main theorem in the final section.
	
	\section{Approximated system and local  existence}\label{sec2}
	In this section, we construct an approximate system with a non-characteristic boundary and establish the local existence of a unique smooth solution to the system. Following the approach of Schochet \cite{schochet1986compressible}, as adopted in \cite{HL24, HZ25}, we approximate system \eqref{1.4} as follows:
	\begin{equation} \label{9}
		\begin{cases}
			\partial_t \rho^{\epsilon}+\partial_r(\rho^{\epsilon} v^{\epsilon})+\dfrac{2}{r} \rho^{\epsilon} v^{\epsilon}=0, \\
			\rho^{\epsilon}\partial_t v^{\epsilon}+\rho^{\epsilon} v^{\epsilon} \partial_r v^{\epsilon}+\partial_r P=\dfrac{2}{3}\partial_r \tilde{S}_1^{\epsilon} +\dfrac{2}{r} \tilde{S}_1^{\epsilon}+\partial_r \tilde{S}_2^{\epsilon},\\
			\tau   \rho^{\epsilon} \left(\partial_t \tilde{S}_1^{\epsilon}+(v^{\epsilon}-\epsilon)\partial_r\tilde{S}_1^{\epsilon} \right) +\tilde{S}_1^{\epsilon}=2\mu(\partial_r v^{\epsilon}- \dfrac{v^{\epsilon}}{r})
			, \\
			\tau  \rho^{\epsilon} \left(\partial_t \tilde{S}_2^{\epsilon}+(v^{\epsilon}-\epsilon)\partial_r\tilde{S}_2^{\epsilon} \right) +\tilde{S}_2^{\epsilon}=\lambda (\partial_r v^{\epsilon}+\dfrac{2}{r}v^{\epsilon}),
		\end{cases}
	\end{equation}
	with initial condition
	\begin{equation}\label{10}
		(\rho^\epsilon,v^\epsilon,\tilde{S}_1^{\epsilon},\tilde{S}_2^{\epsilon})(0,r)=(\rho_0,v_0,\tilde{S}_{10},\tilde{S}_{20})(r),
	\end{equation}
	and boundary condition
	\begin{equation}\label{11}
		v^\epsilon(t,1)=0.
	\end{equation}
	
	We transform the above system into a first-order system for $V^\epsilon:=(\rho^\epsilon,v^\epsilon,\tilde{S}_1^{\epsilon},\tilde{S}_2^{\epsilon})^{T}$,
	\begin{equation}
		A^0(V^\epsilon)\partial_{t}V^\epsilon+A^1(V^\epsilon)\partial_{r}V^\epsilon+B(V^\epsilon)V^\epsilon=0,
	\end{equation} 
	where\\
	$A^0(V^\epsilon)=\left(\begin{array}{ccccc}
		\dfrac{P^{\prime}(\rho^\epsilon)}{\rho^\epsilon} & 0 & 0 & 0 \\
		0 & \rho^\epsilon & 0 & 0 \\
		0 & 0 & \dfrac{\tau  \rho^\epsilon}{3\mu} & 0 \\
		0 & 0 & 0 & \dfrac{\tau \rho^\epsilon}{\lambda}  
	\end{array}\right)$,
	$
	A^1(V^\epsilon)=\left(\begin{array}{ccccc}
		\dfrac{P^{\prime}(\rho^\epsilon)v}{\rho} & P^{\prime}(\rho^\epsilon) & 0 & 0 \\
		P^{\prime}(\rho^\epsilon) & \rho^\epsilon v^\epsilon & -\dfrac{2}{3} & -1 \\
		0 & -\dfrac{2}{3}  & \dfrac{\tau   \rho^\epsilon(v^\epsilon-\epsilon) }{3\mu} & 0 \\
		0 & -1 & 0 &  \dfrac{\tau \rho^\epsilon(v^\epsilon-\epsilon)}{\lambda}
	\end{array}\right)$,\\
	$B(V^\epsilon)=\left(\begin{array}{ccccc}
		0 & \dfrac{2 P^{\prime}(\rho^\epsilon)}{r} & 0 & 0  \\
		0 & 0 & -\dfrac{2}{r} & 0  \\
		0 & \dfrac{2}{3r}  & \dfrac{1}{3\mu} & 0  \\
		0 & -\dfrac{2}{r} & 0 & \dfrac{1}{\lambda}
	\end{array}\right)$,\\
	with initial condition $V^\epsilon(0,r)=V_0:=(\rho_0,v_0,\tilde{S}_{10},\tilde{S}_{20})(r)$ 
	and boundary condition
	$$MV^\epsilon\big|_{\partial\Omega}=0,~~\text{with} 
	~~M=\left(\begin{array}{ccccc}
		0 & 0 & 0 & 0 \\
		0 & 1 & 0 & 0 \\
		0 & 0 & 0 & 0 \\
		0 & 0 & 0 & 0  
	\end{array}\right).$$\\
	
	Note that $det(A^0)^{-1}A^1(V^\epsilon)\big|_{\partial\Omega}=-P^{\prime}(\rho^\epsilon)\epsilon^2\neq0$
	with $P^{\prime}(\rho^\epsilon)>0$ for any $V^\epsilon\in G:=\{(0,\infty)\times\mathbb{R}^3\}$. So, the boundary condition \eqref{11} is a non-characteristic boundary for the approximated system \eqref{9}. 
	
	Next, we need to prove that \eqref{11} satisfies maximally non-negative property. Namely, the matrix $\left.A^1\left(V^\epsilon\right) \cdot \nu\right|_{\partial \Omega}$ is positive semi-definite on the null space $N$ of $M$, yet not on any space encompassing $N$, where $\nu$ denotes the unit outward normal vector on the boundary. Here $\nu=-1$. Let $\xi=\left(\xi_1, 0, \xi_2,\xi_3\right)^T \in \operatorname{ker} M=\operatorname{span}\left\{(1,0,0,0)^T,(0,0,1,0)^T,(0,0,0,1)^T\right\}$, then 
	$$\xi^T A^1 \cdot \nu\big|_{\partial \Omega} \xi=
	\frac{\tau   \rho^{\epsilon}\epsilon}{3\mu}\xi_2^2+\frac{\tau  \rho^{\epsilon}\epsilon}{\lambda}\xi_3^2\geq0.$$
	On the other hand, the only space having $\operatorname{ker} M$ as a proper subspace is $\mathbb{R}^4$. Thus, we take $q=(1,1,0,0)^T\in \mathbb{R}^4$, and  calculate
	$$q^T A^1 \cdot \nu\big|_{\partial \Omega} q=-2P^{\prime}(\rho^{\epsilon})<0.$$
	Thus, the maximally nonnegative property is satisfied. Therefore, classical results implies local well-posedness theory\cite{schochet1986compressible}.
	
	\begin{theorem}
		Suppose $(\rho_0,v_0,\tilde{S}_{10},\tilde{S}_{20})(r)\in H^2$ satisfying the compatibility condition (\ref{compatibility}) and
		\begin{equation*}
			\min\limits_{r \in[1,\infty)} \rho_0(r)>0.
		\end{equation*}
		Then there exists a unique local solution $(\rho^\epsilon,v^\epsilon,\tilde{S}_1^{\epsilon},\tilde{S}_2^{\epsilon})$ to initial boundary value problem (\ref{9})-(\ref{11})
		on some time interval $[0, T]$ with
		$$
		\begin{array}{r}
			\left(\rho^\epsilon,v^\epsilon,\tilde{S}_1^{\epsilon},\tilde{S}_2^{\epsilon}\right) \in C^0\left([0, T], H^2\right) \cap C^1\left([0, T], H^1\right), \\
			\min \limits_{r \in[1,\infty)} \rho(t,r)>0, \quad \forall \quad t>0 . 
		\end{array}
		$$
	\end{theorem}
	\section{Uniform a priori estimates for the approximated system}\label{sec3}
	
	In this section, we derive uniform a priori estimates with respect to \(\tau\) and \(\epsilon\) for the approximated system \eqref{9}-\eqref{11}, yielding uniform global solutions by usual continuation arguments.
	
	For simplicity, we still denote $(\rho^{\epsilon}, v^{\epsilon}, \tilde{S}_1^{\epsilon},\tilde{S}_2^{\epsilon})$ by $(\rho, v,\tilde{S}_1,\tilde{S}_2)$. 
	First, we define the energy 
	
	\begin{equation}
		E(t):=\sup _{0 \leq s \leq t}\bigg( \sum_{k=0}^1\left\|r\partial_{t}^k(\rho-1, v, \sqrt{\tau  } \tilde{S}_1, \sqrt{\tau } \tilde{S}_2)(s,\cdot)\right\|_{H^{2-k}}^2
		+\tau^2\left\|r\partial_{t}^2(\rho-1, v, \sqrt{\tau  } \tilde{S}_1, \sqrt{\tau } \tilde{S}_2)(s,\cdot)\right\|_{L^{2}}^2\bigg)
	\end{equation}
	and  dissipation
	\begin{equation}
		\mathcal{D}(t):=\sum_{|\alpha|=1}^2\left\|r D^\alpha(\rho, v)\right\|_{L^2}^2
		+\sum_{k=0}^1\left\|r\partial_{t}^k(\tilde{S}_1, \tilde{S}_2)\right\|_{H^{2-k}}^2
		+\tau^2\left\|r(\partial_{tt}\tilde{S}_1,\partial_{tt}\tilde{S}_2)\right\|_{L^{2}}^2
	\end{equation}
	where $D^\alpha=\partial_{t}^{\alpha_1}\partial_{r}^{\alpha_2}$ with $|\alpha|=\alpha_1+\alpha_2$. We are aiming to prove the following proposition.
	\begin{proposition}\label{prop1}
		Let Assumption \ref{Assump1} hold and $(\rho, v, \tilde{S}_1, \tilde{S}_2)\in C^0([0,T],H^2)$ be local solutions to system (\ref{9}). Assume that there exists a small $\delta$ such that $E(t)\leq \delta$ for any $0\le t \le T$. Then, for any $0\le t \le T$, we have
		\begin{equation}
			E(t)+\int_0^t \mathcal{D}(s) \mathrm{d}s\leq C \left(E(0)+E^{\frac{3}{2}}(t)+E^{\frac{1}{2}}(t) \int_0^t \mathcal{D}(s) \mathrm{d}s\right)
		\end{equation}
		where $C$ is a constant independent of $\tau$ and $\epsilon$.
	\end{proposition}
	
	In the following lemmas, we suppose \(\delta\) is small enough such that
	\begin{equation}\label{hu3.4}
		\frac{3}{4} \leq \sup_{(t,r)\in[0,T]\times[1,+\infty)} \rho(t,r) \leq \frac{5}{4}.
	\end{equation}
	
	Furthermore, without loss of generality, we assume $\tau \leq 1$ and $\epsilon \ll 1$. $C$ denotes a universal constant which is independent of \(\tau\) and \(\epsilon\).
	
	\subsection{Lower energy estimates}
	With simplified notation, we have equations
	\begin{equation} \label{3.4}
		\begin{cases}
			\partial_t \rho+\partial_r(\rho v)+\dfrac{2}{r} \rho v=0, \\
			\rho\partial_t v+\rho v \partial_r v+\partial_r P=\dfrac{2}{3}\partial_r \tilde{S}_1 +\dfrac{2}{r} \tilde{S}_1+\partial_r \tilde{S}_2,\\
			\tau  \rho \left(\partial_t \tilde{S}_1+(v-\epsilon)\partial_r\tilde{S}_1 \right) +\tilde{S}_1=2\mu\left(\partial_r v- \dfrac{v}{r}\right)
			, \\
			\tau  \rho \left(\partial_t \tilde{S}_2+(v-\epsilon)\partial_r\tilde{S}_2 \right) +\tilde{S}_2=\lambda (\partial_r v+\dfrac{2}{r}v)
			.
		\end{cases}
	\end{equation}
	\begin{lemma}\label{lem1}
		There exists some constant $C$ such that for any $0\le t \le T$
		\begin{align}
			\|r(\rho-1, v, \sqrt{\tau  } \tilde{S}_1, \sqrt{\tau } \tilde{S}_2)\|_{L^2}^2+\int_0^t\|r(\tilde{S}_1,\tilde{S}_2)\|_{L^2}^2 \mathrm{d}t 
			+\int_{0}^{t}\frac{\tau \epsilon}{8 \mu} \tilde{S}_1^2(t,1)\mathrm{d}t \nonumber\\
			\leq 
			C \left(E(0)+E^{\frac{1}{2}}(t) \int_0^t \mathcal{D}(s) \mathrm{d}s\right).
		\end{align}
	\end{lemma}
	
	\begin{proof}
		Multiplying the equation $(\ref{3.4})_2$ by $r^2 v$, we have
		\begin{equation}\nonumber
			(\frac{1}{2}r^2 \rho v^2)_t+\frac{P}{\rho}(r^2 \partial_t\rho+r^2 v \partial_r\rho)=\frac{2}{3}r^2v\partial_{r}\tilde{S}_1+2rv\tilde{S}_1+r^2v\partial_{r}\tilde{S}_2.
		\end{equation}
		Integrating the above equation and using the mass equation, we have
		\begin{equation}\label{0-1}
			\frac{\mathrm{d}}{\mathrm{d} t} \int_{1}^{+\infty}\left(
			\frac{r^2}{\gamma-1}(\rho^{\gamma}-1-\gamma(\rho-1))+\frac{r^2 \rho}{2} v^2\right)\mathrm{d} r=
			\int_{1}^{+\infty}(\frac{2}{3}r^2v\partial_{r}\tilde{S}_1+2rv\tilde{S}_1
			+r^2v\partial_{r}\tilde{S}_2)\mathrm{d} r.
		\end{equation}
		
		Multiplying the equation $(\ref{3.4})_3$ by $\frac{r^2}{3 \mu}\tilde{S}_1$, and multiply the equation $(\ref{3.4})_4$ by $\frac{r^2}{\lambda} \tilde{S}_2$, then one can get
		\begin{equation}
			\begin{aligned}\label{0-2}
				&\frac{\mathrm{d}}{\mathrm{d} t} \int_{1}^{+\infty}\left(\frac{\tau   r^2 \rho}{6 \mu}   \tilde{S}_1^2+\frac{\tau  r^2 \rho}{2\lambda}  \tilde{S}_2^2\right)+\int_{1}^{+\infty} (\frac{r^2}{3 \mu}  \tilde{S}_1^2+\frac{r^2}{\lambda} \tilde{S}_2^2) \mathrm{d} r\\
				&-\epsilon\int_{1}^{+\infty}\frac{\tau   r^2}{3 \mu}  \rho \left(\frac{1}{2}\tilde{S}_1^2\right)_r\mathrm{d} r -\epsilon\int_{1}^{+\infty}\frac{\tau  r^2}{\lambda}  \rho \left(\frac{1}{2}\tilde{S}_2^2\right)_r\mathrm{d} r\\
				&=\int_{1}^{+\infty}(\frac{2}{3}r^2\partial_{r}v\tilde{S}_1-\frac{2r}{3}v\tilde{S}_1+r^2\partial_{r}v\tilde{S}_2+2rv\tilde{S}_2)\mathrm{d} r.
			\end{aligned}
		\end{equation}
		Combining the equations $(\ref{0-1})$ and $(\ref{0-2})$ and using boundary condition (\ref{1.6}), we can get the energy equality
		\begin{equation}
			\begin{aligned}
				&\frac{\mathrm{d}}{\mathrm{d} t} \int_{1}^{+\infty}\left[\frac{r^2}{\gamma-1}(\rho^{\gamma}-1-\gamma(\rho-1))+\frac{r^2 \rho}{2} v^2+\frac{\tau   r^2 \rho}{6 \mu}   \tilde{S}_1^2+\frac{\tau  r^2 \rho}{2\lambda}  \tilde{S}_2^2\right] \mathrm{d} r\\
				&-\epsilon\int_{1}^{+\infty}\frac{\tau   r^2}{3 \mu}  \rho \left(\frac{1}{2}\tilde{S}_1^2\right)_r\mathrm{d} r
				~~~~ -\epsilon\int_{1}^{+\infty}\frac{\tau  r^2}{\lambda}  \rho \left(\frac{1}{2}\tilde{S}_2^2\right)_r\mathrm{d} r +\int_{1}^{+\infty} (\frac{r^2}{3 \mu}  \tilde{S}_1^2+\frac{r^2}{\lambda} 
				\tilde{S}_2^2) \mathrm{d} r=0.
			\end{aligned}
		\end{equation}
		Besides, we have 
		\begin{equation}\nonumber
			\begin{aligned}
				-\frac{\tau \epsilon}{3 \mu} \int_{1}^{+\infty} r^2\rho \left(\frac{1}{2}\tilde{S}_1^2\right)_r\mathrm{d} r
				&=-\frac{\tau   \epsilon}{3 \mu}r^2\rho \frac{1}{2}\tilde{S}_1^2 \bigg|_1^{+\infty} +\frac{\tau \epsilon}{3 \mu} \int_{1}^{+\infty}(2r\rho+r^2\partial_{r}\rho)\frac{1}{2}\tilde{S}_1^2\mathrm{d} r\\
				&= \frac{\tau \epsilon}{6 \mu}\rho(t,1)\tilde{S}_1^2(t,1)
				+\int_{1}^{+\infty} \frac{\tau \epsilon r}{3 \mu} \rho\tilde{S}_1^2
				-C E^{\frac{1}{2}}(t) \mathcal{D}(t)\\
				&\geq \frac{\tau \epsilon}{6 \mu}\rho(t,1)\tilde{S}_1^2(t,1)
				-C E^{\frac{1}{2}}(t) \mathcal{D}(t)
			\end{aligned}
		\end{equation}
		and
		\begin{equation}\nonumber
			\begin{aligned}
				-\frac{\tau  \epsilon}{\lambda} \int_{1}^{+\infty} r^2\rho \left(\frac{1}{2}\tilde{S}_2^2\right)_r\mathrm{d} r
				&=-\frac{\tau  \epsilon}{\lambda} r^2\rho \frac{1}{2}\tilde{S}_2^2 \bigg|_1^{+\infty}
				+\frac{\tau \epsilon}{\lambda} \int_{1}^{+\infty}(2r\rho+r^2\partial_r\rho)\frac{1}{2}\tilde{S}_2^2\mathrm{d} r\\
				&= \frac{\tau \epsilon}{2\lambda}\rho(t,1)\tilde{S}_2^2(t,1) +\int_{1}^{+\infty} \frac{\tau \epsilon r}{\lambda} \rho \tilde{S}_2^2
				-C E^{\frac{1}{2}}(t) \mathcal{D}(t)\\
				&\geq-C E^{\frac{1}{2}}(t) \mathcal{D}(t)
				.
			\end{aligned}
		\end{equation}

		Combining the above result and \eqref{hu3.4}, we obtain
		\begin{equation}
			\begin{aligned}
				\frac{\mathrm{d}}{\mathrm{d} t} \int_{1}^{+\infty}\left[ \frac{r^2}{\gamma-1}\left(\rho^\gamma-1-\gamma(\rho-1)\right) 
				+\frac{r^2 \rho}{2} v^2+\frac{\tau   r^2 \rho}{6 \mu}  \tilde{S}_1^2+\frac{\tau  r^2 \rho}{2\lambda} \tilde{S}_2^2\right] \mathrm{d} r\\
				+\int_{1}^{+\infty} (\frac{r^2} {3\mu} \tilde{S}_1^2+\frac{r^2}{\lambda} \tilde{S}_2^2) \mathrm{d} r
				+\frac{\tau \epsilon}{8 \mu}\tilde{S}_1^2(t,1)\leq
				C E^{\frac{1}{2}}(t) \mathcal{D}(t).
			\end{aligned}
		\end{equation}
		Using Taylor's expansions, we know
		\begin{equation}
			\rho^\gamma-1-\gamma(\rho-1)=\frac{\gamma(\gamma-1)}{2}\xi^{\gamma-2}(\rho-1)^2,
		\end{equation}
		where $\xi\in(1,\rho)$.
		Using \eqref{hu3.4}, we get,
		\begin{equation}
			\|r(\rho-1, v, \sqrt{\tau  } \tilde{S}_1, \sqrt{\tau } \tilde{S}_2)\|_{L^2}^2+\int_0^t\|r(\tilde{S}_1,\tilde{S}_2)\|_{L^2}^2 \mathrm{d}t 
			+\int_{0}^{t}\frac{\tau \epsilon}{8 \mu}\tilde{S}_1^2(t,1)\mathrm{d}t\leq 
			C \left(E(0)+E^{\frac{1}{2}} \int_0^t \mathcal{D}(s) \mathrm{d}s\right).
		\end{equation}
		So, the proof of Lemma \ref{lem1} is finished.
	\end{proof}
	\subsection{First-order estimates}
	
	\begin{lemma}\label{lem2}
		There exists some constant $C$ such that for any $0\le t \le T$
		\begin{equation}\label{hu3.13}
			\begin{aligned}
				&\int_{1}^{+\infty}  r^2\left((\partial_{t}\rho)^2+(\partial_{t}v)^2
				+\tau  (\partial_{t}\tilde{S}_1)^2+\tau (\partial_{t}\tilde{S}_2)^2\right)\mathrm{d}r
				+ \int_0^t \int_{1}^{+\infty}r^2\left(  (\partial_{t}\tilde{S}_1)^2+ (\partial_{t}\tilde{S}_2)^2\right)\mathrm{d}r\mathrm{d}t\\
				&+\int_{0}^{t}\frac{\tau \epsilon}{8 \mu}(\partial_t\tilde{S}_1)^2(t,1)\mathrm{d}t
				\leq C\left(E(0)+ E^{\frac{1}{2}}(t) \int_{0}^{t}\mathcal{D}(s)\mathrm{~d} t\right).	
			\end{aligned}
		\end{equation}
	\end{lemma}
	\begin{proof}
		Taking derivative with respect to $t$ to the equations (\ref{3.4}), one obtain
		
		\begin{equation}
			\label{15}
			\begin{cases}
				\partial_{tt} \rho + \partial_{tr}(\rho v) + \dfrac{2}{r} \partial_t(\rho v) = 0, \\[6pt]
				
				\rho\partial_{tt} v + \partial_t \rho \cdot \partial_t v + \partial_t \rho \cdot v \partial_r v 
				+ \rho \cdot \partial_t( v \partial_r v) + \partial_{tr} P  = \dfrac{2}{3}\partial_{tr} \tilde{S}_1 + \dfrac{2}{r} \partial_t \tilde{S}_1 + \partial_{tr} \tilde{S}_2, \\[6pt]
				
				\tau \rho \left(\partial_{tt} \tilde{S}_1 + \partial_t ((v-\epsilon) \partial_r\tilde{S}_1) \right) 
				+ \tau \partial_t \rho \left(\partial_t \tilde{S}_1 + (v-\epsilon)\partial_r\tilde{S}_1 \right) 
				+ \partial_t \tilde{S}_1  = 2\mu\left(\partial_{tr} v - \dfrac{1}{r}\partial_t v \right), \\[6pt]
				
				\tau \rho \left(\partial_{tt} \tilde{S}_2 + \partial_t ((v-\epsilon)\partial_r\tilde{S}_2) \right) 
				+ \tau \partial_t \rho \left(\partial_t \tilde{S}_2 + (v-\epsilon)\partial_r\tilde{S}_2 \right) 
				+ \partial_t \tilde{S}_2  = \lambda \left(\partial_{tr} v + \dfrac{2}{r}\partial_t v\right).
			\end{cases}
		\end{equation}
	
		Multiplying the equations $(\ref{15})_{1, 2}$ by $r^2 \frac{P^{\prime}(\rho)}{\rho} \partial_t \rho$ and $r^2 \partial_t v$, respectively, 
		and integrating the result,  we get
		\begin{equation}\label{3.17}
			\begin{aligned}
				&\frac{\mathrm{d}}{\mathrm{d} t} \int_{1}^{+\infty}\left(\frac{r^2P^{\prime}(\rho)}{2\rho} \left(\partial_t \rho\right)^2
				+ \frac{r^2\rho}{2}(\partial_tv)^2\right) \mathrm{d} r\\
				&+\int_{1}^{+\infty}
				\left(r^2 \partial_{tr}(\rho v) \frac{\partial_t P} {\rho}
				+2 r \partial_t(\rho v) \frac{\partial_t P}{\rho}
				+r^2 \partial_{tr}P \partial_t v\right)\mathrm{d} r\\
				&
				\leq\int_{1}^{+\infty}\left(\frac{2r^2}{3} \partial_{tr} \tilde{S}_1 \partial_{t}v + 2r \partial_t \tilde{S}_1  \partial_{t}v+r^2\partial_{tr} \tilde{S}_2 \partial_{t}v\right)\mathrm{d} r
				+C E(t)^{\frac{1}{2}}\mathcal{D}(t),	
			\end{aligned}
		\end{equation}
		where
		\begin{equation}\nonumber
			\begin{aligned}
				&\int_{1}^{+\infty}(r^2  \partial_{tr}(\rho v)  \frac{\partial_t P}{\rho}
				+2 r \partial_t(\rho v)  \frac{\partial_t P}{\rho}
				+r^2 \partial_{tr}P \partial_t v)\mathrm{d} r\\
				=&\int_{1}^{+\infty}(r^2 \partial_{tr}v \partial_t P+2r \partial_{t}v\partial_tP+r^2  \partial_{tr}P  \partial_t v)\mathrm{d} r\\
				&+\int_{1}^{+\infty}(r^2  (v\partial_{tr}\rho+\partial_{r}\rho\partial_{t}v+\partial_{t}\rho\partial_{r}v )  \frac{\partial_t P}{\rho}+2 r v \partial_t\rho  \frac{\partial_t P}{\rho})\mathrm{d} r\\
				\geq&\int_{1}^{+\infty}\partial_r(r^2  \partial_{t}v  \partial_t P)\mathrm{d} r-C E(t)^{\frac{1}{2}}\mathcal{D}(t).
			\end{aligned}
		\end{equation}
		Thus, (\ref{3.17}) can be simplified to
		\begin{equation}\label{3.18}
			\begin{aligned}
				\frac{\mathrm{d}}{\mathrm{d} t}& \int_{1}^{+\infty}\left(\frac{r^2P^{\prime}(\rho)}{2\rho} \left(\partial_t \rho\right)^2
				+ \frac{r^2\rho}{2}(\partial_tv)^2\right) \mathrm{d} r
				+\int_{1}^{+\infty}
				\partial_r(r^2  \partial_{t}v  \partial_t P)\mathrm{d} r\\
				&
				\leq\int_{1}^{+\infty}\left(\frac{2r^2}{3} \partial_{tr} \tilde{S}_1 \partial_{t}v + 2r \partial_t \tilde{S}_1  \partial_{t}v+r^2\partial_{tr} \tilde{S}_2 \partial_{t}v\right)\mathrm{d} r
				+C E(t)^{\frac{1}{2}}\mathcal{D}(t).
			\end{aligned}
		\end{equation}
		
		Multiplying the equation $(\ref{15})_3$ by $\dfrac{r^2}{3 \mu} \partial_t \tilde{S}_1$ and integrating over $[1,+\infty)$ with respect to $r$, we have
		\begin{equation}\label{18}
			\begin{aligned}
				\frac{\mathrm{d}}{\mathrm{d} t}& \int_{1}^{+\infty}\left[\frac{\tau}{3 \mu} \cdot \frac{r^2 \rho}{2}\left(\partial_t \tilde{S}_1\right)^2\right]\mathrm{d} r +\int_{1}^{+\infty}(\frac{r^2 \tau}{3 \mu} \partial_t (\rho v) \partial_t \tilde{S}_1 \partial_r \tilde{S}_1
				-\frac{\tau \epsilon r^2}{3 \mu} \partial_{t}\rho\partial_t\tilde{S}_1\partial_r \tilde{S}_1)\mathrm{d} r\\
				&-\int_{1}^{+\infty}\frac{\tau  \epsilon r^2}{3 \mu}\rho(\frac{1}{2}\partial_t\tilde{S}_1^2)_r\mathrm{d} r
				+\int_{1}^{+\infty}\frac{r^2}{3 \mu}(\partial_t \tilde{S}_1)^2\mathrm{d} r
				=\int_{1}^{+\infty}(\frac{2 r^2}{3} \partial_{tr}v \partial_t \tilde{S}_1-\frac{2 r}{3} \partial_{t}v \partial_t \tilde{S}_1)\mathrm{d} r
				,
			\end{aligned}
		\end{equation}
		where
		\begin{equation}\nonumber
			\int_{1}^{+\infty}(\frac{r^2 \tau}{3 \mu} \partial_t (\rho v) \partial_t \tilde{S}_1 \partial_r \tilde{S}_1
			-\frac{\tau \epsilon r^2}{3 \mu} \partial_{t}\rho\partial_t\tilde{S}_1\partial_r \tilde{S}_1)\mathrm{d} r
			\geq -C E^{\frac{1}{2}}(t) \mathcal{D}(t),
		\end{equation}
		and
		\begin{equation}
			\begin{aligned}
				\nonumber
				-\int_{1}^{+\infty}\frac{\tau  \epsilon r^2}{3 \mu}\rho(\frac{1}{2}\partial_t\tilde{S}_1^2)_r\mathrm{d}r
				&=\frac{\tau \epsilon}{6 \mu}\rho(\partial_t\tilde{S}_1)^2(t,1)
				+\int_{1}^{+\infty}\frac{\tau  \epsilon }{3 \mu}(r^2\rho)_r\frac{1}{2}\partial_t\tilde{S}_1^2\mathrm{d}r\\
				&\geq\frac{\tau \epsilon}{6 \mu}\rho(t,1)(\partial_t\tilde{S}_1)^2(t,1) -C E^{\frac{1}{2}}(t) \mathcal{D}(t).
			\end{aligned}
		\end{equation}
		
		Multiplying the equation $(\ref{15})_4$ by $\dfrac{r^2}{\lambda} \partial_t \tilde{S}_2$ and integrating over $[1,+\infty)$ with respect to $r$, we have
		\begin{equation}
			\begin{aligned}\label{19}
				\frac{\mathrm{d}}{\mathrm{d} t}& \int_{1}^{+\infty}\left[\frac{\tau }{\lambda} \cdot \frac{r^2 \rho}{2}\left(\partial_t \tilde{S}_2\right)^2\right]\mathrm{d}r
				+\int_{1}^{+\infty}(\frac{r^2 \tau }{\lambda}   \partial_t (\rho v) \partial_t \tilde{S}_2 \partial_r \tilde{S}_2
				-\frac{\tau \epsilon r^2}{\lambda}\partial_{t}\rho\partial_t\tilde{S}_2\partial_r \tilde{S}_2)\mathrm{d}r\\
				&-\int_{1}^{+\infty}\frac{\tau \epsilon r^2}{\lambda}\rho(\frac{1}{2}\partial_t\tilde{S}_2^2)_r\mathrm{d}r
				+\int_{1}^{+\infty}\frac{r^2}{\lambda}(\partial_t \tilde{S}_2)^2\mathrm{d}r
				=\int_{1}^{+\infty}(r^2 \partial_{tr}v \partial_t \tilde{S}_2+2r \partial_{t}v \partial_t \tilde{S}_2)\mathrm{d}r
				,
			\end{aligned}
		\end{equation}
		where 
		\begin{equation}\nonumber
			\int_{1}^{+\infty}(\frac{r^2 \tau }{\lambda}   \partial_t (\rho v) \partial_t \tilde{S}_2 \partial_r \tilde{S}_2
			-\frac{\tau \epsilon r^2}{\lambda}\partial_{t}\rho\partial_t\tilde{S}_2\partial_r \tilde{S}_1)\mathrm{d}r
			\geq -C E^{\frac{1}{2}}(t) \mathcal{D}(t),
		\end{equation}
		and
		\begin{equation}
			\begin{aligned}
				\nonumber
				-\int_{1}^{+\infty}\frac{\tau \epsilon r^2}{\lambda}\rho(\frac{1}{2}\partial_t\tilde{S}_2^2)_r\mathrm{d}r
				&=\frac{\tau \epsilon}{2\lambda}\rho(t,1)(\partial_t\tilde{S}_2)^2(t,1)
				+\int_{1}^{+\infty}\frac{\tau \epsilon }{\lambda}(r^2\rho)_r\frac{1}{2}\partial_t\tilde{S}_2^2\mathrm{d}r
				\geq-C E^{\frac{1}{2}}(t) \mathcal{D}(t).
			\end{aligned}
		\end{equation}
		
		Combining the equations (\ref{3.18})--(\ref{19}) and the above estimates, we derive
		\begin{equation}\label{order-1-t}
			\begin{aligned}
				&\frac{\mathrm{d}}{\mathrm{d} t} \int_{1}^{+\infty}\left(\frac{r^2 P^{\prime}(\rho)}{2\rho} (\partial_{t}\rho)^2+\frac{r^2\rho}{2}(\partial_{t}v)^2
				+\frac{\tau  r^2 \rho}{6 \mu} (\partial_{t}\tilde{S}_1)^2 +\frac{\tau r^2 \rho}{2\lambda} (\partial_{t}\tilde{S}_2)^2\right) \mathrm{d}r\\
				&+\int_{1}^{+\infty}\left(\frac{r^2}{3 \mu}(\partial_t \tilde{S}_1)^2+\frac{r^2}{\lambda}(\partial_t \tilde{S}_2)^2\right)\mathrm{d}r
				+\frac{\tau \epsilon}{8 \mu}(\partial_t\tilde{S}_1)^2(t,1)\\
				&\leq \int_{1}^{+\infty}\left(\frac{2}{3} \partial_r(r^2 \partial_t v \partial_t \tilde{S}_1) 
				+\partial_r(r^2 \partial_t v \partial_t \tilde{S}_2)
				-\partial_r(r^2  \partial_{t}v  \partial_t P)\right)\mathrm{d}r
				+C E^{\frac{1}{2}}(t) \mathcal{D}(t)
				\leq C E^{\frac{1}{2}}(t) \mathcal{D}(t)
			\end{aligned}
		\end{equation}
		where we used the compatibility condition (\ref{compatibility}) and \eqref{hu3.4}. 
		
		By integrating the above inequality over $(0, t)$, and noticing that
		\begin{align*}
			\int_{1}^{+\infty}\left(\frac{r^2 P^{\prime}(\rho)}{2\rho} (\partial_{t}\rho)^2+\frac{r^2\rho}{2}(\partial_{t}v)^2
			+\frac{\tau  r^2 \rho}{6 \mu} (\partial_{t}\tilde{S}_1)^2 +\frac{\tau r^2 \rho}{2\lambda} (\partial_{t}\tilde{S}_2)^2\right) (t=0,r)\mathrm{d}r\\
			\leq C \|V_0^1\|_{L^2}^2\leq C E(0),
		\end{align*}
		\eqref{hu3.13} follows immediately and this finish the proof of  Lemma $\ref{lem2}$.
	\end{proof}
	
	\begin{lemma}\label{lem2-1}
		There exists some constant C such that for any $0\leq t \leq T$
		\begin{equation}
			\begin{aligned}
				&\int_{1}^{+\infty}  r^2\left((\partial_{r}\rho)^2+(\partial_{r}v)^2
				+\tau  (\partial_{r}\tilde{S}_1)^2+\tau (\partial_{r}\tilde{S}_2)^2\right)\mathrm{d}r
				+ \int_0^t \int_{1}^{+\infty}r^2\left(  (\partial_{r}\tilde{S}_1)^2+ (\partial_{r}\tilde{S}_2)^2\right)\mathrm{d}r\mathrm{d}t\\
				&+\int_{0}^{t}\frac{\tau\epsilon}{16\mu}(\partial_{r}\tilde{S}_1)^2(t,1)\mathrm{d}t
				\leq C\left(E(0)+ E^{\frac{1}{2}}(t) \int_{0}^{t}\mathcal{D}(s)\mathrm{~d} t\right)
				.	
			\end{aligned}
		\end{equation}
	\end{lemma}
	\begin{proof}
		Taking derivatives with respect to $r$ to the equations (\ref{3.4}), we get
		\begin{equation}\label{22}
			\begin{cases}
				\partial_{tr} \rho+\partial_{rr}(\rho v)+\dfrac{2}{r} \partial_r(\rho v)-\dfrac{2}{r^2} (\rho v)=0, \\
				\rho\partial_{tr} v+\rho v \partial_{rr} v+\partial_{rr} P-\dfrac{2}{3}\partial_{rr}  \tilde{S}_1-\dfrac{2}{r} \partial_r \tilde{S}_1+\dfrac{2}{r^2}\tilde{S}_1-\partial_{rr} \tilde{S}_2=f_2,\\
				\tau   \rho \partial_{tr} \tilde{S}_1+\tau   \rho v \partial_{rr}\tilde{S}_1 
				-2\mu\left(\partial_{rr} v- \dfrac{1}{r}\partial_r v
				+\dfrac{1}{r^2} v \right)+\partial_r \tilde{S}_1-\tau   \epsilon\rho\partial_{rr}\tilde{S}_1= f_3, \\
				\tau  \rho \partial_{tr} \tilde{S}_2+\tau  \rho v \partial_{rr}\tilde{S}_2
				-\lambda \left(\partial_{rr} v+\dfrac{2}{r}\partial_r v-\dfrac{2}{r^2} v\right)+\partial_r \tilde{S}_2-\tau  \epsilon\rho\partial_{rr}\tilde{S}_2=f_4,
			\end{cases}
		\end{equation}
		where
		\begin{equation}
			\begin{aligned}\nonumber
				&f_2:=-\partial_r \rho \partial_t v-\partial_r \rho v \partial_r v-\rho (\partial_{r}v)^2,\\
				&f_3:=-\tau   \rho\partial_r v\partial_r \tilde{S}_1
				-\tau   \partial_r \rho \left(\partial_t \tilde{S}_1+v\partial_r\tilde{S}_1 \right)+\tau   \epsilon\partial_{r}\rho\partial_r\tilde{S}_1,\\
				&f_4:=-\tau  \rho\partial_r v\partial_r \tilde{S}_2
				-\tau  \partial_r \rho \left(\partial_t \tilde{S}_2+v\partial_r\tilde{S}_2 \right)+\tau  \epsilon\partial_{r}\rho\partial_r\tilde{S}_2.
			\end{aligned}
		\end{equation}
		
		Multiplying the equation $(\ref{22})_{1,2}$ by $r^2 \frac{P^{\prime}(\rho)}{\rho} \partial_r \rho$ and $r^2 \partial_{r}v$, respectively,  and integrating the results, we can obtain
		\begin{equation}\label{27}
			\begin{aligned}
				&\frac{\mathrm{d}}{\mathrm{d} t}\int_{1}^{+\infty}(\dfrac{P^{\prime}(\rho)r^2}{2\rho}\left(\partial_r \rho\right)^2+\dfrac{r^2\rho}{2} (\partial_{r}v)^2) \mathrm{d}r
				+\int_{1}^{+\infty}\big[
				r^2\partial_{rr}\left(\rho v\right)\frac{ \partial_r P}{\rho}
				+2r\partial_{r}\left(\rho v\right)\frac{ \partial_r P}{\rho}-2v\partial_{r}P\big]\mathrm{d}r\\
				&+\int_{1}^{+\infty}\left(r^2 \partial_{rr}P \partial_{r}v-\dfrac{2 r^2}{3} \partial_{rr}\tilde{S}_1 \partial_{r}v-2r \partial_{r}\tilde{S}_1 \partial_{r}v+2\tilde{S}_1\partial_{r}v-r^2 \partial_{rr}\tilde{S}_2\partial_{r}v\right)\mathrm{d}r\\
				&\le\int_{1}^{+\infty}r^2f_2 \partial_{r}v\mathrm{d}r+ C E(t)^{\frac{1}{2}}\mathcal{D}(t),
			\end{aligned}
		\end{equation}
		for which we have 
		\begin{equation}\nonumber
			\begin{aligned}
				&\int_{1}^{+\infty}(r^2  \partial_{rr}(\rho v)  \frac{\partial_r P}{\rho}
				+2 r \partial_r(\rho v)  \frac{\partial_r P}{\rho}
				+r^2 \partial_{rr}P \partial_r v)\mathrm{d}r\\
				=&\int_{1}^{+\infty} (r^2\partial_{rr} v\partial_r P+2r\partial_{r} v\partial_r P+r^2\partial_{rr} P\partial_r v)	\mathrm{d}r
				+\int_{1}^{+\infty}(r^2(2\partial_r\rho\partial_rv
				+v\partial_{rr}\rho)\frac{\partial_r P}{\rho}+2r v \partial_r \rho\partial_r P)\mathrm{d}r\\
				\geq&\int_{1}^{+\infty} \partial_{r}(r^2\partial_{r} v  \partial_r P)\mathrm{d}r
				-C E(t)^{\frac{1}{2}}\mathcal{D}(t),
			\end{aligned}
		\end{equation}
		and
		\begin{equation}\nonumber
			\int_{1}^{+\infty}r^2f_2 \partial_{r}v\mathrm{d}r\leq C E(t)^{\frac{1}{2}}\mathcal{D}(t).
		\end{equation}
		Now, in order to handle the term $-2v\partial_{r}P$ in (\ref{27}), we multiply $(\ref{3.4})_2$ by $2v$ and integrate over $[1,+\infty)$ to get
		\begin{equation}\label{3.29}
			\int_{1}^{+\infty}\left(\rho (v^2)_t+2\rho v^2\partial_{r}v+2v\partial_{r}P-\dfrac{4}{3}v\partial_r \tilde{S}_1-\dfrac{4 v }{r}\tilde{S}_1-2v\partial_{r}\tilde{S}_2\right)\mathrm{d}r=0.
		\end{equation}
		
		Thus, adding equation (\ref{3.29}) to (\ref{27}),
		one can get
		\begin{equation}\label{3.30}
			\begin{aligned}
				&\frac{\mathrm{d}}{\mathrm{d} t}\int_{1}^{+\infty}\left[\dfrac{P^{\prime}(\rho)r^2}{2\rho}\left(\partial_r \rho\right)^2+\dfrac{r^2\rho}{2} (\partial_{r}v)^2+\rho v^2\right]\mathrm{d}r
				+ \int_{1}^{+\infty}\partial_{r}(r^2\partial_{r} v  \partial_r P)	\mathrm{d}r\\
				&
				+\int_{1}^{+\infty}\left(-\dfrac{2 r^2}{3} \partial_{rr}\tilde{S}_1 \partial_{r}v-2r \partial_{r}\tilde{S}_1 \partial_{r}v+2\tilde{S}_1\partial_{r}v-r^2 \partial_{rr}\tilde{S}_2\partial_{r}v\right)\mathrm{d}r\\
				&+\int_{1}^{+\infty}\left(-\dfrac{4}{3}v\partial_r \tilde{S}_1-\dfrac{4 v }{r}\tilde{S}_1- 2v\partial_{r}\tilde{S}_2\right)\mathrm{d}r
				\leq C E(t)^{\frac{1}{2}}\mathcal{D}(t).
			\end{aligned}
		\end{equation}
		
		Multiplying the equation $(\ref{22})_3$ by $\dfrac{r^2}{3 \mu} \partial_r \tilde{S}_1$ and integrating over $[1,+\infty)$,  we have
		\begin{equation}\label{29}
			\begin{aligned}
				&\frac{\mathrm{d}}{\mathrm{d} t}\int_{1}^{+\infty}\left[\frac{\tau  }{3 \mu} \cdot \frac{r^2 \rho}{2}(\partial_r \tilde{S}_1)^2\right]\mathrm{d}r
				+\int_{1}^{+\infty}\frac{r^2}{3 \mu}(\partial_r \tilde{S}_1)^2\mathrm{d}r
				-\int_{1}^{+\infty}\frac{\tau  \epsilon }{3 \mu}r^2\rho(\frac{1}{2}(\partial_r \tilde{S}_1)^2)_r\mathrm{d}r\\
				&-\int_{1}^{+\infty}\bigg(\frac{2r^2}{3}\partial_{rr}v\partial_r \tilde{S}_1 -\frac{2r}{3}\partial_{r}v\partial_r \tilde{S}_1
				+\frac{2}{3}v\partial_r \tilde{S}_1\bigg)\mathrm{d}r
				=\int_{1}^{+\infty}\frac{r^2}{3\mu}f_3\partial_r \tilde{S}_1\mathrm{d}r ,
			\end{aligned}
		\end{equation}
		where 
		\begin{equation}
			\begin{aligned}
				\nonumber
				-\int_{1}^{+\infty}\frac{\tau \epsilon }{3\mu}r^2\rho(\frac{1}{2}(\partial_r\tilde{S}_1)^2)_r\mathrm{d}r
				&=\frac{\tau \epsilon }{6\mu}\rho (\partial_r\tilde{S}_1)^2(t,1)
				+\int_{1}^{+\infty}\frac{\tau \epsilon }{3\mu}(2r\rho+r^2\partial_{r}\rho)\frac{1}{2}(\partial_r\tilde{S}_1)^2\mathrm{d}r\\
				%&= \frac{\tau \epsilon }{6\mu}\rho(t,1)(\partial_r\tilde{S}_1)^2(t,1)+ \int_{1}^{+\infty}\frac{\tau\epsilon r}{3\mu} \rho (\partial_r\tilde{S}_1)^2\mathrm{d}r
				%+\int_{1}^{+\infty}\frac{\tau\epsilon r^2} {6\mu} \partial_{r}\rho(\partial_r\tilde{S}_1)^2\mathrm{d}r\\
				%&\geq\frac{\tau \epsilon }{6\mu}\rho(t,1)(\partial_r\tilde{S}_1)^2(t,1)
				%+\int_{1}^{+\infty}\frac{\tau\epsilon r^2} {6\mu} \partial_{r}\rho(\partial_r\tilde{S}_1)^2\mathrm{d}r\\
				&
				\geq \frac{\tau \epsilon }{6\mu}\rho(t,1)(\partial_r\tilde{S}_1)^2(t,1)-C E^{\frac{1}{2}}(t) \mathcal{D}(t),
			\end{aligned}
		\end{equation}
		and
		\begin{equation}\nonumber
			\int_{1}^{+\infty}\frac{r^2}{3\mu}f_3\partial_r \tilde{S}_1\mathrm{d}r\leq C E^{\frac{1}{2}}(t) \mathcal{D}(t).
		\end{equation}
		Similarly,  multiplying the equation $(\ref{22})_4$ by $\dfrac{r^2}{\lambda} \partial_r \tilde{S}_2$,  we get
		\begin{equation}\label{30}
			\begin{aligned}
				&\frac{\mathrm{d}}{\mathrm{d} t}\int_{1}^{+\infty}\left[\frac{\tau }{\lambda} \cdot \frac{r^2 \rho}{2}(\partial_r \tilde{S}_2)^2\right]\mathrm{d}r
				+\int_{1}^{+\infty}\frac{r^2}{\lambda}(\partial_r \tilde{S}_2)^2\mathrm{d}r
				-\int_{1}^{+\infty}\frac{\tau \epsilon }{\lambda}r^2\rho(\frac{1}{2}(\partial_r \tilde{S}_2)^2)_r\mathrm{d}r\\
				&-\int_{1}^{+\infty}\bigg(r^2\partial_{rr}v\partial_r \tilde{S}_2+2r\partial_{r}v\partial_r \tilde{S}_2
				-2v\partial_r \tilde{S}_2\bigg)\mathrm{d}r
				=\int_{1}^{+\infty}\frac{r^2}{\lambda}f_4\partial_r \tilde{S}_2\mathrm{d}r,
			\end{aligned}
		\end{equation}
		where
		\begin{equation}
			\begin{aligned}
				\nonumber
				-\int_{1}^{+\infty}\frac{\tau \epsilon }{\lambda}r^2\rho(\frac{1}{2}(\partial_r\tilde{S}_2)^2)_r\mathrm{d}r
				&=\frac{\tau \epsilon }{2\lambda}\rho (\partial_r\tilde{S}_2)^2(t,1)
				+\int_{1}^{+\infty}\frac{\tau \epsilon }{\lambda}(2r\rho+r^2\partial_{r}\rho)\frac{1}{2}(\partial_r\tilde{S}_2)^2\mathrm{d}r\\
				&\geq-C E^{\frac{1}{2}}(t) \mathcal{D}(t),
			\end{aligned}
		\end{equation}
		and
		\begin{equation}\nonumber
			\int_{1}^{+\infty}\frac{r^2}{\lambda}f_4\partial_r \tilde{S}_2\mathrm{d}r\leq C E^{\frac{1}{2}}(t) \mathcal{D}(t).
		\end{equation}
		
		Combining the equations (\ref{3.30})--(\ref{30}) and the above estimates, we derive that
		\begin{equation}\label{3.33}
			\begin{aligned}
				&\frac{\mathrm{d}}{\mathrm{d} t}\int_{1}^{+\infty}(\dfrac{P^{\prime}(\rho)r^2}{2\rho}\left(\partial_r \rho\right)^2+\dfrac{r^2\rho}{2} (\partial_{r}v)^2+\rho v^2
				+ \frac{\tau r^2\rho}{6\mu}(\partial_r \tilde{S}_1)^2
				+ \frac{\tau r^2\rho}{2\lambda}(\partial_r \tilde{S}_2)^2) \mathrm{d}r\\
				&
				+ \int_{1}^{+\infty}(\partial_{r}(r^2\partial_{r} v  \partial_r P)	-\dfrac{2}{3}\partial_{r}(r^2\partial_{r}v\partial_{r}\tilde{S}_1)-\partial_{r}(r^2\partial_{r}v\partial_{r}\tilde{S}_2))\mathrm{d}r\\
				&+\int_{1}^{+\infty}(\frac{r^2}{3\mu}(\partial_r \tilde{S}_1)^2+\frac{r^2}{\lambda}(\partial_r \tilde{S}_2)^2)\mathrm{d}r
				+\frac{\tau\epsilon}{6\mu}\rho(t,1)(\partial_{r}\tilde{S}_1)^2(t,1)
				\\
				&+\int_{1}^{+\infty}\left(2\tilde{S}_1\partial_{r}v-2v\partial_r \tilde{S}_1-\dfrac{4 v }{r}\tilde{S}_1\right)\mathrm{d}r\\
				&\leq C E(t)^{\frac{1}{2}}\mathcal{D}(t).
			\end{aligned}
		\end{equation}
		To address the last term on the left-hand side of equation (\ref{3.33}), we
		multiply $(\ref{3.4})_3$ by $\frac{2}{\mu}\tilde{S}_1$ and add the result to equation (\ref{3.33}) to get
		\begin{equation}\label{3.34}
			\begin{aligned}
				&\frac{\mathrm{d}}{\mathrm{d} t}\int_{1}^{+\infty}(\dfrac{P^{\prime}(\rho)r^2}{2\rho}\left(\partial_r \rho\right)^2+\dfrac{r^2\rho}{2} (\partial_{r}v)^2+\rho v^2
				+ \frac{\tau r^2\rho}{6\mu}(\partial_r \tilde{S}_1)^2
				+ \frac{\tau r^2\rho}{2\lambda}(\partial_r \tilde{S}_2)^2) \mathrm{d}r\\
				&+\int_{1}^{+\infty}(\frac{r^2}{3\mu}(\partial_r \tilde{S}_1)^2+\frac{r^2}{\lambda}(\partial_r \tilde{S}_2)^2)\mathrm{d}r
				+\frac{\tau\epsilon}{6\mu}\rho(t,1)(\partial_{r}\tilde{S}_1)^2 (t,1)
				\\
				&
				+ \int_{1}^{+\infty}\left(\partial_{r}(r^2\partial_{r} v  \partial_r P)	-\dfrac{2}{3}\partial_{r}(r^2\partial_{r}v\partial_{r}\tilde{S}_1)-\partial_{r}(r^2\partial_{r}v\partial_{r}\tilde{S}_2)\right)\mathrm{d}r\\
				&+\int_{1}^{+\infty}\left(2\tilde{S}_1\partial_{r}v-2v\partial_r \tilde{S}_1-\dfrac{4 v }{r}\tilde{S}_1\right)\mathrm{d}r\\
				&+\int_{1}^{+\infty}\left[\left(\tau   \rho \left(\partial_t \tilde{S}_1+(v-\epsilon)\partial_r\tilde{S}_1 \right) +\tilde{S}_1-2\mu\left(\partial_r v- \dfrac{v}{r}\right)\right) \frac{2}{\mu}\tilde{S}_1\right] \mathrm{d}r
				\leq C E(t)^{\frac{1}{2}}\mathcal{D}(t),
			\end{aligned}
		\end{equation}
		where
		\begin{equation}\nonumber
			-\int_{1}^{+\infty}\frac{2\tau\epsilon}{\mu}\rho\tilde{S}_1
			\partial_r\tilde{S}_1\mathrm{d}r
			=\frac{\tau\epsilon}{\mu}\rho(t,1)\tilde{S}_1^2(t,1)+\int_{1}^{+\infty}\frac{\tau\epsilon}{\mu}\partial_r\rho\tilde{S}_1^2
			\mathrm{d}r
			\geq %\frac{\tau\epsilon}{\mu}\rho(t,1)\tilde{S}_1^2(t,1)
			-C E(t)^{\frac{1}{2}}\mathcal{D}(t).
		\end{equation}
		Integrating the equation (\ref{3.34}) over $(0,t)$ and using \eqref{hu3.4}, we have
		\begin{equation}\label{34}
			\begin{aligned}
				&\int_{1}^{+\infty}\bigg(\dfrac{r^2 P^{\prime}(\rho)}{2\rho} \left(\partial_r \rho\right)^2 +\dfrac{1}{2}\rho\left(r^2 (\partial_rv)^2+2 v^2\right)+\dfrac{\tau  }{4 \mu} \rho \left[\dfrac{2}{3} r^2(\partial_r\tilde{S}_1)^2+4 \tilde{S}_1^2\right]
				+\dfrac{\tau }{2\lambda}\rho r^2 (\partial_r \tilde{S}_2)^2\bigg)\mathrm{d}r\\
				&+\int_{0}^{t}\int_{1}^{+\infty}\bigg(\dfrac{r^2}{3\mu} \left(\partial_r\tilde{S}_1\right)^2+\dfrac{2}{\mu}\tilde{S}_1^2
				+\dfrac{r^2}{\lambda}\left(\partial_r\tilde{S}_2\right)^2
				\bigg)\mathrm{d}r\mathrm{d}t
				+\int_{0}^{t}\frac{3\tau\epsilon}{8\mu}(\partial_{r}\tilde{S}_1)^2 (t,1)\mathrm{d}t\\
				&\leq \int_{0}^{t} r^2\partial_{r}v\left(\partial_{r}P-\dfrac{2}{3}\partial_r\tilde{S}_1-\partial_r\tilde{S}_2\right)(t,1)\mathrm{d}t
				-\int_{0}^{t}2v\tilde{S}_1(t,1)\mathrm{d}t
				+C\left(E(0)+ E^{\frac{1}{2}}(t)\int_{0}^{t} \mathcal{D}(s)\mathrm{d}s\right)
				.
			\end{aligned}
		\end{equation}
		Using the boundary condition (\ref{1.6}), we know $-\int_{0}^{t}2v\tilde{S}_1(t,1)\mathrm{d}t=0$. For the first term on the right hand side of equation \eqref{34}, we have
		\begin{equation}
			\begin{aligned}
				\left[r^2\partial_r v \left( \partial_r P - \dfrac{2}{3}\partial_r \tilde{S}_1 - \partial_r \tilde{S}_2 \right)\right](t,1)
				&= r^2\partial_r v (\partial_r P - \dfrac{2}{3}\partial_r \tilde{S}_1 - \dfrac{2}{r}\tilde{S}_1 - \partial_r \tilde{S}_2 )(t,1) + (2r\tilde{S}_1 \partial_r v) (t,1) \\
				&= -[r^2\partial_r v\left( \rho\partial_t v + \rho v\partial_r v \right)](t,1)+ (2r\tilde{S}_1 \partial_r v) (t,1) \\
				&=  2(\tilde{S}_1 \partial_r v) (t,1),
			\end{aligned}
		\end{equation}
		where we used the momentum equation $(\ref{3.4})_2$ and boundary condition (\ref{1.6}).
		By using the equation $(\ref{3.4})_3$ and $\epsilon-Young$ inequality, we have 
		\begin{equation}\nonumber
			\begin{aligned}
				\int_{0}^{t}2(\tilde{S}_1 \partial_r v )(t,1) \mathrm{d}t
				=&\int_{0}^{t}\left[2\tilde{S}_1\left(\dfrac{\tau  \rho}{2\mu}(\partial_t\tilde{S}_1+(v-\epsilon)\partial_r\tilde{S}_1)+\dfrac{1}{2\mu}\tilde{S}_1+\dfrac{v}{r}\right)\right](t,1)\mathrm{d}t\\
				=&\int_{0}^{t}\left[\dfrac{\tau \rho}{\mu}(\dfrac{1}{2}\tilde{S}_1^2)_t(t,1)
				-\dfrac{\epsilon\tau }{\mu} \rho \tilde{S}_1\partial_r\tilde{S}_1(t,1)
				+\dfrac{1}{\mu}\tilde{S}_1^2(t,1)\right]\mathrm{d}t\\
				\leq& \dfrac{\tau}{2\mu} \rho\tilde{S}_1^2(t,1)\big|_0^t
				-\int_{0}^{t}\frac{\tau}{2\mu}\partial_{t}\rho\tilde{S}_1^2(t,1)\mathrm{d}t
				+\int_{0}^{t}\dfrac{1}{\mu}\tilde{S}_1^2(t,1)\mathrm{d}t\\
				&+\int_{0}^{t}\dfrac{\epsilon\tau }{\mu}\rho(t,1)\left(\eta(\partial_{r}\tilde{S}_1)^2(t,1)+C(\eta)\tilde{S}_1^2(t,1)\right)\mathrm{d}t
			\end{aligned}
		\end{equation}
		where
		$$
		\begin{aligned}
			-\int_{0}^{t}\frac{\tau}{2\mu}\partial_{t}\rho\tilde{S}_1^2(t,1)\mathrm{d}t
			&\leq
			\int_{0}^{t}\frac{\tau}{2\mu}\|\partial_{t}\rho\tilde{S}_1^2\|_{L^{\infty}}\mathrm{d}t
			\leq \int_{0}^{t}\frac{\tau}{2\mu}\|\partial_{t}\rho\|_{H^{1}}\|\tilde{S}_1\|_{H^{1}}^2\mathrm{d}t\\
			&\leq C\left(E(0)+ E^{\frac{1}{2}}(t)\int_{0}^{t} \mathcal{D}(s)\mathrm{d}s\right),
		\end{aligned}
		$$
		and
		\begin{equation}
			\begin{aligned}\nonumber
				&\int_{0}^{t}\dfrac{1}{\mu} \tilde{S}_1^2(t,1)\mathrm{d}t
				\leq \int_{0}^{t}(\dfrac{1}{\mu}\|\tilde{S}_1\|_{L^\infty}^2)\mathrm{d}t
				\leq \int_{0}^{t}\dfrac{1}{\mu}\|\partial_{r}\tilde{S}_1\|_{L^2}\|\tilde{S}_1\|_{L^2}\mathrm{d}t\\
				&\leq \int_{0}^{t}\dfrac{1}{\mu} (\eta\|\partial_{r}\tilde{S}_1\|_{L^2}^2+C(\eta)\|\tilde{S}_1\|_{L^2}^2)\mathrm{d}t
				\leq \int_{0}^{t}\dfrac{\eta}{\mu}\|\partial_{r}\tilde{S}_1\|_{L^2}^2\mathrm{d}t+C\left(E(0)+E^{\frac{1}{2}}(t) \int_{0}^t \mathcal{D}(s)\mathrm{d}s\right).
			\end{aligned}
		\end{equation}
		Here and after, $\eta$ is any positive constant to be chosen later.
		
		Combining the above results, Lemma \ref{lem1} and \eqref{hu3.4}, we derive
		\begin{equation*}
			\begin{aligned}
				&\int_{1}^{+\infty}\bigg(\dfrac{r^2 P^{\prime}(\rho)}{2\rho} \left(\partial_r \rho\right)^2 +\dfrac{\rho}{2}\left[r^2 (\partial_rv)^2+2 v^2\right]+\dfrac{\tau\rho r^2}{6\mu}   \left(\partial_r\tilde{S}_1\right)^2
				+\dfrac{\tau\rho}{\mu}  \tilde{S}_1^2
				+\dfrac{\tau \rho r^2}{2\lambda}  \left(\partial_r \tilde{S}_2\right)^2
				\bigg)\mathrm{d}r\\
				&+\int_0^t\int_{1}^{+\infty}\left(\dfrac{r^2}{4\mu} \left(\partial_r\tilde{S}_1\right)^2+\dfrac{2}{\mu}\tilde{S}_1^2
				+\dfrac{r^2}{\lambda}\left(\partial_r\tilde{S}_2\right)^2 \right)\mathrm{d}r\mathrm{d}t
				+\int_{0}^{t}\frac{\tau\epsilon}{16\mu}(\partial_{r}\tilde{S}_1)^2 (t,1)\mathrm{d}t\\
				&\leq \dfrac{\tau}{2\mu} \rho(t,1)\tilde{S}_1^2(t,1)
				+C\left(E(0)+ E^{\frac{1}{2}}(t)\int_{0}^{t} \mathcal{D}(s)\mathrm{d}s\right)
				,
			\end{aligned}
		\end{equation*}
		and, by applying the {\it Gagliardo-Nirenberg} inequality and {\it Young} inequality, we have 
		$$\begin{aligned}
			\dfrac{\tau}{2\mu} \rho(t,1)\tilde{S}_1^2(t,1)
			&\leq\dfrac{\tau}{\mu} r^2 \|\tilde{S}_1\|_{L^\infty}^2
			\leq \dfrac{\tau}{\mu}r^2\|\partial_{r}\tilde{S}_1\|_{L^2}\|\tilde{S}_1\|_{L^2}
			\leq \dfrac{\tau}{\mu}r^2 (\eta\|\partial_{r}\tilde{S}_1\|_{L^2}^2+C(\eta)\|\tilde{S}_1\|_{L^2}^2)\\
			&
			\leq \dfrac{\tau\eta}{\mu}r^2\|\partial_{r}\tilde{S}_1\|_{L^2}^2+
			C\left(E(0)+ E^{\frac{1}{2}}(t)\int_{0}^{t} \mathcal{D}(s)\mathrm{d}s\right).
		\end{aligned}$$
		By choosing $\eta<\frac{1}{12}$, we get the Lemma $\ref{lem2-1}$ immediately.
		%The estimates $(\ref{order-1-t})$ and $(\ref{order-1-r})$ prove Lemma $\ref{lem2}$.
	\end{proof}
	The next lemma show the dissipative estimates of $D(\rho, v)$.
	\begin{lemma}\label{lem3.4}
		There exists a constant $C$ such that for any $0\le t \le T$ 
		\begin{equation}
			\begin{aligned}
				\int_0^t\int_{1}^{+\infty}(v^2+r^2|D(\rho, v)|^2)\mathrm{d}r\mathrm{d}t
				\leq C\left(E(0)+E^{\frac{1}{2}}(t)\int_{0}^{t} \mathcal{D}(s)\mathrm{d}s\right).
			\end{aligned}
		\end{equation}
	\end{lemma}
	\begin{proof}
		By manipulating $\dfrac{\lambda}{\mu}(\ref{3.4})_{3}+(\ref{3.4})_{4}$
		and applying Lemma \ref{lem2}, we derive
		
		\begin{equation}\label{dis1}
			\int_0^t\int_{1}^{+\infty}r^2(\partial_{r}v)^2
			\mathrm{d}r\mathrm{d}t
			\leq C(E(0)+E^{\frac{1}{2}}(t)\int_{0}^{t} \mathcal{D}(s)\mathrm{d}s).
		\end{equation}
		Similarly, by computing $-\dfrac{\lambda}{2\mu}(\ref{3.4})_{3}+(\ref{3.4})_{4}$, one can get
		\begin{equation}
			\int_0^t\int_{1}^{+\infty}v^2
			\mathrm{d}r\mathrm{d}t
			\leq C(E(0)+E^{\frac{1}{2}}(t)\int_{0}^{t} \mathcal{D}(s)\mathrm{d}s).
		\end{equation}
		Using mass  equation $(\ref{3.4})_{1}$, we get immediately
		\begin{equation}\label{dis2}
			\int_0^t\int_{1}^{+\infty}r^2(\partial_{t}\rho)^2
			\mathrm{d}r\mathrm{d}t
			\leq C(E(0)+E^{\frac{1}{2}}(t)\int_{0}^{t} \mathcal{D}(s)\mathrm{d}s).
		\end{equation}
		On the other hand, multiplying the equation
		$(\ref{3.4})_{2}$ by $r^2\partial_{t}v$ yields
		\begin{equation}\nonumber
			\begin{aligned}
				&\int_0^t\int_{1}^{+\infty}r^2(\partial_{t}v)^2\mathrm{d}r\mathrm{d}t\\
				=&\int_0^t\int_{1}^{+\infty}(-r^2\partial_{r}P\partial_{t}v+\frac{2r^2}{3}\partial_r\tilde{S}_1\partial_{t}v+2r\tilde{S}_1\partial_{t}v
				+r^2\partial_r\tilde{S}_2\partial_{t}v
				-r^2\rho v \partial_{r}v\partial_{t}v)\mathrm{d}r\mathrm{d}t\\
				=&-\int_0^t\frac{\mathrm{d}}{\mathrm{d}t}\int_{1}^{+\infty}r^2\partial_{r}Pv\mathrm{d}r\mathrm{d}t
				-\int_0^t\int_{1}^{+\infty}r^2\partial_{t}P\partial_{r}v\mathrm{d}r\mathrm{d}t
				-\int_0^t\int_{1}^{+\infty}2r\partial_{t}Pv\mathrm{d}r\mathrm{d}t
				\\
				&+\int_0^t\int_{1}^{+\infty}\frac{2r^2}{3}\partial_r\tilde{S}_1\partial_{t}v\mathrm{d}r\mathrm{d}t
				+\int_0^t\int_{1}^{+\infty}2r\tilde{S}_1\partial_{t}v\mathrm{d}r\mathrm{d}t\\
				&+\int_0^t\int_{1}^{+\infty}r^2\partial_r\tilde{S}_2\partial_{t}v\mathrm{d}r\mathrm{d}t
				-\int_0^t\int_{1}^{+\infty}r^2\rho v \partial_{r}v\partial_{t}v\mathrm{d}r\mathrm{d}t\\
				\leq&\frac{1}{2}\int_0^t\int_{1}^{+\infty}r^2(\partial_{t}v)^2\mathrm{d}r\mathrm{d}t
				+C\int_{1}^{+\infty}r^2((\partial_{r}\rho)^2+v^2)\mathrm{d}r
				+C(E(0)+E^{\frac{1}{2}}(t)\int_{0}^{t} \mathcal{D}(s)\mathrm{d}s)\\
				&+C\int_0^t\int_{1}^{+\infty}\left(r^2((\partial_{t}\rho)^2
				+(\partial_{r}v)^2+(\partial_r\tilde{S}_1)^2 +\tilde{S}_1^2+(\partial_r\tilde{S}_2)^2)
				+v^2\right)
				\mathrm{d}r\mathrm{d}t
			\end{aligned}
		\end{equation}
		which implies
		\begin{equation}\label{dis3}
			\int_0^t\int_{1}^{+\infty}r^2(\partial_{t}v)^2
			\mathrm{d}r\mathrm{d}t
			\leq C(E(0)+E^{\frac{1}{2}}(t)\int_{0}^{t} \mathcal{D}(s)\mathrm{d}s).
		\end{equation}
		
		By equation $(\ref{3.4})_2$, we also get
		\begin{equation}\label{dis4}
			\int_0^t\int_{1}^{+\infty}r^2(\partial_{r}\rho)^2
			\mathrm{d}r\mathrm{d}t
			\leq C(E(0)+E^{\frac{1}{2}}(t)\int_{0}^{t} \mathcal{D}(s)\mathrm{d}s).
		\end{equation}
		
		Thus, from the equations (\ref{dis1})-(\ref{dis4}), we get the desired result.
	\end{proof}
	
	Combining Lemmas \ref{lem1}-\ref{lem3.4},  we get the following lemma.
	\begin{lemma}\label{lem3.5}
		There exists some constant C such that for any $0\le t \le T$
		\begin{equation}
			\begin{aligned}
				&\sum_{|\alpha|=0}^1\left\|rD^\alpha(\rho-1, v, \sqrt{\tau} \tilde{S}_1, \sqrt{\tau} \tilde{S}_2)\right\|_{L^{2}}^2
				+\int_{0}^{t}\bigg(\left\|r D(\rho, v)\right\|_{L^2}^2
				+\sum_{|\alpha|=0}^1\left\|r D^\alpha(\tilde{S}_1, \tilde{S}_2)\right\|_{L^{2}}^2+v^2\bigg)\mathrm{~d} t
				\\&
				\leq C \left(E(0)+E^{\frac{1}{2}}(t) \int_{0}^{t}\mathcal{D}(s)\mathrm{d}s\right).
			\end{aligned}
		\end{equation}
	\end{lemma}
	\subsection{Second-order estimates}
	\begin{lemma}\label{lem3}
		There exists some constant C such that for any $0\le t \le T$
		\begin{equation}
			\begin{aligned}
				&\int_{1}^{+\infty} \tau^2r^2 \left((\partial_{tt}\rho)^2+(\partial_{tt}v)^2+\tau(\partial_{tt}\tilde{S}_1)^2+\tau(\partial_{tt}\tilde{S}_2)^2\right) \mathrm{d} r \\
				&+\int_0^t \int_{1}^{+\infty}\tau^2r^2\left((\partial_{tt}\tilde{S}_1)^2+(\partial_{tt}\tilde{S}_2)^2\right) \mathrm{d} r \mathrm{d} t 
				\leq C\left(E(0)+ E^{\frac{1}{2}}(t) \int_{0}^t \mathcal{D}(s)\mathrm{d}s\right).
			\end{aligned}
		\end{equation}
	\end{lemma}
	\begin{proof}
		Taking derivative with respect to $t$ twice to equation $(\ref{3.4})$ , one get
		
		\begin{equation} \label{38}
			\begin{cases}
				\partial_{ttt} \rho+\partial_{ttr}(\rho v)+\dfrac{2}{r}\partial_{tt}(\rho v)=0, \\
				\rho\partial_{ttt} v+2\partial_{t}\rho \partial_{tt}v+ \partial_{tt}\rho \partial_{t}v+\partial_{tt}\rho (v \partial_r v)+2\partial_{t}\rho\partial_{t}(v\partial_{r}v)+\rho\partial_{tt}(v\partial_{r}v)\\
				\qquad\qquad\qquad\qquad\qquad\qquad\qquad\qquad\qquad\qquad+\partial_{ttr} P=\dfrac{2}{3}\partial_{ttr}  \tilde{S}_1+\dfrac{2}{r} \partial_{tt}\tilde{S}_1+\partial_{ttr} \tilde{S}_2,\\
				\tau \rho \left(\partial_{ttt} \tilde{S}_1+\partial_{tt} ((v-\epsilon)\partial_r\tilde{S}_1) \right)
				+2\tau \partial_t \rho \left(\partial_{tt} \tilde{S}_1
				+\partial_{t} ((v-\epsilon)\partial_r\tilde{S}_1) \right) \\
				\qquad\qquad\qquad\qquad\qquad+\tau \partial_{tt} \rho \left(\partial_t \tilde{S}_1+(v-\epsilon)\partial_r\tilde{S}_1 \right) +\partial_{tt} \tilde{S}_1=2\mu\left(\partial_{ttr} v- \dfrac{1}{r}\partial_{tt} v \right), \\
				\tau  \rho \left(\partial_{ttt} \tilde{S}_2+\partial_{tt} ((v-\epsilon)\partial_r\tilde{S}_2) \right)
				+2\tau  \partial_t \rho \left(\partial_{tt} \tilde{S}_2+\partial_{t} ((v-\epsilon)\partial_r\tilde{S}_2) \right) \\
				\qquad\qquad\qquad\qquad\qquad+\tau  \partial_{tt} \rho \left(\partial_t \tilde{S}_2+(v-\epsilon)\partial_r\tilde{S}_2 \right) +\partial_{tt} \tilde{S}_2=\lambda \left(\partial_{ttr} v+\dfrac{2}{r}\partial_{tt} v\right).
			\end{cases}
		\end{equation}
		Multiplying equations
		$(\ref{38})_1$ and $(\ref{38})_2$ 
		by $r^2\frac{{P}^{\prime}(\rho)}{\rho}\partial_{tt}\rho$ and $r^2\partial_{tt}v$, respectively, and  integrating the result, we obtain
			\begin{align}\label{3.47}
				\nonumber
				&\frac{\mathrm{d}}{\mathrm{d} t}\int_{1}^{+\infty}\bigg(\frac{r^2 P^{\prime}(\rho)}{2\rho} (\partial_{tt}\rho)^2+\frac{r^2\rho}{2}(\partial_{tt}v)^2
				\bigg)\mathrm{d}r\\ \nonumber
				&+\int_{1}^{+\infty}\bigg(-[\frac{ P^{\prime}(\rho)}{\rho}]_t\frac{r^2}{2}(\partial_{tt}\rho)^2
				+2r^2\partial_{t}\rho (\partial_{tt}v)^2
				+r^2\partial_{tt}\rho [\frac{1}{2}(\partial_{t}v)^2]_t
				+r^2\partial_{tt}\rho \partial_{tt}v (\frac{v^2}{2})_r\\ \nonumber
				&\qquad\qquad\quad+r^2\partial_{t}\rho \partial_{r}v \left[(\partial_{t}v)^2\right]_t
				+2r^2\partial_{t}\rho v \partial_{tt}v \partial_{tr}v
				+r^2\rho \partial_{r}v (\partial_{tt}v)^2+r^2\rho \left[(\partial_{t}v)^2\right]_r \partial_{tt}v\bigg)\mathrm{~~d}r\\ \nonumber
				&+\int_{1}^{+\infty}\left(r^2\partial_{ttr}(\rho v)\frac{ P^{\prime}(\rho)}{\rho}\partial_{tt}\rho
				+2r\partial_{tt}(\rho v)\frac{ P^{\prime}(\rho)}{\rho}\partial_{tt}\rho+r^2\partial_{ttr}P \partial_{tt}v\right) \mathrm{d}r\\ 
				=&\int_{1}^{+\infty}\left(\frac{2}{3}r^2\partial_{tt}v\partial_{ttr} \tilde{S}_1 +2r\partial_{tt}v\partial_{tt} \tilde{S}_1+r^2\partial_{tt}v\partial_{ttr} \tilde{S}_2\right)\mathrm{d}r.
			\end{align}

		It is straightforward to see that the second term in the above equation can be estimated by $CE(t)^{\frac{1}{2}} \mathcal{D}(t)$.
		We now focus on the last term on the left-hand side of the above equation, which is given by
		\begin{equation}\nonumber
			\begin{aligned}
				&\int_{1}^{+\infty}\left(r^2\partial_{ttr}(\rho v)\frac{ P^{\prime}(\rho)}{\rho}\partial_{tt}\rho
				+2r\partial_{tt}(\rho v)\frac{ P^{\prime}(\rho)}{\rho}\partial_{tt}\rho+r^2\partial_{ttr}P \partial_{tt}v\right) \mathrm{d}r\\
				=&\int_{1}^{+\infty} \left(r^2\partial_{ttr}vP^{\prime}(\rho)\partial_{tt}\rho
				+2r\partial_{tt}vP^{\prime}(\rho)\partial_{tt}\rho+r^2\partial_{ttr}P \partial_{tt}v\right)\mathrm{d}r \\
				&+\int_{1}^{+\infty}\bigg(\frac{r^2 P^{\prime}(\rho)}{\rho}\partial_{tt}\rho
				\left(\partial_{r}\rho \partial_{tt}v+2\partial_{t}\rho \partial_{tr}v +2\partial_{tr}\rho \partial_{t}v+\partial_{tt}\rho \partial_{r}v\right)\\
				&-\left(r^2\frac{ P^{\prime}(\rho)}{\rho}v\right)_r \frac{1}{2}(\partial_{tt}\rho)^2
				+2r\frac{ P^{\prime}(\rho)}{\rho} \partial_{tt}\rho\left(2\partial_{t}\rho \partial_{t}v+\partial_{tt}\rho v\right)\bigg) \mathrm{d}r\\
				\geq&\int_{1}^{+\infty} \left(r^2\partial_{ttr}vP^{\prime}(\rho)\partial_{tt}\rho
				+2r\partial_{tt}vP^{\prime}(\rho)\partial_{tt}\rho+r^2\partial_{ttr}P \partial_{tt}v\right) \mathrm{d}r
				-C E(t)^{\frac{1}{2}} \mathcal{D}(t).
			\end{aligned}
		\end{equation}
		Furthermore,
		\begin{equation}\nonumber
			\begin{aligned}
				&\int_{1}^{+\infty} \left(r^2\partial_{ttr}vP^{\prime}(\rho)\partial_{tt}\rho
				+2r\partial_{tt}vP^{\prime}(\rho)\partial_{tt}\rho+r^2\partial_{ttr}P \partial_{tt}v\right) \mathrm{d}r\\
				=&\int_{1}^{+\infty}\left(\partial_{r}\left(r^2\partial_{tt}P\partial_{tt}v\right)-r^2P^{\prime\prime}(\rho)(\partial_{t}\rho)^2\partial_{ttr}v
				-2rP^{\prime\prime}(\rho)(\partial_{t}\rho)^2\partial_{tt}v\right)\mathrm{d}r\\
				=&\int_{1}^{+\infty}\left(r^2 \partial_{r}\left(P^{\prime\prime}(\rho)(\partial_{t}\rho)^2\right)\partial_{tt}v\right)\mathrm{d}r \ge-CE(t)^{\frac{1}{2}} \mathcal{D}(t).
			\end{aligned}
		\end{equation}
		Therefore, by combining the above estimation results, equation (\ref{3.47}) reduces to
		\begin{equation}\label{39}
			\begin{aligned}
				&\frac{\mathrm{d}}{\mathrm{d} t}\int_{1}^{+\infty}\bigg(\frac{r^2 P^{\prime}(\rho)}{2\rho} (\partial_{tt}\rho)^2+\frac{r^2\rho}{2}(\partial_{tt}v)^2
				\bigg)\mathrm{d}r\\
				&\leq\int_{1}^{+\infty}\left(\frac{2}{3}r^2\partial_{tt}v\partial_{ttr} \tilde{S}_1 +2r\partial_{tt}v\partial_{tt} \tilde{S}_1+r^2\partial_{tt}v\partial_{ttr} \tilde{S}_2\right)\mathrm{d}r+CE(t)^{\frac{1}{2}} \mathcal{D}(t).
			\end{aligned}
		\end{equation}
		
		Multiplying the equation $(\ref{38})_3$ 
		by  $\dfrac{r^2}{3\mu}\partial_{tt} \tilde{S}_1$ and integrating over $[1,+\infty)$ with respect to $r$ yields
		\begin{equation}\label{42}
			\begin{aligned}
				&\frac{\mathrm{d}}{\mathrm{d} t}\int_{1}^{+\infty}\left[\frac{\tau}{3 \mu} \cdot \frac{r^2 \rho}{2}(\partial_{tt} \tilde{S}_1)^2\right]\mathrm{d}r
				+\int_{1}^{+\infty}\frac{r^2}{3 \mu}(\partial_{tt}\tilde{S}_1)^2\mathrm{d}r\\
				&+\int_{1}^{+\infty}\bigg(\frac{2\tau}{3 \mu}r^2\partial_{t}\rho(\partial_{tt} \tilde{S}_1)^2+\frac{2\tau}{3 \mu}r^2\rho\partial_{t}v\partial_{tr} \tilde{S}_1\partial_{tt} \tilde{S}_1+\frac{\tau}{3 \mu}r^2\rho\partial_{tt}v\partial_{r} \tilde{S}_1\partial_{tt} \tilde{S}_1
				+\frac{2\tau}{3 \mu}r^2\partial_{t}\rho\partial_{t}v\partial_{r} \tilde{S}_1\partial_{tt} \tilde{S}_1\\
				&\qquad\qquad\qquad
				+\frac{2\tau}{3 \mu}r^2\partial_{t}\rho v \partial_{tr} \tilde{S}_1\partial_{tt} \tilde{S}_1
				+\frac{\tau}{3 \mu}r^2\partial_{tt}\rho[\frac{1}{2}(\partial_{t} \tilde{S}_1)^2]_t
				+\frac{\tau}{3 \mu}r^2\partial_{tt}\rho v\partial_{r} \tilde{S}_1\partial_{tt} \tilde{S}_1\bigg)\mathrm{d}r\\
				&+\int_{1}^{+\infty}
				-\frac{\tau\epsilon}{3 \mu}r^2\rho[\frac{1}{2}(\partial_{tt}\tilde{S}_1)^2]_r\mathrm{d}r
				+\int_{1}^{+\infty}\left(-\frac{\tau\epsilon}{3 \mu}r^2\partial_{tt}\rho \partial_{r}\tilde{S}_1\partial_{tt} \tilde{S}_1
				-\frac{2\tau\epsilon}{3 \mu}r^2 \partial_{t}\rho \partial_{tr}\tilde{S}_1\partial_{tt}\tilde{S}_1\right)\mathrm{d}r\\
				&=\int_{1}^{+\infty}\left(\frac{2}{3}r^2\partial_{tt} \tilde{S}_1\partial_{ttr}v
				-\frac{2}{3}r\partial_{tt} \tilde{S}_1\partial_{tt}v\right) \mathrm{d}r
				,
			\end{aligned}
		\end{equation}
		where
		\begin{equation}\nonumber
			\begin{aligned}
				&\int_{1}^{+\infty}\left(\frac{2\tau}{3 \mu}r^2\partial_{t}\rho(\partial_{tt} \tilde{S}_1)^2
				+\frac{\tau}{3 \mu}r^2\partial_{tt}\rho [\frac{1}{2}(\partial_{t} \tilde{S}_1)^2]_t\right)\mathrm{d}r
				\geq-\frac{C}{\tau}E(t)^{\frac{1}{2}} \mathcal{D}(t),\\
				&\int_{1}^{+\infty}\left(-\frac{\tau\epsilon}{3 \mu}r^2\partial_{tt}\rho \partial_{r}\tilde{S}_1\partial_{tt} \tilde{S}_1
				-\frac{2\tau\epsilon}{3 \mu}r^2 \partial_{t}\rho \partial_{tr}\tilde{S}_1\partial_{tt}\tilde{S}_1\right)\mathrm{d}r
				\geq -CE(t)^{\frac{1}{2}} \mathcal{D}(t),\\
				&\int_{1}^{+\infty}\bigg(\frac{2\tau}{3 \mu}r^2\rho\partial_{t}v\partial_{tr} \tilde{S}_1\partial_{tt} \tilde{S}_1+\frac{\tau}{3 \mu}r^2\rho\partial_{tt}v\partial_{r} \tilde{S}_1\partial_{tt} \tilde{S}_1
				+\frac{2\tau}{3 \mu}r^2\partial_{t}\rho\partial_{t}v\partial_{r} \tilde{S}_1\partial_{tt} \tilde{S}_1
				+\frac{2\tau}{3 \mu}r^2\partial_{t}\rho v \partial_{tr} \tilde{S}_1\partial_{tt} \tilde{S}_1\\
				&~~~~
				+\frac{\tau}{3 \mu}r^2\partial_{tt}\rho v\partial_{r} \tilde{S}_1\partial_{tt} \tilde{S}_1-\frac{\tau\epsilon}{3 \mu}r^2\partial_{tt}\rho \partial_{r}\tilde{S}_1\partial_{tt} \tilde{S}_1
				-\frac{2\tau\epsilon}{3 \mu}r^2\partial_{t}\rho\partial_{tr}\tilde{S}_1\partial_{tt}\tilde{S}_1\bigg)\mathrm{d}r
				\geq-CE(t)^{\frac{1}{2}} \mathcal{D}(t),
			\end{aligned}
		\end{equation}
		and
		\begin{equation}\nonumber
			\begin{aligned}
				\int_{1}^{+\infty}-\frac{\tau\epsilon}{3 \mu}r^2\rho[\frac{1}{2}(\partial_{tt}\tilde{S}_1)^2]_r \mathrm{d}r
				&=-\frac{\tau\epsilon}{3 \mu}r^2\rho\frac{1}{2}(\partial_{tt}\tilde{S}_1)^2\bigg|_{1}^{+\infty}
				+\int_{1}^{+\infty}\frac{\tau\epsilon}{3 \mu}(r^2\rho)_r\frac{1}{2}(\partial_{tt}\tilde{S}_1)^2\mathrm{d}r\\
				&=\frac{\tau\epsilon}{6\mu}\rho(t,1)(\partial_{tt}\tilde{S}_1)^2(t,1) +\int_{1}^{+\infty}\frac{\tau\epsilon}{6 \mu}(r^2\partial_{r}\rho+2r\rho)(\partial_{tt}\tilde{S}_1)^2\mathrm{d}r\\
				&\geq  \int_{1}^{+\infty}\frac{\tau\epsilon r^2}{6 \mu}\partial_{r}\rho(\partial_{tt}\tilde{S}_1)^2\mathrm{d}r
				\geq-\frac{C}{\tau}E(t)^{\frac{1}{2}} \mathcal{D}(t).
			\end{aligned}
		\end{equation}
		
		Therefore, we derive
		\begin{equation}
		\begin{aligned}\label{42aa}
			&\frac{\mathrm{d}}{\mathrm{d} t}\int_{1}^{+\infty}\frac{\tau r^2 \rho}{6\mu}(\partial_{tt} \tilde{S}_1)^2\mathrm{d}r
			+\int_{1}^{+\infty}\frac{r^2}{3 \mu}(\partial_{tt}\tilde{S}_1)^2\mathrm{d}r\\
			&\qquad
			\leq \int_{1}^{+\infty}\left(\frac{2}{3}r^2\partial_{tt} \tilde{S}_1\partial_{ttr}v
			-\frac{2}{3}r\partial_{tt} \tilde{S}_1\partial_{tt}v\right) \mathrm{d}r
			+\frac{C}{\tau}E(t)^{\frac{1}{2}} \mathcal{D}(t).
		\end{aligned}
		\end{equation}
		Multiplying the equation $(\ref{38})_4$ 
		by $\dfrac{r^2}{\lambda}\partial_{tt} \tilde{S}_2$ and integrating over $[1,+\infty)$ with respect to $r$ yields
		
			\begin{align}\label{43}
				&\frac{\mathrm{d}}{\mathrm{d} t}\int_{1}^{+\infty}\left[\frac{\tau }{\lambda} \cdot \frac{r^2 \rho}{2}(\partial_{tt} \tilde{S}_2)^2\right]\mathrm{d}r
				+\int_{1}^{+\infty}\frac{r^2}{\lambda}(\partial_{tt}{S_2})^2\mathrm{d}r \nonumber \\
				&+\int_{1}^{+\infty}\bigg(\frac{2\tau }{\lambda}r^2\partial_{t}\rho(\partial_{tt} \tilde{S}_2)^2
				+\frac{2\tau }{\lambda}r^2\rho\partial_{t}v\partial_{tr} \tilde{S}_2\partial_{tt} \tilde{S}_2
				+\frac{\tau }{\lambda}r^2\rho\partial_{tt}v\partial_{r} \tilde{S}_2\partial_{tt} \tilde{S}_2
				+\frac{2\tau }{\lambda}r^2\partial_{t}\rho\partial_{t}v\partial_{r} \tilde{S}_2\partial_{tt} \tilde{S}_2\nonumber \\
				&\qquad\qquad\qquad
				+\frac{2\tau}{\lambda}r^2\partial_{t}\rho v \partial_{tr} \tilde{S}_2\partial_{tt} \tilde{S}_2
				+\frac{\tau}{\lambda}r^2\partial_{tt}\rho[\frac{1}{2}(\partial_{t} \tilde{S}_2)^2]_t
				+\frac{\tau}{\lambda}r^2\partial_{tt}\rho v\partial_{r} \tilde{S}_2\partial_{tt} \tilde{S}_2\bigg)\mathrm{d}r\nonumber \\
				&+\int_{1}^{+\infty}
				-\frac{\tau \epsilon}{\lambda}r^2\rho [\frac{1}{2}(\partial_{tt}\tilde{S}_2)^2]_r\mathrm{d}r
				+\int_{1}^{+\infty}\left(-\frac{\tau \epsilon}{\lambda}r^2\partial_{tt}\rho \partial_{r}\tilde{S}_2\partial_{tt} \tilde{S}_2
				-\frac{2\tau\epsilon}{\lambda}r^2\partial_{t}\rho\partial_{tr}\tilde{S}_2\partial_{tt}\tilde{S}_2\right)\mathrm{d}r \nonumber \\
				\leq&\int_{1}^{+\infty}\left(r^2\partial_{tt} \tilde{S}_2\partial_{ttr}v
				+2r\partial_{tt} \tilde{S}_2\partial_{tt}v\right)\mathrm{d}r.
			\end{align}
	
		Similar to the method used for handling equation (\ref{42}), we derive
		\begin{equation}\label{43a}
			\begin{aligned}
			\frac{\mathrm{d}}{\mathrm{d} t}\int_{1}^{+\infty}\frac{\tau r^2 \rho}{2\lambda}(\partial_{tt} \tilde{S}_2)^2\mathrm{d}r
			+\int_{1}^{+\infty}\frac{r^2}{\lambda}(\partial_{tt}\tilde{S}_2)^2\mathrm{d}r
			\leq \int_{1}^{+\infty}\bigg(r^2\partial_{tt} \tilde{S}_2\partial_{ttr}v
			+2r\partial_{tt} \tilde{S}_2\partial_{tt}v\bigg)\mathrm{d}r\\
			+\frac{C}{\tau }E(t)^{\frac{1}{2}} \mathcal{D}(t).
			\end{aligned}
		\end{equation}
		
		By adding the equations $(\ref{39})$, $(\ref{42aa})$ and $(\ref{43a})$ and using $\partial_{tt}v\big|_{1}^{+\infty}=0$, we obtain
		
		\begin{equation}\label{42a}
			\begin{aligned}
				&\frac{\mathrm{d}}{\mathrm{d} t}\int_{1}^{+\infty}
				\left(\frac{r^2 P^{\prime}(\rho)}{2\rho} (\partial_{tt}\rho)^2
				+\frac{r^2\rho}{2}(\partial_{tt}v)^2
				+\frac{\tau r^2 \rho}{6\mu} (\partial_{tt}\tilde{S}_1)^2 +\frac{\tau r^2 \rho}{2\lambda}  (\partial_{tt}\tilde{S}_1)^2 \right)\mathrm{d}r\\
				&+\int_{1}^{+\infty} \left(\frac{r^2}{3 \mu}(\partial_{tt}\tilde{S}_1)^2
				+\frac{r^2}{\lambda}(\partial_{tt}\tilde{S}_2)^2\right)\mathrm{d}r \le \frac{C}{\tau} E^{\frac{1}{2}}(t) \mathcal{D}(t)
				.
			\end{aligned}
		\end{equation}

		%\begin{equation}\nonumber
		%	\begin{aligned}
			%		&r^2\partial_{tt}v\left[-\partial_{tt}P+\dfrac{2}{3}\partial_{tt}\tilde{S}_1+\partial_{tt}\tilde{S}_2\right]\bigg|_{1}^{+\infty}=0\\
			%		&=r^2v_{tt}\left[-\partial_{tt}P+\dfrac{2}{3}\partial_{tt}\tilde{S}_1+\partial_{tt}\tilde{S}_2+\dfrac{2}{r}\partial_{t}\tilde{S}_1\right]\bigg|_{1}^{+\infty}-
			%		\dfrac{2}{r}\partial_{t}\tilde{S}_1\cdot r^2 v_{tt} \bigg|_{1}^{+\infty}\\
			%		&=0-2r\partial_{t}\tilde{S}_1\cdot v_{tt} \bigg|_{1}^{+\infty}=0.\\
			%	\end{aligned}
		%\end{equation}
		%where used the  equation $(\ref{15})_2$. 
		%Therefore, by integrating equation $(\ref{34})$ over $(0,t)$, we derive
		
		Integrating the above inequality over $(0, t)$, and noticing that
		\begin{align*}\nonumber
			\int_{1}^{+\infty}\tau^2
			\left(\frac{r^2 P^{\prime}(\rho)}{2\rho} (\partial_{tt}\rho)^2
			+\frac{r^2\rho}{2}(\partial_{tt}v)^2
			+\frac{\tau r^2 \rho}{6\mu} (\partial_{tt}\tilde{S}_1)^2  
			+\frac{\tau r^2 \rho}{2\lambda}  (\partial_{tt}\tilde{S}_1)^2 \right)(t=0,r)\mathrm{d}r\\
			\leq C \|V_0^2\|_{L^2}^2\leq C E(0).
		\end{align*}
		Thus, we have
		\begin{equation}\label{order-2-tt}
			\begin{aligned}
				&\int_{1}^{+\infty} \tau^2
				\left(\frac{r^2 P^{\prime}(\rho)}{2\rho} (\partial_{tt}\rho)^2
				+\frac{r^2\rho}{2}(\partial_{tt}v)^2
				+\frac{\tau r^2 \rho}{6\mu} (\partial_{tt}\tilde{S}_1)^2 +\frac{\tau r^2 \rho}{2\lambda}  (\partial_{tt}\tilde{S}_1)^2 \right)\mathrm{d}r\\
				&+\int_{0}^{t}\int_{1}^{+\infty} \tau^2\left(\frac{r^2}{3 \mu}(\partial_{tt}\tilde{S}_1)^2
				+\frac{r^2}{\lambda}(\partial_{tt}\tilde{S}_2)^2\right)
				\mathrm{d}r\mathrm{d}t
				\leq C\left(E(0)+ E^{\frac{1}{2}}(t) \int_{0}^t \mathcal{D}(s)\mathrm{d}s\right).
			\end{aligned}
		\end{equation}
		This finish the proof of the Lemma $\ref{lem3}$.
	\end{proof}
	
	\begin{lemma}\label{lem4}
		There exists some constant $C$ such that for any $0\leq t \leq T$
		\begin{equation}
			\begin{aligned}
				&\int_{1}^{+\infty}r^2\bigg(\left(\partial_{tr} \rho\right)^2 + (\partial_{tr}v)^2+ (\partial_{t}v)^2
				+\tau (\partial_{tr}\tilde{S}_1)^2
				+\tau(\partial_{t}\tilde{S}_1)^2+\tau(\partial_{tr} \tilde{S}_2)^2\bigg)\mathrm{d}r\\
				&+\int_{0}^{t}\int_{1}^{+\infty}r^2\bigg((\partial_{tr}\tilde{S}_1)^2+ (\partial_{t}\tilde{S}_1)^2
				+(\partial_{tr}\tilde{S}_2)^2\bigg)\mathrm{d}r \mathrm{d} t\\
				&+\int_{0}^{t}\left(\frac{\tau\epsilon}{16\mu} (\partial_{tr}\tilde{S}_1)^2(t,1)
				+\frac{3\tau\epsilon}{8\lambda}(\partial_{tr}\tilde{S}_2)^2(t,1)\right)\mathrm{d}t\\
				&\leq C\left(E(0)+E^{\frac{3}{2}}(t)+ E^{\frac{1}{2}}(t)\int_{0}^{t} \mathcal{D}(s)\mathrm{d}s\right).
			\end{aligned}
		\end{equation}
	\end{lemma}
	\begin{proof}
		Taking the derivative of equation \eqref{3.4} with respect to $t$ and $r$ once, respectively, we get
		\begin{equation}\label{47}
			\begin{cases}
				\partial_{ttr} \rho+\partial_{trr}(\rho v)+\dfrac{2}{r} \partial_{tr}(\rho v)-\dfrac{2}{r^2} \partial_{t}(\rho v)=0, \\
				\rho\partial_{ttr} v+\rho v \partial_{trr} v+\partial_{trr} P-\dfrac{2}{3}\partial_{trr}  \tilde{S}_1-\dfrac{2}{r} \partial_{tr} \tilde{S}_1+\dfrac{2}{r^2}\partial_{t}\tilde{S}_1-\partial_{trr} \tilde{S}_2=g_2,\\
				\tau \rho \partial_{ttr} \tilde{S}_1+\tau \rho v \partial_{trr}\tilde{S}_1 
				-2\mu\left(\partial_{trr} v- \dfrac{1}{r}\partial_{tr} v
				+\dfrac{1}{r^2} \partial_{t}v \right)+\partial_{tr} \tilde{S}_1
				-\tau \epsilon\rho\partial_{trr}\tilde{S}_1
				= g_3, \\
				\tau  \rho \partial_{ttr} \tilde{S}_2+\tau  \rho v \partial_{trr}\tilde{S}_2
				-\lambda \left(\partial_{trr} v+\dfrac{2}{r}\partial_{tr} v-\dfrac{2}{r^2} \partial_{t}v\right)+\partial_{tr} \tilde{S}_2
				-\tau  \epsilon\rho\partial_{trr}\tilde{S}_2
				=g_4,
			\end{cases}
		\end{equation}
		where
		\begin{equation}
			\begin{aligned}\nonumber
				&g_2:=\partial_{t}f_2-\partial_{t}(\rho v)\partial_{rr} v,\\
				&g_3:=\partial_{t}f_3
				-\tau \partial_{t}\rho \partial_{tr} \tilde{S}_1
				-\tau \partial_{t}(\rho v) \partial_{rr}\tilde{S}_1
				+\tau \epsilon\partial_{t}\rho\partial_{rr}\tilde{S}_1,\\
				&g_4:=\partial_{t}f_4-\tau  \partial_{t}\rho \partial_{tr} \tilde{S}_2-\tau  \partial_{t}(\rho v) \partial_{rr}\tilde{S}_2
				+\tau  \epsilon\partial_{t}\rho\partial_{rr}\tilde{S}_2.
			\end{aligned}
		\end{equation}
		
		Multiplying the equation $(\ref{47})_1$ and $(\ref{47})_2$ 
		by  $r^2\frac{\partial_{tr}P}{\rho}$ and $r^2\partial_{tr}v$, respectively, and  integrating the result, we have
		\begin{equation}\label{48}
			\begin{aligned}
				&\frac{\mathrm{d}}{\mathrm{d} t}\int_{1}^{+\infty}\left(\dfrac{r^2 P^{\prime}(\rho)}{2\rho} \left(\partial_{tr} \rho\right)^2+\dfrac{r^2\rho}{2} (\partial_{tr}v)^2\right)\mathrm{d}r
				-\int_{1}^{+\infty}\frac{r^2}{2}[\dfrac{ P^{\prime}(\rho)}{\rho}]_t\left(\partial_{tr} \rho\right)^2\mathrm{d}r\\
				&+\frac{\mathrm{d}}{\mathrm{d} t}\int_{1}^{+\infty} \dfrac{r^2 P^{\prime\prime}(\rho)}{\rho} \partial_{t}\rho\partial_{r}\rho\partial_{tr}\rho\mathrm{~d}r
				-\int_{1}^{+\infty} [\dfrac{r^2 P^{\prime\prime}(\rho)}{\rho} \partial_{t}\rho\partial_{r}\rho]_t\partial_{tr}\rho\mathrm{~d}r\\
				&+\int_{1}^{\infty}\bigg(r^2\partial_{trr}\left(\rho v\right) \frac{\partial_{tr}P}{\rho}
				+2r \partial_{tr} (\rho v) \frac{\partial_{tr}P}{\rho}
				-2\partial_{t} (\rho v) \frac{\partial_{tr}P}{\rho}\bigg)\mathrm{d}r\\
				&+\int_{1}^{+\infty}\bigg(
				r^2 \partial_{trr}P \partial_{tr}v-\dfrac{2 r^2}{3} \partial_{trr} \tilde{S}_1\partial_{tr}v-2r \partial_{tr}\tilde{S}_1 \partial_{tr}v+2\partial_{t}\tilde{S}_1\partial_{tr}v
				-r^2\partial_{trr}\tilde{S}_2\partial_{tr}v\bigg)\mathrm{d}r\\
				&=\int_{1}^{+\infty}r^2g_2 \partial_{tr}v\mathrm{d}r.
			\end{aligned}
		\end{equation} 
		
		Firstly, we have
		\begin{equation}\nonumber
			\begin{aligned}
				-&\int_{1}^{+\infty}\frac{r^2}{2}[\dfrac{ P^{\prime}(\rho)}{\rho}]_t\left(\partial_{tr} \rho\right)^2\mathrm{d}r
				-\int_{1}^{+\infty} [\dfrac{r^2 P^{\prime\prime}(\rho)}{\rho} \partial_{t}\rho\partial_{r}\rho]_t\partial_{tr}\rho\mathrm{~d}r
				\geq -C E^{\frac{1}{2}}(t) \mathcal{D}(t),\\
				&\int_{1}^{+\infty}-2\partial_{t} (\rho v) \frac{\partial_{tr}P}{\rho}\mathrm{d}r
				\geq \int_{1}^{+\infty}-2\partial_{t}v\partial_{tr}P\mathrm{d}r
				-C E^{\frac{1}{2}}(t) \mathcal{D}(t).
			\end{aligned}
		\end{equation}
		Moreover, we have
		\begin{equation}\nonumber
			\begin{aligned}
				&\int_{1}^{\infty}\bigg(r^2\partial_{trr}\left(\rho v\right) \frac{\partial_{tr}P}{\rho}
				+2r \partial_{tr} (\rho v) \frac{\partial_{tr}P}{\rho}
				+r^2 \partial_{trr}P \partial_{tr}v\bigg)\mathrm{d}r\\
				=&\int_{1}^{\infty}\bigg(r^2\partial_{trr}v \partial_{tr}P+2r \partial_{tr} v \partial_{tr}P
				+r^2 \partial_{trr}P \partial_{tr}v\bigg)\mathrm{d}r\\
				&+\int_{1}^{\infty}\bigg(\frac{r^2\partial_{tr}P}{\rho}(2\partial_{r}\rho \partial_{tr} v+2\partial_{tr}\rho \partial_{r} v+\partial_{t}\rho \partial_{rr} v+\partial_{rr}\rho \partial_{t} v)\bigg)\mathrm{d}r
				-\int_{1}^{\infty}[\frac{r^2vP^{\prime}(\rho)}{2\rho}]_r(\partial_{tr}\rho)^2\mathrm{d}r\\
				&
				-\int_{1}^{\infty}r^2P^{\prime\prime}(\rho)[\frac{\partial_{t}\rho\partial_{r}\rho v}{\rho}]_t \partial_{rr}\rho\mathrm{d}r
				+\int_{1}^{\infty}2r \frac{\partial_{tr}P}{\rho}(\partial_{t}\rho \partial_{r} v+\partial_{r}\rho \partial_{t} v+v \partial_{tr}\rho)\mathrm{d} r\\
				\geq& \int_{1}^{\infty}\partial_{r}\left(r^2\partial_{tr}P\partial_{tr}v\right)\mathrm{d}r -C E^{\frac{1}{2}}(t) \mathcal{D}(t),
			\end{aligned}
		\end{equation}
		and
		\begin{equation}\nonumber
			\int_{1}^{+\infty}r^2g_2 \partial_{tr}v\mathrm{d}r\leq C E^{\frac{1}{2}}(t) \mathcal{D}(t).
		\end{equation}
		
		Thus, combining the above estimates, equation $(\ref{48})$ can be simplified to
		\begin{align}\label{49}
				\frac{\mathrm{d}}{\mathrm{d} t}&\int_{1}^{+\infty}\left(\dfrac{r^2 P^{\prime}(\rho)}{2\rho} \left(\partial_{tr} \rho\right)^2+\dfrac{r^2\rho}{2} (\partial_{tr}v)^2\right)\mathrm{d}r
				+\int_{1}^{+\infty}\partial_{r}\left(r^2\partial_{tr}P\partial_{tr}v\right)\mathrm{d}r
				-\int_{1}^{+\infty}2\partial_{t}v\partial_{tr}P\mathrm{d}r\nonumber \\
				+&\int_{1}^{+\infty}\bigg(-\dfrac{2 r^2}{3}\partial_{trr}\tilde{S}_1 \partial_{tr}v-2r \partial_{tr}\tilde{S}_1 \partial_{tr}v+2\partial_{t}\tilde{S}_1 \partial_{tr}v
				-r^2 \partial_{trr}\tilde{S}_2 \partial_{tr}v\bigg)\mathrm{d}r
				\nonumber \\
				\leq& C E^{\frac{1}{2}}(t) \mathcal{D}(t)
				-\frac{\mathrm{d}}{\mathrm{d} t}\int_{1}^{+\infty} \dfrac{r^2 P^{\prime\prime}(\rho)}{\rho} \partial_{t}\rho\partial_{r}\rho\partial_{tr}\rho\mathrm{~d}r.
		\end{align}
		
		Multiplying the equation $(\ref{47})_3$  by  $\dfrac{r^2}{3\mu}\partial_{tr} \tilde{S}_1$, and  integrating the result, we obtain
		\begin{equation}\label{50}
			\begin{aligned}
				&\frac{\mathrm{d}}{\mathrm{d} t}\int_{1}^{+\infty}\left[\frac{\tau}{3 \mu} \frac{r^2 \rho}{2}(\partial_{tr}\tilde{S}_1)^2\right]\mathrm{d}r
				+\int_{1}^{+\infty}\frac{r^2}{3 \mu}(\partial_{tr}\tilde{S}_1)^2 \mathrm{d}r
				-\int_{1}^{+\infty}
				\frac{\tau\epsilon}{3 \mu}r^2\rho[\frac{1}{2}(\partial_{tr}\tilde{S}_1)^2]_r\mathrm{d}r\\
				&-\int_{1}^{+\infty}\bigg(\frac{2r^2}{3}\partial_{trr}v\partial_{tr}\tilde{S}_1-\frac{2r}{3}\partial_{tr}v\partial_{tr}\tilde{S}_1
				+\frac{2}{3}\partial_{t}v\partial_{tr}\tilde{S}_1\bigg)\mathrm{d}r
				=\int_{1}^{+\infty}\frac{r^2}{3\mu}g_3\partial_{tr} \tilde{S}_1\mathrm{d}r,\\
			\end{aligned}
		\end{equation}
		where
		\begin{equation}\nonumber
			\begin{aligned}
				-\int_{1}^{+\infty}
				\frac{\tau\epsilon}{3 \mu}r^2\rho[\frac{1}{2}(\partial_{tr}\tilde{S}_1)^2]_r\mathrm{d}r
				&=\frac{\tau\epsilon}{6\mu}\rho(t,1)(\partial_{tr}\tilde{S}_1)^2(t,1)
				+\int_{1}^{+\infty}
				\frac{\tau\epsilon}{6\mu}(r^2\rho)_r(\partial_{tr}\tilde{S}_1)^2\mathrm{d}r\\
				&\geq\frac{\tau\epsilon}{6\mu}\rho(t,1)(\partial_{tr}\tilde{S}_1)^2(t,1) -C E^{\frac{1}{2}}(t) \mathcal{D}(t),\\
				\int_{1}^{+\infty}\frac{r^2}{3\mu}g_3\partial_{tr} \tilde{S}_1\mathrm{d}r
				&\leq C E^{\frac{1}{2}}(t) \mathcal{D}(t).\\
			\end{aligned}
		\end{equation}
		Similarly, multiplying the equation $(\ref{47})_4$ by  $\dfrac{r^2}{\lambda}\partial_{tr} \tilde{S}_2$, and then integrating the result, one get 
		\begin{equation}
			\begin{aligned}\label{51}
				&\frac{\mathrm{d}}{\mathrm{d} t}\int_{1}^{+\infty}\left[\frac{\tau }{\lambda} \frac{r^2 \rho}{2}(\partial_{tr}{S_2})^2\right]\mathrm{d}r
				+\int_{1}^{+\infty}\frac{r^2}{\lambda}(\partial_{tr}{S_2})^2
				\mathrm{d}r
				-\int_{1}^{+\infty}
				\frac{\tau \epsilon}{\lambda}r^2\rho[\frac{1}{2}(\partial_{tr}\tilde{S}_2)^2]_r\mathrm{d}r\\
				&-\int_{1}^{+\infty}\bigg(r^2\partial_{trr}v\partial_{tr}{S_2}
				-2r\partial_{tr}v\partial_{tr}{S_2}
				+2\partial_tv\partial_{tr}{S_2}\bigg)\mathrm{d}r
				=\int_{1}^{+\infty}\frac{r^2}{\lambda}g_4\partial_{tr} \tilde{S}_2\mathrm{d}r
				,
			\end{aligned}
		\end{equation}
		where
		\begin{equation}\nonumber
			\begin{aligned}
				-\int_{1}^{+\infty}
				\frac{\tau\epsilon}{\lambda}r^2\rho[\frac{1}{2}(\partial_{tr}\tilde{S}_2)^2]_r\mathrm{d}r
				&=\frac{\tau\epsilon}{2\lambda}\rho(t,1)(\partial_{tr}\tilde{S}_2)^2(t,1)
				+\int_{1}^{+\infty}
				\frac{\tau\epsilon}{2\lambda}(r^2\rho)_r(\partial_{tr}\tilde{S}_2)^2\mathrm{d}r\\
				&\geq\frac{\tau\epsilon}{2\lambda}\rho(t,1)(\partial_{tr}\tilde{S}_2)^2(t,1) -C E^{\frac{1}{2}}(t) \mathcal{D}(t)
				,\\
				\int_{1}^{+\infty}\frac{r^2}{\lambda}g_4\partial_{tr} \tilde{S}_2\mathrm{d}r
				&\leq C E^{\frac{1}{2}}(t) \mathcal{D}(t).\\
			\end{aligned}
		\end{equation}
		Adding equations $(\ref{49})-(\ref{51})$ together and combining the above estimates with \eqref{hu3.4},  we can obtain
		\begin{equation}\label{52}
			\begin{aligned}
				&\frac{\mathrm{d}}{\mathrm{d} t}\int_{1}^{+\infty}\bigg(\dfrac{r^2 P^{\prime}(\rho)}{2\rho} \left(\partial_{tr} \rho\right)^2 +\dfrac{r^2\rho}{2} (\partial_{tr}v)^2+\frac{\tau}{3 \mu} \frac{r^2 \rho}{2} (\partial_{tr}\tilde{S}_1)^2
				+\dfrac{\tau }{\lambda}\frac{r^2 \rho}{2} (\partial_{tr} \tilde{S}_2)^2\bigg)\mathrm{d}r\\
				&+\int_{1}^{+\infty}\bigg(\dfrac{r^2}{3\mu} (\partial_{tr}\tilde{S}_1)^2
				+\dfrac{r^2}{\lambda}(\partial_{tr}\tilde{S}_2)^2\bigg)\mathrm{d}r
				+\frac{\tau\epsilon}{8\mu} (\partial_{tr}\tilde{S}_1)^2(t,1)
				+\frac{3\tau\epsilon}{8\lambda}(\partial_{tr}\tilde{S}_2)^2(t,1)\\
				&+\int_{1}^{+\infty}\bigg(\partial_{r}\left(r^2\partial_{tr}P\partial_{tr}v\right)-\dfrac{2}{3}\partial_{r}(r^2\partial_{tr}\tilde{S}_1\partial_{tr}v) 
				-\partial_{r}(r^2\partial_{tr}\tilde{S}_2\partial_{tr}v)
				\bigg)\mathrm{d}r\\
				&+\int_{1}^{+\infty}\bigg(-2\partial_{t}v\partial_{tr}P+2\partial_{t}\tilde{S}_1\partial_{tr}v -\frac{2}{3}\partial_{t}v\partial_{tr}\tilde{S}_1+2\partial_{t}v\partial_{tr}\tilde{S}_2\bigg)\mathrm{d}r\\
				&\leq 
				C E(t)^{\frac{1}{2}}\mathcal{D}(t)
				-\frac{\mathrm{d}}{\mathrm{d} t}\int_{1}^{+\infty} \dfrac{r^2 P^{\prime\prime}(\rho)}{\rho} \partial_{t}\rho\partial_{r}\rho\partial_{tr}\rho\mathrm{~d}r.  
			\end{aligned}
		\end{equation}
		To handle the last term on the left-hand side of the above equation, by multiplying equation $(\ref{15})_2$ by $2\partial_{t}v$ and $(\ref{15})_3$ by $\dfrac{2}{\mu}\partial_{t}\tilde{S}_1$, respectively,  integrating the results, and adding the results to equation $(\ref{52})$, we get
		\begin{equation}
			\begin{aligned}\label{3.64}
				&\frac{\mathrm{d}}{\mathrm{d} t}\int_{1}^{+\infty}\bigg(\dfrac{r^2 P^{\prime}(\rho)}{2\rho} \left(\partial_{tr} \rho\right)^2 +\dfrac{r^2\rho}{2} (\partial_{tr}v)^2+\frac{\tau}{3 \mu} \frac{r^2 \rho}{2} (\partial_{tr}\tilde{S}_1)^2
				+\dfrac{\tau}{\lambda}\frac{r^2 \rho}{2} (\partial_{tr} \tilde{S}_2)^2\bigg)\mathrm{d}r\\
				&+\int_{1}^{+\infty}\bigg(\dfrac{r^2}{3\mu} (\partial_{tr}\tilde{S}_1)^2
				+\dfrac{r^2}{\lambda}(\partial_{tr}\tilde{S}_2)^2\bigg)\mathrm{d}r
				+\frac{\tau\epsilon}{8\mu} (\partial_{tr}\tilde{S}_1)^2(t,1)
				+\frac{3\tau\epsilon}{8\lambda}(\partial_{tr}\tilde{S}_2)^2(t,1)\\
				&+\int_{1}^{+\infty}\bigg(\partial_{r}\left(r^2\partial_{tr}P\partial_{tr}v\right)-\dfrac{2}{3}\partial_{r}(r^2\partial_{tr}\tilde{S}_1\partial_{tr}v) 
				-\partial_{r}(r^2\partial_{tr}\tilde{S}_2\partial_{tr}v)
				\bigg)\mathrm{d}r\\
				&+\int_{1}^{+\infty}\bigg(-2\partial_{t}v\partial_{tr}P+2\partial_{t}\tilde{S}_1\partial_{tr}v -\frac{2}{3}\partial_{t}v\partial_{tr}\tilde{S}_1+2\partial_{t}v\partial_{tr}\tilde{S}_2\bigg)\mathrm{d}r\\
				&+\int_{1}^{+\infty}\left(\rho\partial_{tt} v+\partial_{tr} P
				-(\dfrac{2}{3}\partial_{tr} \tilde{S}_1+\dfrac{2}{r} \partial_t\tilde{S}_1+\partial_{tr} \tilde{S}_2)\right) 2\partial_{t}v\mathrm{d}r
				\\
				&+\int_{1}^{+\infty}\left(\tau \rho \left(\partial_{tt} \tilde{S}_1+\partial_t ((v-\epsilon) \partial_r\tilde{S}_1) \right) +\partial_t \tilde{S}_1
				-2\mu(\partial_{tr} v- \dfrac{1}{r}\partial_t v )\right)\dfrac{2}{\mu}\partial_{t}\tilde{S}_1\mathrm{d}r
				\\
				&\leq 
				C E(t)^{\frac{1}{2}}\mathcal{D}(t)
				-\frac{\mathrm{d}}{\mathrm{d} t}\int_{1}^{+\infty} \dfrac{r^2 P^{\prime\prime}(\rho)}{\rho} \partial_{t}\rho\partial_{r}\rho\partial_{tr}\rho\mathrm{~d}r.
			\end{aligned}
		\end{equation}
		Notice that
		\begin{equation}\nonumber
			\begin{aligned}
				-\int_{1}^{+\infty}\frac{2\tau\epsilon}{\mu}\rho\partial_{t}\tilde{S}_1
				\partial_{tr}\tilde{S}_1\mathrm{d}r
				&=\frac{\tau\epsilon}{\mu}\rho(t,1)(\partial_{t}\tilde{S}_1)^2(t,1)
				+\int_{1}^{+\infty}\frac{\tau\epsilon}{\mu}\partial_r\rho(\partial_{t}\tilde{S}_1)^2
				\mathrm{d}r
				\geq %\frac{\tau\epsilon}{\mu}\rho(t,1)(\partial_{t}\tilde{S}_1)^2(t,1)
				-C E(t)^{\frac{1}{2}}\mathcal{D}(t),
			\end{aligned}
		\end{equation}
		and we integrate equation (\ref{3.64}) over $(0,t)$ to obtain
		\begin{equation}\nonumber
			\begin{aligned}
				&\int_{1}^{+\infty}\bigg(\dfrac{r^2 P^{\prime}(\rho)}{2\rho} \left(\partial_{tr} \rho\right)^2 +\dfrac{1}{2}\rho\left[r^2 (\partial_{tr}v)^2+2 (\partial_{t}v)^2\right]+ \left[\dfrac{\tau\rho}{6\mu} r^2(\partial_{tr}\tilde{S}_1)^2+\dfrac{\tau\rho}{\mu} \partial_{t}\tilde{S}_1^2\right]
				+\dfrac{\tau r^2\rho}{2\lambda} (\partial_{tr} \tilde{S}_2)^2\bigg)\mathrm{d}r\\
				&\quad+\int_{0}^{t}\int_{1}^{+\infty}\bigg(\left[\dfrac{r^2}{3\mu}(\partial_{tr}\tilde{S}_1)^2+\dfrac{2}{\mu} \partial_{t}\tilde{S}_1^2\right]
				+\dfrac{r^2}{\lambda}(\partial_{tr}\tilde{S}_2)^2\bigg)\mathrm{d}r\mathrm{d}t\\
				&\quad+\int_{0}^{t}\left(\frac{\tau\epsilon}{8\mu}(\partial_{tr}\tilde{S}_1)^2(t,1)
				+\frac{3\tau\epsilon}{8\lambda}(\partial_{tr}\tilde{S}_2)^2(t,1)\right)\mathrm{d}t\\
				&\leq 
				\int_{0}^{t}-\left(r^2\partial_{tr}v \partial_{tr}P
				-\dfrac{2}{3} r^2  \partial_{tr}v \partial_{tr}\tilde{S}_1
				-r^2  \partial_{tr}v \partial_{tr}\tilde{S}_2
				-2\partial_{t}v\partial_{t}\tilde{S}_1\right)\bigg|_{1}^{+\infty} \mathrm{d}t\\
				&\quad+C\left(E(0)+E(t)^{\frac{3}{2}}+ E^{\frac{1}{2}}(t) \int_{0}^t \mathcal{D}(s)\mathrm{d}s\right)
				.
			\end{aligned}
		\end{equation}
		
		Noting that $2\partial_{t}v\partial_{t}\tilde{S}_1\big|_{1}^{+\infty}=0$, and using momentum equation \eqref{22}, one has
		
		\begin{equation}
			\begin{aligned}
				&r^2  \partial_{tr}v\left[-\partial_{tr}P+\frac{2}{3}\partial_{tr}\tilde{S}_1
				+\partial_{tr}\tilde{S}_2\right]\bigg|_{1}^{+\infty}\\
				&=r^2  \partial_{tr}v\left[-\partial_{tr}P+\frac{2}{3}\partial_{tr}\tilde{S}_1+\frac{2}{r}\partial_{t}\tilde{S}_1
				+\partial_{tr}\tilde{S}_2\right]\bigg|_{1}^{+\infty}-2r\partial_{t}\tilde{S}_1\partial_{tr}v\bigg|_{1}^{+\infty}\\
				&=r^2  \partial_{tr}v[\rho\partial_{tt} v+\partial_t \rho  \partial_t v+\partial_t \rho \cdot v \partial_r v+\rho \partial_t( v \partial_r v)]\bigg|_{1}^{+\infty}
				-2r\partial_{t}\tilde{S}_1\partial_{tr}v\bigg|_{1}^{+\infty}\\
				&=2(\partial_{t}\tilde{S}_1\partial_{tr}v)(t,1).
			\end{aligned}
		\end{equation}
		Furthermore, using equation $(\ref{15})_3$, we have
		\begin{equation}
			\begin{aligned}\nonumber
				&\int_{0}^{t}2(\partial_{t}\tilde{S}_1\partial_{tr}v)(t,1)\mathrm{d}t\\
				=&\int_{0}^{t} 2\partial_{t}\tilde{S}_1(t,1)
				\bigg[\frac{\tau\rho}{2\mu} \left(\partial_{tt} \tilde{S}_1+\partial_t ((v-\epsilon)\partial_r\tilde{S}_1) \right) +\frac{\tau\partial_{t}\rho}{2\mu} \left(\partial_t \tilde{S}_1+(v-\epsilon)\partial_r\tilde{S}_1 \right) \\
				&\quad+\frac{1}{2\mu}\partial_t \tilde{S}_1+\dfrac{1}{r}\partial_t v\bigg](t,1)
				\mathrm{d}t\\
				=&\frac{\tau}{2\mu}\rho(\partial_{t}\tilde{S}_1)^2(t,1)\big|_0^t
				+\int_{0}^{t}\frac{\tau}{2\mu}\partial_{t}\rho(\partial_{t}\tilde{S}_1)^2(t,1)\mathrm{d}t
				-\int_{0}^{t}\frac{\epsilon\tau}{\mu}\partial_{t}\rho\partial_{t}\tilde{S}_1\partial_{r}\tilde{S}_1(t,1)\mathrm{d}t\\
				&\quad-\int_{0}^{t}\frac{\epsilon\tau}{\mu}\rho\partial_{t}\tilde{S}_1\partial_{tr}\tilde{S}_1(t,1)\mathrm{d}t
				+\int_{0}^{t}\frac{1}{\mu}(\partial_{t}\tilde{S}_1)^2(t,1)\mathrm{d}t.
			\end{aligned}
		\end{equation}
		Now, we deal with each term on the right hand side of the above equation. Firstly, we have 
		\begin{equation}\nonumber
			\begin{aligned}
				&\int_{0}^{t}\frac{\tau}{2\mu}\partial_{t}\rho(\partial_{t}\tilde{S}_1)^2(t,1)\mathrm{d}t
				-\int_{0}^{t}\frac{\epsilon\tau}{\mu}\partial_{t}\rho\partial_{t}\tilde{S}_1\partial_{r}\tilde{S}_1(t,1)\mathrm{d}t
				\\
				\leq&\int_{0}^{t}\frac{\tau}{2\mu}\|\partial_{t}\rho(\partial_{t}\tilde{S}_1)^2\|_{L^\infty}\mathrm{d}t
				+\int_{0}^{t}\frac{\epsilon\tau}{\mu}\|\partial_{t}\rho\partial_{t}\tilde{S}_1\partial_{r}\tilde{S}_1\|_{L^\infty}\mathrm{d}t\\
				\leq&\int_{0}^{t}\frac{\tau}{2\mu}\|\partial_{t}\rho\|_{H^1}\|\partial_{t}\tilde{S}_1\|_{H^1}^2\mathrm{d}t
				+\int_{0}^{t}\frac{\epsilon\tau}{\mu}\|\partial_{t}\rho\|_{H^1}\|\partial_{t}\tilde{S}_1\|_{H^1}\|\partial_{r}\tilde{S}_1\|_{H^1}\mathrm{d}t\\
				\leq&C\left(E(0)+E^{\frac{3}{2}}(t)+E^{\frac{1}{2}}(t) \int_{0}^t \mathcal{D}(s)\mathrm{d}s\right).
			\end{aligned}
		\end{equation}
		Secondly, we know
		\begin{equation}\nonumber
			\begin{aligned}
				&-\int_{0}^{t}\frac{\epsilon\tau}{\mu}\rho\partial_{t}\tilde{S}_1\partial_{tr}\tilde{S}_1(t,1)\mathrm{d}t\\
				\leq&\frac{\epsilon\tau}{\mu}\int_{0}^{t}\rho(t,1) \left(C(\eta)(\partial_{t}\tilde{S}_1)^2(t,1)+\eta(\partial_{tr}\tilde{S}_1)^2(t,1)\right)\mathrm{d}t\\
				\leq&\frac{\epsilon\tau}{\mu}\eta\int_{0}^{t}\rho(t,1)(\partial_{tr}\tilde{S}_1)^2(t,1)\mathrm{d}t
				+C\left(E(0)+E^{\frac{1}{2}}(t) \int_{0}^t \mathcal{D}(s)\mathrm{d}s\right),
			\end{aligned}
		\end{equation}
		and
		\begin{equation}
			\begin{aligned}\nonumber
				&\int_{0}^{t}\dfrac{1}{\mu} (\partial_{t}\tilde{S}_1)^2(t,1)\mathrm{d}t
				\leq \int_{0}^{t}(\dfrac{1}{\mu}\|\partial_{t}\tilde{S}_1\|_{L^\infty}^2)\mathrm{d}t
				\leq \int_{0}^{t}\dfrac{1}{\mu}\|\partial_{tr}\tilde{S}_1\|_{L^2}\|\partial_{t}\tilde{S}_1\|_{L^2}\mathrm{d}t\\
				&\leq \int_{0}^{t}\dfrac{1}{\mu} (\eta\|\partial_{tr}\tilde{S}_1\|_{L^2}^2+C(\eta)\|\partial_{t}\tilde{S}_1\|_{L^2}^2)\mathrm{d}t
				\leq \int_{0}^{t}\dfrac{\eta}{\mu}\|\partial_{tr}\tilde{S}_1\|_{L^2}^2\mathrm{d}t+C\left(E(0)+E^{\frac{1}{2}}(t) \int_{0}^t \mathcal{D}(s)\mathrm{d}s\right).
			\end{aligned}
		\end{equation}
		Here and after, $\eta$ is any positive constant to be chosen later.

		Combining the above results, one can derive
		\begin{equation}\label{57}
			\begin{aligned}
				&\int_{1}^{+\infty}\bigg(\dfrac{r^2 P^{\prime}(\rho)}{2\rho} \left(\partial_{tr} \rho\right)^2 +\dfrac{1}{2}\rho\left[r^2 (\partial_{tr}v)^2+2 (\partial_{t}v)^2\right]
				+\dfrac{\tau\rho r^2}{6\mu}  (\partial_{tr}\tilde{S}_1)^2
				+\dfrac{\tau\rho }{\mu} (\partial_{t}\tilde{S}_1)^2
				+\dfrac{\tau\rho r^2}{2\lambda} (\partial_{tr} \tilde{S}_2)^2\bigg)\mathrm{d}r\\
				&+\int_{0}^{t}\left(\frac{\tau\epsilon}{16\mu} (\partial_{tr}\tilde{S}_1)^2(t,1)
				+\frac{3\tau\epsilon}{8\lambda}\rho(t,1)(\partial_{tr}\tilde{S}_2)^2(t,1)\right)\mathrm{d}t\\
				&+\int_{0}^{t}\int_{1}^{+\infty}\bigg(\dfrac{r^2}{4\mu}(\partial_{tr}\tilde{S}_1)^2+\frac{2}{\mu} (\partial_{t}\tilde{S}_1)^2
				+\dfrac{r^2}{\lambda}(\partial_{tr}\tilde{S}_2)^2\bigg)\mathrm{d}r\\
				&\leq \frac{\tau}{2\mu}\rho(\partial_{t}\tilde{S}_1)^2(t,1)+C\left(E(0)+E^{\frac{3}{2}}(t)+ E^{\frac{1}{2}}(t)\int_{0}^{t} \mathcal{D}(s)\mathrm{d}s\right).
			\end{aligned}
		\end{equation}
		By applying the {\it Gagliardo-Nirenberg} inequality and {\it Young} inequality, we know
		\begin{equation}
			\begin{aligned}\nonumber
				&\frac{\tau}{2\mu}\rho(\partial_{t}\tilde{S}_1)^2(t,1)
				\leq \frac{\tau}{\mu}r^2(\partial_{t}\tilde{S}_1)^2(t,1)
				\leq \frac{\tau}{\mu}r^2\|\partial_{t}\tilde{S}_1\|_{L^{\infty}}^2
				\leq\frac{\tau}{\mu}r^2\|\partial_{tr}\tilde{S}_1\|_{L^{2}}\|\partial_{t}\tilde{S}_1\|_{L^{2}}\\
				&\leq \frac{\tau}{\mu}r^2(\eta\|\partial_{tr}\tilde{S}_1\|_{L^{2}}^2+C(\eta)\|\partial_{t}\tilde{S}_1\|_{L^{2}}^2)
				\leq \frac{\tau\eta}{\mu} r^2\|\partial_{tr}\tilde{S}_1\|_{L^{2}}^2
				+C\left(E(0)+E^{\frac{1}{2}}(t) \int_{0}^t \mathcal{D}(s)\mathrm{d}s\right).
			\end{aligned}
		\end{equation}
		Choosing $\eta<\frac{1}{12}$, 
		we get the Lemma \ref{lem4} immediately.
	\end{proof}
	
	Next, before estimating the second-order partial derivative with respect to $r$, we first derive the second-order dissipation estimates for $(\rho, v)$.
	
	\begin{lemma}\label{lem5}
		There exists a constant $C$ such that
		\begin{equation}
			\begin{aligned}
				&\int_0^t\int_{1}^{+\infty}r^2\left((\partial_{tt}v)^2+(\partial_{tr}v)^2+(\partial_{tt}\rho)^2+(\partial_{tr}\rho)^2\right)\mathrm{d}r\mathrm{d}t
				\leq C\left(E(0)+E^{\frac{3}{2}}(t)+E^{\frac{1}{2}}(t)\int_{0}^{t} \mathcal{D}(s)\mathrm{d}s\right)
			\end{aligned}
		\end{equation}
		and
		\begin{equation}
			\begin{aligned}
				&\int_0^t\int_{1}^{+\infty}r^2(\partial_{rr}v)^2\mathrm{d}r\mathrm{d}t\\
				&\leq
				C\left(\epsilon\int_{0}^{t}\int_{1}^{+\infty}r^2((\partial_{rr}\tilde{S}_1)^2+(\partial_{rr}\tilde{S}_2)^2)\mathrm{d}r\mathrm{d}t+E(0)+E^{\frac{3}{2}}(t)+E^{\frac{1}{2}}(t)\int_{0}^{t} \mathcal{D}(s)\mathrm{d}s\right).
			\end{aligned}
		\end{equation}
	\end{lemma}
	\begin{proof}
		We manipulate $\dfrac{\lambda}{\mu}(\ref{15})_{3}+(\ref{15})_{4}$
		and by applying Lemmas \ref{lem1}-\ref{lem3}, we can obtain
		\begin{equation}\label{2dis1}
			\begin{aligned}
				\int_0^t\int_{1}^{+\infty}r^2(\partial_{tr}v)^2
				\mathrm{d}r\mathrm{d}t
				\leq&\int_0^t\int_{1}^{+\infty}\tau^2r^2((\partial_{tt}\tilde{S}_1)^2+\epsilon(\partial_{tr}\tilde{S}_1)^2+(\partial_{t}\tilde{S}_1)^2+\epsilon(\partial_{r}\tilde{S}_1)^2)\mathrm{d}r\mathrm{d}t\\
				&+\int_0^t\int_{1}^{+\infty}\tau^2r^2((\partial_{tt}\tilde{S}_2)^2+\epsilon(\partial_{tr}\tilde{S}_2)^2+(\partial_{t}\tilde{S}_2)^2+\epsilon(\partial_{r}\tilde{S}_2)^2)\mathrm{d}r\mathrm{d}t\\
				\leq& C(E(0)+E^{\frac{3}{2}}(t)+E^{\frac{1}{2}}(t)\int_{0}^{t} \mathcal{D}(s)\mathrm{d}s).
			\end{aligned}
		\end{equation}
		%	Similarly, by computing $-\dfrac{\lambda}{2\mu}(\ref{15})_{3}+(\ref{15})_{4}$, we can get
		%	\begin{equation}
			%		\int_0^t\int_{1}^{+\infty}r^2(\partial_{t}v)^2
			%		\mathrm{d}r\mathrm{d}t
			%		\leq C(E(0)+E^{\frac{1}{2}}(t)\int_{0}^{t} \mathcal{D}(s)\mathrm{d}s).
			%	\end{equation}
		
		From equation $(\ref{15})_{1}$, we get immediately
		\begin{equation}\label{2dis2}
			\int_0^t\int_{1}^{+\infty}r^2(\partial_{tt}\rho)^2
			\mathrm{d}r\mathrm{d}t
			\leq C(E(0)+E^{\frac{3}{2}}(t)+E^{\frac{1}{2}}(t)\int_{0}^{t} \mathcal{D}(s)\mathrm{d}s).
		\end{equation}
		On the other hand, by multiplying the equation
		$(\ref{15})_{2}$ by $r^2\partial_{tt}v$, it yields
		\begin{equation}\nonumber
			\begin{aligned}
				&\int_0^t\int_{1}^{+\infty}r^2\rho(\partial_{tt}v)^2\mathrm{d}r\mathrm{d}t\\
				\leq&-\int_0^t\frac{\mathrm{d}}{\mathrm{d}t}\int_{1}^{+\infty}r^2\partial_{tr}P\partial_{t}v\mathrm{d}r\mathrm{d}t
				-\int_0^t\int_{1}^{+\infty}r^2\partial_{tt}P\partial_{tr}v\mathrm{d}r\mathrm{d}t
				-\int_0^t\int_{1}^{+\infty}2r\partial_{tt}P\partial_{t}v\mathrm{d}r\mathrm{d}t\\
				&+\int_0^t\int_{1}^{+\infty}\frac{2r^2}{3}\partial_{tr}\tilde{S}_1\partial_{tt}v\mathrm{d}r\mathrm{d}t
				+\int_0^t\int_{1}^{+\infty}2r\partial_{t}\tilde{S}_1\partial_{tt}v\mathrm{d}r\mathrm{d}t
				+\int_0^t\int_{1}^{+\infty}r^2\partial_{tr}\tilde{S}_2\partial_{tt}v\mathrm{d}r\mathrm{d}t\\
				&+C(E(0)+E^{\frac{1}{2}}(t)\int_{0}^{t} \mathcal{D}(s)\mathrm{d}s)\\
				\leq&\frac{1}{2}\int_0^t\int_{1}^{+\infty}r^2(\partial_{tt}v)^2\mathrm{d}r\mathrm{d}t
				+C\int_{1}^{+\infty}r^2((\partial_{tr}\rho)^2+(\partial_{t}v)^2)\mathrm{d}r\\
				&+C\int_0^t\int_{1}^{+\infty}r^2((\partial_{tt}\rho)^2
				+(\partial_{tr}v)^2+(\partial_{t}v)^2
				+(\partial_{tr}\tilde{S}_1)^2+(\partial_{t}\tilde{S}_1)^2+(\partial_{tr}\tilde{S}_2)^2)
				\mathrm{d}r\mathrm{d}t\\
				&+C(E(0)+E^{\frac{1}{2}}(t)\int_{0}^{t} \mathcal{D}(s)\mathrm{d}s)
			\end{aligned}
		\end{equation}
		which implies
		\begin{equation}\label{2dis3}
			\int_0^t\int_{1}^{+\infty}r^2(\partial_{tt}v)^2
			\mathrm{d}r\mathrm{d}t
			\leq C(E(0)+E^{\frac{3}{2}}(t)+E^{\frac{1}{2}}(t)\int_{0}^{t} \mathcal{D}(s)\mathrm{d}s).
		\end{equation}
		By equation $(\ref{15})_2$, we also get
		\begin{equation}\label{2dis4}
			\int_0^t\int_{1}^{+\infty}r^2(\partial_{tr}\rho)^2
			\mathrm{d}r\mathrm{d}t
			\leq C(E(0)+E^{\frac{3}{2}}(t)+E^{\frac{1}{2}}(t)\int_{0}^{t} \mathcal{D}(s)\mathrm{d}s).
		\end{equation}
		
		Furthermore, by manipulating $\dfrac{\lambda}{\mu}(\ref{22})_{3}+(\ref{22})_{4}$, we can obtain
		
		\begin{equation*}%\label{2dis5}
			\begin{aligned}
				\int_0^t\int_{1}^{+\infty}r^2(\partial_{rr}v)^2
				\mathrm{d}r\mathrm{d}t
				\leq&\int_0^t\int_{1}^{+\infty}\tau^2r^2((\partial_{tr}\tilde{S}_1)^2+\epsilon(\partial_{rr}\tilde{S}_1)^2+(\partial_{t}\tilde{S}_1)^2+\epsilon(\partial_{r}\tilde{S}_1)^2)\mathrm{d}r\mathrm{d}t\\
				&+\int_0^t\int_{1}^{+\infty}\tau^2r^2((\partial_{tr}\tilde{S}_2)^2+\epsilon(\partial_{rr}\tilde{S}_2)^2+(\partial_{t}\tilde{S}_2)^2+\epsilon(\partial_{r}\tilde{S}_2)^2)\mathrm{d}r\mathrm{d}t\\
				\leq& C\epsilon\int_{0}^{t}\int_{1}^{+\infty}r^2((\partial_{rr}\tilde{S}_1)^2+(\partial_{rr}\tilde{S}_2)^2)\mathrm{d}r\mathrm{d}t
				+C\bigg(E(0)+E^{\frac{3}{2}}(t)\\
				&\qquad\qquad\qquad\qquad\qquad\qquad\qquad\qquad\qquad\qquad+E^{\frac{1}{2}}(t)\int_{0}^{t} \mathcal{D}(s)\mathrm{d}s\bigg).
			\end{aligned}
		\end{equation*}
		
		Thus,  we get the desired result.
	\end{proof}
	
	\begin{lemma}\label{lem8}
		There exists some constant C such that for any $0\leq t\leq T$
		\begin{equation}\label{sed-rr}
			\begin{aligned}
				&\int_{1}^{+\infty}\bigg(\left(r\partial_{rr} \rho+r\partial_{r} \rho\right)^2
				+r^2 (\partial_{rr}v)^2+(\partial_{r}v)^2
				+\tau\left(r^2(\partial_{rr}\tilde{S}_1)^2+(\partial_{r}\tilde{S}_1)^2+\frac{1}{ r^2}(\tilde{S}_1)^2\right)\\
				&+\tau\left(r\partial_{rr} \tilde{S}_2+2\partial_{r} \tilde{S}_2\right)^2\bigg)\mathrm{d}r\\
				&+\int_0^t\int_{1}^{+\infty}\bigg(\left(r^2(\partial_{rr}\tilde{S}_1)^2+(\partial_{r}\tilde{S}_1)^2+\frac{1}{ r^2}(\tilde{S}_1)^2\right)
				+\left(r\partial_{rr} \tilde{S}_2+2\partial_{r} \tilde{S}_2\right)^2\bigg)\mathrm{d}r\mathrm{d}t\\
				&
				+\int_{0}^{t}\left(\frac{\tau\epsilon}{4\mu}(\partial_{rr}\tilde{S}_1)^2(t,1)
				+\frac{3\tau\epsilon}{8\lambda}(\partial_{rr}\tilde{S}_2)^2(t,1)\right)\mathrm{d}t\\
				&\leq C \left(E(0)+E^{\frac{3}{2}}(t)+E^{\frac{1}{2}}(t) \int_{0}^{t}\mathcal{D}(s)\mathrm{d}s
				+\int_{0}^{t}((\partial_{rr}v)^2(t,1)+(\partial_{tr}v)^2(t,1))\mathrm{d}t\right).
				\\ 
			\end{aligned}
		\end{equation}
	\end{lemma}
	\begin{proof}
		Taking derivative to the equations $(\ref{3.4})$ with respect to $r$ twice, we get
		\begin{equation}\label{60}
			\begin{cases}
				\partial_{trr} \rho+\partial_{rrr}(\rho v)+\dfrac{2}{r} \partial_{rr}(\rho v)-\dfrac{4}{r^2} \partial_{r}(\rho v)+\dfrac{4}{r^3}\rho v=0, \\
				\rho\partial_{trr} v+\rho v \partial_{rrr} v+\partial_{rrr} P-\dfrac{2}{3}\partial_{rrr}  \tilde{S}_1-\dfrac{2}{r} \partial_{rr} \tilde{S}_1+\dfrac{4}{r^2}\partial_{r}\tilde{S}_1-\dfrac{4}{r^3}\tilde{S}_1-\partial_{rrr} \tilde{S}_2=l_2,\\
				\tau \rho \partial_{trr} \tilde{S}_1+\tau \rho v \partial_{rrr}\tilde{S}_1 
				-2\mu\left(\partial_{rrr} v- \dfrac{1}{r}\partial_{rr} v
				+\dfrac{2}{r^2} \partial_{r}v-\dfrac{2}{r^3}v \right)+\partial_{rr} \tilde{S}_1-\tau \epsilon\rho\partial_{rrr}\tilde{S}_1
				=l_3, \\
				\tau  \rho \partial_{trr} \tilde{S}_2+\tau  \rho v \partial_{rrr}\tilde{S}_2
				-\lambda \left(\partial_{rrr} v+\dfrac{2}{r}\partial_{rr} v-\dfrac{4}{r^2} \partial_{r}v+\dfrac{4}{r^3}v\right)+\partial_{rr} \tilde{S}_2 -\tau  \epsilon\rho\partial_{rrr}\tilde{S}_2
				=l_4,
			\end{cases}
		\end{equation}
		where
		\begin{equation}
			\begin{aligned}\nonumber
				&l_2:=\partial_{r}f_2-\partial_{r}(\rho v)\partial_{rr} v,\\
				&l_3:=\partial_{r}f_3
				-\tau \partial_{r}\rho \partial_{tr} \tilde{S}_1
				-\tau \partial_{r}(\rho v) \partial_{rr}\tilde{S}_1
				+\tau \epsilon\partial_{r}\rho\partial_{rr}\tilde{S}_1,\\
				&l_4:=\partial_{r}f_4-\tau  \partial_{r}\rho \partial_{tr} \tilde{S}_2-\tau  \partial_{r}(\rho v) \partial_{rr}\tilde{S}_2
				+\tau  \epsilon\partial_{r}\rho\partial_{rr}\tilde{S}_2.
			\end{aligned}
		\end{equation}
		
		Multiplying equations $(\ref{60})_1$ and $(\ref{60})_2$ 
		by  $r^2\dfrac{\partial_{rr}P}{\rho}$ and $r^2\partial_{rr}v$, respectively, and  integrating the result, we get
		\begin{align}\label{3.66}
				&\frac{\mathrm{d}}{\mathrm{d} t}\int_{1}^{+\infty}\left(\frac{r^2 P^{\prime}(\rho)}{2\rho} \left(\partial_{rr} \rho\right)^2+\frac{r^2\rho}{2} (\partial_{rr}v)^2\right)\mathrm{d}r
				-\int_{1}^{+\infty}\frac{r^2}{2}[\frac{P^{\prime}(\rho)}{\rho}]_t\left(\partial_{rr} \rho\right)^2\mathrm{d}r \nonumber\\
				&+\frac{\mathrm{d}}{\mathrm{d} t}\int_{1}^{+\infty}\frac{r^2 P^{\prime\prime}(\rho)}{\rho} \partial_{rr} \rho(\partial_{r}\rho)^2\mathrm{d}r -\int_{1}^{+\infty}r^2[\frac{P^{\prime\prime}(\rho)}{\rho}(\partial_r\rho)^2]_t\partial_{rr}\rho\mathrm{d}r \nonumber\\
				&+\int_{1}^{+\infty}\bigg(r^2\partial_{rrr}\left(\rho v\right)  \frac{\partial_{rr}P}{\rho}+2r \partial_{rr} (\rho v) \frac{\partial_{rr}P}{\rho} 
				-4\partial_{r} (\rho v) \frac{\partial_{rr}P}{\rho}+\frac{4 v\partial_{rr}P}{r}\bigg)\mathrm{d}r \nonumber\\
				&+\int_{1}^{+\infty}\bigg(
				r^2 \partial_{rrr}P \partial_{rr}v-\frac{2 r^2}{3} \partial_{rrr}\tilde{S}_1 \partial_{rr}v-2r\partial_{rr}\tilde{S}_1\partial_{rr}v+4\partial_{r}\tilde{S}_1  \partial_{rr}v-\frac{4}{r}\tilde{S}_1\partial_{rr}v
				-r^2 \partial_{rrr}\tilde{S}_2\partial_{rr}v\bigg)\mathrm{d}r \nonumber\\
				&=\int_{1}^{+\infty}r^2l_2 \partial_{rr}v\mathrm{d}r.
		\end{align}
		Firstly, we have
		\begin{equation}\nonumber
			\begin{aligned}
				&-\int_{1}^{+\infty}\frac{r^2}{2}[\frac{P^{\prime}(\rho)}{\rho}]_t\left(\partial_{rr} \rho\right)^2\mathrm{d}r -\int_{1}^{+\infty}r^2[\frac{P^{\prime\prime}(\rho)}{\rho}(\partial_r\rho)^2]_t\partial_{rr}\rho\mathrm{d}r
				\geq -C E(t)^{\frac{1}{2}}\mathcal{D}(t),
			\end{aligned}
		\end{equation}
		and
		\begin{equation}\nonumber
			\begin{aligned}
				&\int_{1}^{+\infty}\bigg(r^2\partial_{rrr}\left(\rho v\right)  \frac{\partial_{rr}P}{\rho}+2r \partial_{rr} (\rho v) \frac{\partial_{rr}P}{\rho}+r^2 \partial_{rrr}P \partial_{rr}v 
				\bigg)\mathrm{d}r\\
				&=\int_{1}^{+\infty}(r^2 \partial_{rrr}v\partial_{rr}P+2r \partial_{rr}v\partial_{rr}P+r^2 \partial_{rrr}P \partial_{rr}v)\mathrm{d}r\\
				&~~~~+\int_{1}^{+\infty}r^2(3\partial_{r}\rho\partial_{rr}v+3\partial_{rr}\rho\partial_{r}v+v\partial_{rrr}\rho)\frac{\partial_{rr}P}{\rho} \mathrm{d}r+\int_{1}^{+\infty}2r(v\partial_{rr}\rho+2\partial_{r}\rho\partial_{r}v)\frac{\partial_{rr}P}{\rho} \mathrm{d}r\\
				&\geq \int_{1}^{+\infty}\partial_{r}(r^2 \partial_{rr}v\partial_{rr}P)\mathrm{d}r
				-C E(t)^{\frac{1}{2}}\mathcal{D}(t)
			\end{aligned}
		\end{equation}
		where we used
		\begin{equation}\nonumber
			\begin{aligned}
				\int_{1}^{+\infty}r^2v\partial_{rrr}\rho\frac{\partial_{rr}P}{\rho} \mathrm{d}r&=
				\int_{1}^{+\infty}\frac{P^\prime(\rho)}{\rho}r^2v((\partial_{rr}\rho)^2)_r
				+\int_{1}^{+\infty}\frac{P^{\prime\prime}(\rho)(\partial_{r}\rho)^2}{\rho}r^2v\partial_{rrr}\rho\\
				&=-\int_{1}^{+\infty}\frac{P^\prime(\rho)}{\rho}(r^2v)_r(\partial_{rr}\rho)^2
				-\int_{1}^{+\infty}(\frac{P^{\prime\prime}(\rho)(\partial_{r}\rho)^2}{\rho}r^2v)_r\partial_{rr}\rho\\
				&\geq-C E(t)^{\frac{1}{2}}\mathcal{D}(t).
			\end{aligned}
		\end{equation}
		Moreover, we know
		\begin{equation}\nonumber
			\begin{aligned}
				-\int_{1}^{+\infty}4\partial_{r} (\rho v) \frac{\partial_{rr}P}{\rho}\mathrm{d}r
				\geq -\int_{1}^{+\infty}4\partial_{r}v\partial_{rr}P\mathrm{d}r
				-C E(t)^{\frac{1}{2}}\mathcal{D}(t),
				~~~~\int_{1}^{+\infty}r^2l_2 \partial_{rr}v\mathrm{d}r
				\leq C E(t)^{\frac{1}{2}}\mathcal{D}(t).
			\end{aligned}
		\end{equation}
		Thus, we can sort out the equation (\ref{3.66}) to obtain
		\begin{equation}\label{72}
			\begin{aligned}
				&\frac{\mathrm{d}}{\mathrm{d} t}\int_{1}^{+\infty}\left(\frac{r^2 P^{\prime}(\rho)}{2\rho} \left(\partial_{rr} \rho\right)^2+\frac{r^2\rho}{2} (\partial_{rr}v)^2\right)\mathrm{d}r
				+\frac{\mathrm{d}}{\mathrm{d} t}\int_{1}^{+\infty}\frac{r^2 P^{\prime\prime}(\rho)}{\rho} \partial_{rr} \rho(\partial_{r}\rho)^2\mathrm{d}r\\
				&+\int_{1}^{+\infty}(-\frac{2 r^2}{3} \partial_{rrr}\tilde{S}_1 \partial_{rr}v-2r\partial_{rr}\tilde{S}_1\partial_{rr}v+4\partial_{r}\tilde{S}_1  \partial_{rr}v-\frac{4}{r}\tilde{S}_1\partial_{rr}v
				-r^2 \partial_{rrr}\tilde{S}_2 \partial_{rr}v)\mathrm{d}r\\
				&+\int_{1}^{+\infty}\partial_{r}(r^2 \partial_{rr}v\partial_{rr}P)\mathrm{d}r
				+\int_{1}^{+\infty}(-4\partial_{r}v\partial_{rr}P+\frac{4 v}{r}\partial_{rr}P)\mathrm{d}r
				\leq C E(t)^{\frac{1}{2}}\mathcal{D}(t).\\
			\end{aligned}
		\end{equation}

		The term \((-4\partial_{r}v\partial_{rr}P+\dfrac{4 v}{r}\partial_{rr}P)\) in the above equation cannot be directly estimated. To address this, we exploit the structure of system (\ref{3.4}) by performing the following multiplications:
		multiply equation $(\ref{22})_1$ by $(\dfrac{4\partial_r P}{\rho} + \dfrac{2r\partial_{rr}P}{\rho})$, $(\ref{60})_1$ by $\dfrac{2r\partial_r P}{\rho}$,  $(\ref{3.4})_2$ by $(\dfrac{4}{r^2}v - 2\partial_{rr}v - \dfrac{4}{r}\partial_r v)$,  $(\ref{22})_2$ by $(4\partial_r v + 2r\partial_{rr}v - \dfrac{4}{r}v)$, and $(\ref{60})_2$ by $(2r\partial_r v - 2v)$, respectively. 
		
		Summing these manipulated equations and combining the result with equation (\ref{72}), we obtain
			\begin{align}\label{3.81}
				&\frac{\mathrm{d}}{\mathrm{d} t}\int_{1}^{+\infty}\bigg(\frac{r^2 P^{\prime}(\rho)}{2\rho} \left(\partial_{rr} \rho\right)^2
				+\frac{r^2\rho}{2} (\partial_{rr}v)^2\bigg)\mathrm{d}r
				+\frac{\mathrm{d}}{\mathrm{d} t}\int_{1}^{+\infty}\frac{r^2 P^{\prime\prime}(\rho)}{\rho} \partial_{rr} \rho(\partial_{r}\rho)^2\mathrm{d}r
				 \nonumber \\
				&+\int_{1}^{+\infty}\partial_{r}(r^2 \partial_{rr}v\partial_{rr}P)\mathrm{d}r
				+\int_{1}^{+\infty}(-4\partial_{r}v\partial_{rr}P+\frac{4 v}{r}\partial_{rr}P)\mathrm{d}r \nonumber\\
				&+\int_{1}^{+\infty}\bigg[\left(\partial_{tr}\rho+\partial_{rr}(\rho v)+\dfrac{2}{r} \partial_r(\rho v)-\dfrac{2}{r^2} \rho v\right)(\dfrac{4\partial_{r}{P}}{\rho}+\dfrac{2r\partial_{rr}{P}}{\rho}) \nonumber\\
				&+\left(\partial_{trr}\rho+\partial_{rrr}(\rho v)+\dfrac{2}{r} \partial_{rr}(\rho v)-\dfrac{4}{r^2} \partial_{r}(\rho v)+\dfrac{4}{r^3}\rho v\right) \dfrac{2r\partial_{r}{P}}{\rho} \nonumber\\
				&+\left(\rho\partial_{t}v+\rho v \partial_r v+\partial_r P-\dfrac{2}{3}\partial_r \tilde{S}_1-\dfrac{2}{r} \tilde{S}_1-\partial_r \tilde{S}_2\right)(\dfrac{4}{r^2}v-2\partial_{rr}v-\dfrac{4}{r}\partial_{r}v) \nonumber\\
				&+\left(\rho\partial_{tr}v+\rho v \partial_{rr} v+\partial_{rr} P-\dfrac{2}{3}\partial_{rr}  \tilde{S}_1-\dfrac{2}{r} \partial_r \tilde{S}_1+\dfrac{2}{r^2}\tilde{S}_1-\partial_{rr} \tilde{S}_2-f_2\right)(4\partial_{r}v+2r\partial_{rr}v-\dfrac{4}{r}v) \nonumber\\
				&+\left(\rho\partial_{trr}v+\rho v \partial_{rrr} v+\partial_{rrr} P-\dfrac{2}{3}\partial_{rrr}  \tilde{S}_1-\dfrac{2}{r} \partial_{rr} \tilde{S}_1+\dfrac{4}{r^2}\partial_{r}\tilde{S}_1-\dfrac{4}{r^3}\tilde{S}_1-\partial_{rrr} \tilde{S}_2-l_2\right)
				(2r\partial_{r}v-2v)\bigg]\mathrm{d}r \nonumber\\
				&+\int_{1}^{+\infty}\bigg(-\frac{2 r^2}{3} \partial_{rrr}\tilde{S}_1 \partial_{rr}v-2r\partial_{rr}\tilde{S}_1\partial_{rr}v+4\partial_{r}\tilde{S}_1  \partial_{rr}v-\frac{4}{r}\tilde{S}_1\partial_{rr}v
				-r^2 \partial_{rrr}\tilde{S}_2 \partial_{rr}v\bigg)\mathrm{d}r \nonumber\\
				&\leq C E(t)^{\frac{1}{2}}\mathcal{D}(t).
			\end{align}
		
		Similar to before, we will temporarily not consider the terms involving $\tilde{S}_1$ and $\tilde{S}_2$. 
		We introduce $G_1$ and $G_2$, defined as
		\begin{equation}\nonumber
			\begin{aligned}
				G_1
				&=\partial_{tr} \rho(\dfrac{4\partial_{r}{P}}{\rho}+\dfrac{2r\partial_{rr}{P}}{\rho})
				+\partial_{trr} \rho\dfrac{2r\partial_{r}{P}}{\rho}+(\rho\partial_{t} v+\rho v \partial_{r} v)(\dfrac{4}{r^2}v-2\partial_{rr}v-\dfrac{4}{r}\partial_{r}v)\\
				&\quad+(\rho\partial_{tr} v+\rho v \partial_{rr} v)(4\partial_{r}v+2r\partial_{rr}v-\dfrac{4}{r}v)
				+(\rho\partial_{trr} v+\rho v \partial_{rrr} v)(2r\partial_{r}v-2v),
			\end{aligned}
		\end{equation}
		and
		\begin{equation}\nonumber
			\begin{aligned}
				G_2&=
				(\partial_{rr}(\rho v)+\dfrac{2}{r} \partial_r(\rho v)-\dfrac{2}{r^2} \rho v) (\dfrac{4\partial_{r}{P}}{\rho}+\dfrac{2r\partial_{rr}{P}}{\rho})\\
				&\quad+(\partial_{rrr}(\rho v)+\dfrac{2}{r} \partial_{rr}(\rho v)-\dfrac{4}{r^2} \partial_{r}(\rho v)+\dfrac{4}{r^3}\rho v) \dfrac{2r\partial_{r}{P}}{\rho}\\
				&\quad+\partial_{r} P (\dfrac{4}{r^2}v-2\partial_{rr}v-\dfrac{4}{r}\partial_{r}v)
				+\partial_{rr} P (4\partial_{r}v+2r\partial_{rr}v-\dfrac{4}{r}v)
				+\partial_{rrr} P(2r\partial_{r}v-2v).
			\end{aligned}
		\end{equation}
		Specifically, we have
		%\begin{equation}\nonumber
			\begin{align*}
				&\int_{1}^{+\infty}G_1\mathrm{d}r\\
				\geq&\int_{1}^{+\infty}\bigg(\frac{2 P^{\prime}(\rho)}{\rho}((\partial_{r}\rho)^2)_t
				+\frac{2r P^{\prime\prime}(\rho)}{\rho}\partial_{tr}\rho(\partial_{r}\rho)^2
				+\left[\frac{2rP^{\prime}(\rho)}{\rho}\partial_{r}\rho\partial_{rr}\rho\right]_t 
				-\left[\frac{P^{\prime}(\rho)}{\rho}\right]_t2r\partial_{r}\rho\partial_{rr}\rho\bigg)\mathrm{d}r\\
				&+\int_{1}^{+\infty}\bigg(
				\frac{2\rho}{r^2}(v^2)_t -2\rho\partial_{t} v\partial_{rr} v-\frac{4\rho}{r}\partial_{t} v\partial_{r} v
				+2\rho((\partial_{r}v)^2)_t+2r\rho\partial_{tr} v\partial_{rr} v-\frac{4\rho}{r}v\partial_{tr} v\bigg)\mathrm{d}r\\
				&+\int_{1}^{+\infty}\bigg(
				2r\rho\partial_{trr}v\partial_{r} v-2\rho v\partial_{trr} v+2r\rho v \partial_{rrr}v\partial_{r} v-2\rho v^2 \partial_{rrr}v
				\bigg)\mathrm{d}r-C E(t)^{\frac{1}{2}}\mathcal{D}(t) \displaybreak\\
				\geq&\int_{1}^{+\infty}\bigg(\frac{2 P^{\prime}(\rho)}{\rho}((\partial_{r}\rho)^2)_t
				+\left[\frac{2rP^{\prime}(\rho)}{\rho}\partial_{r}\rho\partial_{rr}\rho\right]_t\bigg)\mathrm{d}r
				+\int_{1}^{+\infty}\bigg((\frac{2\rho}{r^2}v^2)_t -(2\rho v\partial_{rr} v)_t -(\frac{4\rho}{r} v\partial_{r} v)_t\\
				&\qquad\qquad\quad+(2\rho(\partial_{r} v)^2)_t
				+(2r\rho\partial_{rr}v\partial_{r} v)_t
				-(2r\rho v\partial_{r} v)_r\partial_{rr} v
				+(2\rho v^2)_r\partial_{rr} v
				\bigg)\mathrm{d}r
				-C E(t)^{\frac{1}{2}}\mathcal{D}(t)\\
				\geq& \frac{\mathrm{d}}{\mathrm{d} t}\int_{1}^{+\infty}
				\bigg(\frac{2 P^{\prime}(\rho)}{\rho}(\partial_{r}\rho)^2+\frac{P^{\prime}(\rho)}{\rho}2r\partial_{r}\rho\partial_{rr}\rho
				+2\rho(\partial_{r}v)^2+\frac{2\rho}{r^2}v^2+2r\rho\partial_{rr}v\partial_{r}v-2\rho v\partial_{rr}v\\
				&\qquad\qquad\quad-\frac{4\rho v}{r}\partial_{r}v\bigg)\mathrm{d}r
				-C E(t)^{\frac{1}{2}}\mathcal{D}(t),
			\end{align*}
		%\end{equation}
		and
		\begin{equation}\nonumber
			\begin{aligned}
				&\int_{1}^{+\infty}G_2\mathrm{d}r\\
				\geq&
				\int_{1}^{+\infty}\bigg((4\partial_{r}P\partial_{rr}v+\frac{8}{r}\partial_{r}P\partial_{r}v-\frac{8}{r^2}v\partial_{r}P)+(2r\partial_{rr}P\partial_{rr}v
				+4\partial_{rr}P\partial_{r}v-\frac{4}{r}v\partial_{rr}P)\\
				&\qquad+(2r\partial_{rr}P\partial_{rrr}v+4\partial_{r}P\partial_{rr}v
				-\frac{8}{r}\partial_{r}P\partial_{r}v+\frac{8}{r^2}v\partial_{r}P)+(-2\partial_{r}P\partial_{rr}v-\frac{4}{r}\partial_{r}P\partial_{r}v+\frac{4}{r^2}v\partial_{r}P)\\
				&\qquad+(2r\partial_{rr}P\partial_{rr}v+4\partial_{rr}P\partial_{r}v-
				\frac{4}{r}v\partial_{rr}P)+(2r\partial_{rrr}P\partial_{rr}v-2v\partial_{rrr}P)\bigg)\mathrm{d}r
				-C E(t)^{\frac{1}{2}}\mathcal{D}(t)\\
				=&\int_{1}^{+\infty}\bigg(2\partial_{r}P\partial_{rr}v-\frac{4}{r}\partial_{r}P\partial_{r}v+\frac{4}{r^2}\partial_{r}Pv+2r\partial_{r}P\partial_{rrr}v
				+4\partial_{r}P\partial_{rr}v
				+4r\partial_{rr}P\partial_{rr}v+8\partial_{rr}P\partial_{r}v\\
				&\qquad-\frac{8}{r}v\partial_{rr}P+2r\partial_{rrr}P\partial_{r}v-2v\partial_{rrr}P\bigg)\mathrm{d}r
				-C E(t)^{\frac{1}{2}}\mathcal{D}(t)\\
				=&\int_{1}^{+\infty}\bigg(\partial_{r}\left(2r\partial_{r}P\partial_{rr}v\right)+\partial_{r}\left(2r\partial_{rr}P\partial_{r}v\right)+
				\partial_{r}\left(4\partial_{r}v\partial_{r}P\right)
				-2\partial_{r}\left(v\partial_{rr}P\right)
				-\partial_{r}(\frac{4}{r}v\partial_{r}P)\\
				&\qquad+(4\partial_{r}v\partial_{rr}P-\dfrac{4 v}{r}\partial_{rr}P)\bigg)\mathrm{d}r
				-C E(t)^{\frac{1}{2}}\mathcal{D}(t).
			\end{aligned}
		\end{equation}
		To simplify notation, the terms involving $\tilde{S}_1$ and $\tilde{S}_2$ in equation (\ref{3.81}) are denoted as \(G_3\), 
		\begin{align*}\nonumber
				G_3&=
				\left(-\dfrac{2}{3}\partial_r \tilde{S}_1-\dfrac{2}{r} \tilde{S}_1-\partial_r \tilde{S}_2\right)(\dfrac{4}{r^2}v-2\partial_{rr}v-\dfrac{4}{r}\partial_{r}v)\\
				&\quad+\left(-\dfrac{2}{3}\partial_{rr}  \tilde{S}_1-\dfrac{2}{r} \partial_r \tilde{S}_1+\dfrac{2}{r^2}\tilde{S}_1-\partial_{rr} \tilde{S}_2\right)(4\partial_{r}v+2r\partial_{rr}v-\dfrac{4}{r}v)\\
				&\quad+\left(-\dfrac{2}{3}\partial_{rrr}  \tilde{S}_1-\dfrac{2}{r} \partial_{rr} \tilde{S}_1+\dfrac{4}{r^2}\partial_{r}\tilde{S}_1-\dfrac{4}{r^3}\tilde{S}_1-\partial_{rrr} \tilde{S}_2\right)
				(2r\partial_{r}v-2v)\\
				&\quad+\bigg(-\frac{2 r^2}{3} \partial_{rrr}\tilde{S}_1 \partial_{rr}v-2r\partial_{rr}\tilde{S}_1\partial_{rr}v+4\partial_{r}\tilde{S}_1  \partial_{rr}v-\frac{4}{r}\tilde{S}_1\partial_{rr}v
				-r^2 \partial_{rrr}\tilde{S}_2 \partial_{rr}v\bigg)
		\end{align*}
		which will be addressed later by combining equations $(\ref{60})_3$ and $(\ref{60})_4$.
		
		Therefore, equation (\ref{3.81}) can be simplified to
		
		\begin{align}\label{61}
				\frac{\mathrm{d}}{\mathrm{d} t} &\int_{1}^{+\infty}\bigg(\frac{r^2 P^{\prime}(\rho)}{2\rho} \left(\partial_{rr} \rho\right)^2+\frac{ P^{\prime}(\rho)}{\rho}2(\partial_{r}\rho)^2+\frac{P^{\prime}(\rho)}{\rho}2r\partial_{r}\rho\partial_{rr}\rho
				+\frac{r^2\rho}{2} (\partial_{rr}v)^2+2\rho(\partial_{r}v)^2+\frac{2\rho}{r^2}v^2 \nonumber\\
				&\qquad\quad+2r\rho\partial_{rr}v\partial_{r}v-2\rho\partial_{rr}v v-\frac{4\rho v}{r}\partial_{r}v\bigg)\mathrm{d}r
				+\frac{\mathrm{d}}{\mathrm{d} t}\int_{1}^{+\infty}\frac{r^2 P^{\prime\prime}(\rho)}{\rho} \partial_{rr} \rho(\partial_{r}\rho)^2\mathrm{d}r
				+\int_{1}^{+\infty}G_3\mathrm{d}r \nonumber\\
				\leq&
				\int_{1}^{+\infty}I_1\mathrm{d}r+ C E(t)^{\frac{1}{2}}\mathcal{D}(t),
		\end{align}
		where
		
	\begin{align*}				I_1=-\bigg(\partial_{r}\left(r^2\partial_{rr}P\partial_{rr}v\right) +\partial_{r}\left(2r\partial_{r}P\partial_{rr}v\right)+\partial_{r}\left(2r\partial_{rr}P\partial_{r}v\right)+
	\partial_{r}\left(4\partial_{r}v\partial_{r}P\right)-2\partial_{r}\left(v\partial_{rr}P\right)\\
	-\partial_{r}(\frac{4}{r}v\partial_{r}P)\bigg).
	\end{align*}
		
		Multiplying the equation $(\ref{60})_3$ 
		by  $\dfrac{r^2}{3\mu}\partial_{rr} \tilde{S}_1$  and  integrating over \(r \in [1, +\infty)\), we obtain
		
		\begin{equation}\label{62}
			\begin{aligned}
				&\frac{\mathrm{d}}{\mathrm{d} t}\int_{1}^{+\infty}\left[\frac{\tau}{3 \mu} \frac{r^2 \rho}{2}(\partial_{rr} \tilde{S}_1)^2\right]\mathrm{d}r
				+\int_{1}^{+\infty}\frac{r^2}{3 \mu}(\partial_{rr} \tilde{S}_1)^2\mathrm{d}r
				-\int_{1}^{+\infty}
				\frac{\tau\epsilon}{3 \mu}r^2\rho [\frac{1}{2}(\partial_{rr}\tilde{S}_1)^2]_r\mathrm{d}r\\
				&-\int_{1}^{+\infty}\bigg(\frac{2r^2}{3}\partial_{rrr} v\partial_{rr} \tilde{S}_1- \frac{2r}{3}\partial_{rr} v\partial_{rr}\tilde{S}_1
				+\frac{4}{3} \partial_{r}v\partial_{rr} \tilde{S}_1-\frac{4}{3r}v\partial_{rr} \tilde{S}_1
				\bigg)\mathrm{d}r
				=\int_{1}^{+\infty}\frac{r^2}{3\mu}l_3\partial_{rr} \tilde{S}_1
			\end{aligned}
		\end{equation}
		where
		\begin{equation}\nonumber
			\begin{aligned}
				&-\int_{1}^{+\infty}
				\frac{\tau\epsilon}{3 \mu}r^2\rho[\frac{1}{2}(\partial_{rr}\tilde{S}_1)^2]_r\mathrm{d}r
				\geq \frac{\tau\epsilon}{6 \mu}\rho(t,1)(\partial_{rr}\tilde{S}_1)^2(t,1)
				- C E^{\frac{1}{2}}(t) \mathcal{D}(t),\\
				&\int_{1}^{+\infty}\frac{r^2}{3\mu}l_3\partial_{rr} \tilde{S}_1\mathrm{d}r
				\leq C E^{\frac{1}{2}}(t) \mathcal{D}(t).
			\end{aligned}
		\end{equation}
		
		Similarly, multiplying the equation $(\ref{60})_4$ by  $\dfrac{r^2}{\lambda}\partial_{rr} \tilde{S}_2$, we get
		
		\begin{equation}
			\begin{aligned}\label{63}
				&\frac{\mathrm{d}}{\mathrm{d} t}\int_{1}^{+\infty}\left[\frac{\tau }{\lambda} \frac{r^2 \rho}{2}(\partial_{rr} \tilde{S}_2)^2\right]\mathrm{d}r
				+\int_{1}^{+\infty}\frac{r^2}{\lambda}(\partial_{rr} \tilde{S}_2)^2\mathrm{d}r
				-\int_{1}^{+\infty}
				\frac{\tau\epsilon}{\lambda}r^2\rho [\frac{1}{2}(\partial_{rr}\tilde{S}_2)^2]_r\mathrm{d}r\\
				&-\int_{1}^{+\infty}\bigg(r^2\partial_{rrr}v\partial_{rr} \tilde{S}_2 +2r\partial_{rr}v\partial_{rr} \tilde{S}_2
				-4\partial_{r}v\partial_{rr} \tilde{S}_2+\frac{4}{r}v\partial_{rr} \tilde{S}_2\bigg)\mathrm{d}r =\int_{1}^{+\infty}\frac{r^2}{\lambda}l_4\partial_{rr}\tilde{S}_2
			\end{aligned}
		\end{equation}
		where
		\begin{equation}\nonumber
			\begin{aligned}
				&-\int_{1}^{+\infty}
				\frac{\tau \epsilon}{\lambda}r^2\rho[\frac{1}{2}(\partial_{rr}\tilde{S}_2)^2]_r\mathrm{d}r
				\geq \frac{\tau\epsilon}{2\lambda}\rho(t,1)(\partial_{rr}\tilde{S}_2)^2(t,1)
				- C E^{\frac{1}{2}}(t) \mathcal{D}(t),\\
				&\int_{1}^{+\infty}\frac{r^2}{\lambda}l_4\partial_{rr} \tilde{S}_2\mathrm{d}r
				\leq C E^{\frac{1}{2}}(t) \mathcal{D}(t).
			\end{aligned}
		\end{equation}
		Thus, by adding (\ref{62}) and (\ref{63}), we derive
		\begin{equation}\label{3.80}
			\begin{aligned}
				&\frac{\mathrm{d}}{\mathrm{d} t}\int_{1}^{+\infty}\left[\frac{\tau}{3 \mu} \frac{r^2 \rho}{2}(\partial_{rr} \tilde{S}_1)^2+\frac{\tau }{\lambda} \frac{r^2 \rho}{2}(\partial_{rr} \tilde{S}_2)^2\right]\mathrm{d}r
				+\int_{1}^{+\infty}\left(\frac{r^2}{3 \mu}(\partial_{rr} \tilde{S}_1)^2+\frac{r^2}{\lambda}(\partial_{rr} \tilde{S}_2)^2\right)\mathrm{d}r
				\\
				&-\int_{1}^{+\infty}\bigg(\frac{2r^2}{3}\partial_{rrr} v\partial_{rr} \tilde{S}_1- \frac{2r}{3}\partial_{rr} v\partial_{rr}\tilde{S}_1
				+\frac{4}{3} \partial_{r}v\partial_{rr} \tilde{S}_1-\frac{4}{3r}v\partial_{rr} \tilde{S}_1
				\bigg)\mathrm{d}r\\
				&-\int_{1}^{+\infty}\bigg(r^2\partial_{rrr}v\partial_{rr} \tilde{S}_2 +2r\partial_{rr}v\partial_{rr} \tilde{S}_2
				-4\partial_{r}v\partial_{rr} \tilde{S}_2+\frac{4}{r}v\partial_{rr} \tilde{S}_2\bigg)\mathrm{d}r\\
				&+\frac{\tau\epsilon}{6 \mu}\rho(t,1)(\partial_{rr}\tilde{S}_1)^2(t,1)
				+\frac{\tau\epsilon}{2\lambda}\rho(t,1)(\partial_{rr}\tilde{S}_2)^2(t,1)
				\leq C E^{\frac{1}{2}}(t) \mathcal{D}(t).
			\end{aligned}
		\end{equation}
		
		To address the third and fourth terms on the left-hand side of \eqref{3.80}, we exploit the structure of system (\ref{3.4}) and couple them with equation (\ref{61}) via the following operations:\\ 
		multiply equation \((\ref{3.4})_3\) by \((\dfrac{12}{\mu r^2}\tilde{S}_1 - \dfrac{2}{\mu}\partial_{rr}\tilde{S}_1 - \dfrac{4}{\mu r}\partial_{r}\tilde{S}_1)\), 
		\((\ref{22})_3\) by \((\dfrac{4}{3\mu}\partial_{r}\tilde{S}_1 + \dfrac{2r}{3\mu}\partial_{rr}\tilde{S}_1 - \dfrac{4}{\mu r}\tilde{S}_1)\),
		\((\ref{60})_3\) by \((\dfrac{2r}{3\mu}\partial_{r}\tilde{S}_1 - \dfrac{2}{\mu}\tilde{S}_1)\),
		\((\ref{22})_4\) by \((\dfrac{4}{\lambda}\partial_{r}\tilde{S}_2 + \dfrac{2r}{\lambda}\partial_{rr}\tilde{S}_2)\),
		\((\ref{60})_4\) by \(\dfrac{2r}{\lambda}\partial_{r}\tilde{S}_2\), respectively.  
		
		Upon adding the resulting equations to \((\ref{3.80})\), we obtain 
			\begin{align}\label{64}
				&\frac{\mathrm{d}}{\mathrm{d} t}\int_{1}^{+\infty}\left[\frac{\tau}{3 \mu} \frac{r^2 \rho}{2}(\partial_{rr} \tilde{S}_1)^2+\frac{\tau }{\lambda} \frac{r^2 \rho}{2}(\partial_{rr} \tilde{S}_2)^2\right]\mathrm{d}r
				+\int_{1}^{+\infty}(\frac{r^2}{3 \mu}(\partial_{rr} \tilde{S}_1)^2+\frac{r^2}{\lambda}(\partial_{rr} \tilde{S}_2)^2)\mathrm{d}r
				\nonumber\\
				&+\frac{\tau\epsilon}{6 \mu}\rho(t,1)(\partial_{rr}\tilde{S}_1)^2(t,1) +\frac{\tau\epsilon}{2\lambda}\rho(t,1)(\partial_{rr}\tilde{S}_2)^2(t,1) \nonumber\\
				&-\int_{1}^{+\infty}\bigg(\frac{2r^2}{3}\partial_{rrr} v\partial_{rr} \tilde{S}_1- \frac{2r}{3}\partial_{rr} v\partial_{rr}\tilde{S}_1
				+\frac{4}{3} \partial_{r}v\partial_{rr} \tilde{S}_1-\frac{4}{3r}v\partial_{rr} \tilde{S}_1
				\bigg)\mathrm{d}r \nonumber\\
				&-\int_{1}^{+\infty}\bigg(r^2\partial_{rrr}v\partial_{rr} \tilde{S}_2 +2r\partial_{rr}v\partial_{rr} \tilde{S}_2
				-4\partial_{r}v\partial_{rr} \tilde{S}_2+\frac{4}{r}v\partial_{rr} \tilde{S}_2\bigg)\mathrm{d}r \nonumber
				\\
				&+\int_{1}^{+\infty}\bigg[\left(\tau \rho\partial_{t}\tilde{S}_1+\tau \rho v\partial_r\tilde{S}_1  +\tilde{S}_1-2\mu(\partial_r v- \dfrac{v}{r})-\tau\epsilon\rho\partial_{r}\tilde{S}_1\right)(\dfrac{12}{\mu r^2}\tilde{S}_1-\dfrac{2}{\mu}\partial_{rr}\tilde{S}_1-\dfrac{4}{\mu r}\partial_{r}\tilde{S}_1) \nonumber\\
				&\quad
				+\bigg(\tau \rho\partial_{tr}\tilde{S}_1+\tau \rho v \partial_{rr}\tilde{S}_1 
				-2\mu\big(\partial_{rr} v- \dfrac{1}{r}\partial_r v
				+\dfrac{1}{r^2} v \big)+\partial_r \tilde{S}_1 \nonumber \\
				&\qquad\qquad\qquad\qquad\qquad\qquad\qquad\qquad\qquad\qquad\qquad\quad-\tau\epsilon\rho\partial_{rr}\tilde{S}_1
				\bigg)(\dfrac{4}{3\mu}\partial_{r}\tilde{S}_1+\dfrac{2r}{3\mu}\partial_{rr}\tilde{S}_1 -\dfrac{4}{\mu r}\tilde{S}_1) \nonumber\\
				&\quad
				+\bigg(\tau \rho\partial_{trr}\tilde{S}_1+\tau \rho v \partial_{rrr}\tilde{S}_1 
				-2\mu\big(\partial_{rrr} v- \dfrac{1}{r}\partial_{rr} v
				+\dfrac{2}{r^2} \partial_{r}v-\dfrac{2}{r^3}v \big)+\partial_{rr} \tilde{S}_1 \nonumber \\ &\qquad\qquad\qquad\qquad\qquad\qquad\qquad\qquad\qquad\qquad\qquad\quad-\tau\epsilon\rho\partial_{rrr}\tilde{S}_1\bigg)(\dfrac{2r}{3\mu}\partial_{r}\tilde{S}_1 -\dfrac{2}{\mu}\tilde{S}_1) \nonumber\\
				&\quad
				+\left(\tau \rho\partial_{tr}\tilde{S}_2+\tau  \rho v \partial_{rr}\tilde{S}_2
				-\lambda \big(\partial_{rr} v+\dfrac{2}{r}\partial_r v-\dfrac{2}{r^2} v\big)+\partial_r \tilde{S}_2 -\tau\epsilon\rho\partial_{rr}\tilde{S}_2\right)
				(\dfrac{4}{\lambda}\partial_{r}\tilde{S}_2+\dfrac{2r}{\lambda}\partial_{rr}\tilde{S}_2) \nonumber \\
				&\quad
				+\left(\tau \rho\partial_{trr}\tilde{S}_2+\tau  \rho v \partial_{rrr}\tilde{S}_2
				-\lambda \big(\partial_{rrr} v+\dfrac{2}{r}\partial_{rr} v-\dfrac{4}{r^2} \partial_{r}v+\dfrac{4}{r^3}v\big)+\partial_{rr} \tilde{S}_2-\tau\epsilon\rho\partial_{rrr}\tilde{S}_2
				\right)\dfrac{2r}{\lambda}\partial_{r}\tilde{S}_2\bigg]\mathrm{d}r \nonumber \\
				&
				~\leq ~C E^{\frac{1}{2}}(t) \mathcal{D}(t).
			\end{align}
		We rewrite the last three terms on the left-hand side of (\ref{64}) as $\int_{1}^{+\infty}(H_1+H_2+H_3+H_4)\mathrm{d}r$, where
		\begin{equation}\nonumber
			\begin{aligned}
				H_1=&	
				\left(\tau \rho\partial_{t}\tilde{S}_1+\tau \rho v\partial_r\tilde{S}_1  \right)(\dfrac{12}{\mu r^2}\tilde{S}_1-\dfrac{2}{\mu}\partial_{rr}\tilde{S}_1-\dfrac{4}{\mu r}\partial_{r}\tilde{S}_1)\\
				&
				+\left(\tau \rho\partial_{tr}\tilde{S}_1+\tau \rho v \partial_{rr}\tilde{S}_1 
				\right)(\dfrac{4}{3\mu}\partial_{r}\tilde{S}_1+\dfrac{2r}{3\mu}\partial_{rr}\tilde{S}_1 -\dfrac{4}{\mu r}\tilde{S}_1)\\
				&+\left(\tau \rho\partial_{trr}\tilde{S}_1+\tau \rho v \partial_{rrr}\tilde{S}_1 
				\right)(\dfrac{2r}{3\mu}\partial_{r}\tilde{S}_1 -\dfrac{2}{\mu}\tilde{S}_1)\\
				&
				+\left(\tau \rho\partial_{tr}\tilde{S}_2+\tau  \rho v \partial_{rr}\tilde{S}_2
				\right)
				(\dfrac{4}{\lambda}\partial_{r}\tilde{S}_2+\dfrac{2r}{\lambda}\partial_{rr}\tilde{S}_2)
				+\left(\tau \rho\partial_{trr}\tilde{S}_2+\tau  \rho v \partial_{rrr}\tilde{S}_2
				\right)\dfrac{2r}{\lambda}\partial_{r}\tilde{S}_2,
			\end{aligned}
		\end{equation}
		\begin{equation}\nonumber
			\begin{aligned}
				H_2
				=&\tilde{S}_1(\dfrac{12}{\mu r^2}\tilde{S}_1-\dfrac{2}{\mu}\partial_{rr}\tilde{S}_1-\dfrac{4}{\mu r}\partial_{r}\tilde{S}_1)
				+\partial_{r}\tilde{S}_1(\dfrac{4}{3\mu}\partial_{r}\tilde{S}_1+\dfrac{2r}{3\mu}\partial_{rr}\tilde{S}_1 -\dfrac{4}{\mu r}\tilde{S}_1)\\
				&
				+\partial_{rr}\tilde{S}_1(\dfrac{2r}{3\mu}\partial_{r}\tilde{S}_1 -\dfrac{2}{\mu}\tilde{S}_1)
				+\partial_{r}\tilde{S}_2
				(\dfrac{4}{\lambda}\partial_{r}\tilde{S}_2+\dfrac{2r}{\lambda}\partial_{rr}\tilde{S}_2)
				+\dfrac{2r}{\lambda}\partial_{rr}\tilde{S}_2\partial_{r}\tilde{S}_2,
			\end{aligned}
		\end{equation}
		\begin{equation}\nonumber
			\begin{aligned}
				H_3
				=
				&-2\mu(\partial_r v- \dfrac{v}{r})(\dfrac{12}{\mu r^2}\tilde{S}_1-\dfrac{2}{\mu}\partial_{rr}\tilde{S}_1-\dfrac{4}{\mu r}\partial_{r}\tilde{S}_1)\\
				&-2\mu\left(\partial_{rr} v- \dfrac{1}{r}\partial_r v
				+\dfrac{1}{r^2} v \right)(\dfrac{4}{3\mu}\partial_{r}\tilde{S}_1+\dfrac{2r}{3\mu}\partial_{rr}\tilde{S}_1 -\dfrac{4}{\mu r}\tilde{S}_1)\\
				&
				-2\mu\left(\partial_{rrr} v- \dfrac{1}{r}\partial_{rr} v
				+\dfrac{2}{r^2} \partial_{r}v-\dfrac{2}{r^3}v \right)(\dfrac{2r}{3\mu}\partial_{r}\tilde{S}_1 -\dfrac{2}{\mu}\tilde{S}_1)\\
				&-\lambda \left(\partial_{rr} v+\dfrac{2}{r}\partial_r v-\dfrac{2}{r^2} v\right)
				(\dfrac{4}{\lambda}\partial_{r}\tilde{S}_2+\dfrac{2r}{\lambda}\partial_{rr}\tilde{S}_2)
				-\lambda \left(\partial_{rrr} v+\dfrac{2}{r}\partial_{rr} v-\dfrac{4}{r^2} \partial_{r}v+\dfrac{4}{r^3}v\right)\dfrac{2r}{\lambda}\partial_{r}\tilde{S}_2\\
				&-\bigg(\frac{2r^2}{3}\partial_{rrr} v\partial_{rr} \tilde{S}_1- \frac{2r}{3}\partial_{rr} v\partial_{rr}\tilde{S}_1
				+\frac{4}{3} \partial_{r}v\partial_{rr} \tilde{S}_1-\frac{4}{3r}v\partial_{rr} \tilde{S}_1
				\bigg)\\
				&-\bigg(r^2\partial_{rrr}v\partial_{rr} \tilde{S}_2 +2r\partial_{rr}v\partial_{rr} \tilde{S}_2
				-4\partial_{r}v\partial_{rr} \tilde{S}_2+\frac{4}{r}v\partial_{rr} \tilde{S}_2\bigg),
			\end{aligned}
		\end{equation}
		and
		\begin{equation}\nonumber
			\begin{aligned}
				H_4=
				&-\tau\epsilon\rho\partial_{r}\tilde{S}_1(\dfrac{12}{\mu r^2}\tilde{S}_1-\dfrac{2}{\mu}\partial_{rr}\tilde{S}_1-\dfrac{4}{\mu r}\partial_{r}\tilde{S}_1)
				-\tau\epsilon\rho\partial_{rr}\tilde{S}_1
				(\dfrac{4}{3\mu}\partial_{r}\tilde{S}_1+\dfrac{2r}{3\mu}\partial_{rr}\tilde{S}_1 -\dfrac{4}{\mu r}\tilde{S}_1)\\
				&-{\tau\epsilon\rho\partial_{rrr}\tilde{S}_1}(\dfrac{2r}{3\mu}\partial_{r}\tilde{S}_1 -\dfrac{2}{\mu}\tilde{S}_1)
				-\tau\epsilon\rho\partial_{rr}\tilde{S}_2
				(\dfrac{4}{\lambda}\partial_{r}\tilde{S}_2+\dfrac{2r}{\lambda}\partial_{rr}\tilde{S}_2)
				-\tau\epsilon\rho\partial_{rrr}\tilde{S}_2
				\dfrac{2r}{\lambda}\partial_{r}\tilde{S}_2.
			\end{aligned}
		\end{equation}
		Next, we handle each \(\int_{1}^{+\infty}H_i\mathrm{d}r\) individually, \(i=1,2,3,4\).
		\begin{equation}\nonumber
			\begin{aligned}
				&\int_{1}^{+\infty}H_1\mathrm{d}r\\
				\geq&
				\int_{1}^{+\infty}\tau\rho\bigg(
				\frac{12}{\mu r^2}\partial_{t}\tilde{S}_1\tilde{S}_1
				-\frac{2}{\mu}\partial_{t}\tilde{S}_1\partial_{rr}\tilde{S}_1
				-\frac{4}{\mu r} \partial_{t}\tilde{S}_1\partial_{r}\tilde{S}_1
				+\frac{4}{3\mu} \partial_{tr}\tilde{S}_1\partial_{r}\tilde{S}_1
				+\frac{2r}{3\mu} \partial_{tr}\tilde{S}_1\partial_{rr}\tilde{S}_1
				\\
				&\qquad\qquad\quad-\frac{4}{\mu r} \partial_{tr}\tilde{S}_1\tilde{S}_1
				+\frac{2r}{3\mu} \partial_{trr}\tilde{S}_1\partial_{r}\tilde{S}_1
				-\frac{2}{\mu} \partial_{trr}\tilde{S}_1\tilde{S}_1
				+\frac{2r}{3\mu} v\partial_{rrr}\tilde{S}_1\partial_{r}\tilde{S}_1
				-\frac{2}{\mu} v\partial_{rrr}\tilde{S}_1\tilde{S}_1
				\bigg)\mathrm{d}r\\
				&+\int_{1}^{+\infty}\tau\rho\bigg(
				\frac{4}{\lambda}\partial_{tr}\tilde{S}_2\partial_{r}\tilde{S}_2
				+\frac{2r}{\lambda}\partial_{tr}\tilde{S}_2\partial_{rr}\tilde{S}_2
				+\frac{2r}{\lambda}\partial_{trr}\tilde{S}_2\partial_{r}\tilde{S}_2
				+\frac{2r}{\lambda}v\partial_{rrr}\tilde{S}_2\partial_{r}\tilde{S}_2
				\bigg)\mathrm{d}r
				-C E^{\frac{1}{2}}(t) \mathcal{D}(t)\\
				\geq&
				\int_{1}^{+\infty}\tau\rho\bigg(
				\frac{6}{\mu r^2}(\tilde{S}_1^2)_t
				-\frac{2}{\mu}(\tilde{S}_1\partial_{rr}\tilde{S}_1)_t
				-\frac{4}{\mu r}(\partial_{r}\tilde{S}_1\tilde{S}_1)_t
				+\frac{2}{3\mu} [(\partial_{r}\tilde{S}_1)^2]_t
				+\frac{2r}{3\mu} (\partial_{r}\tilde{S}_1\partial_{rr}\tilde{S}_1)_t
				\\
				&\qquad\qquad-(\frac{2r}{3\mu} v\partial_{r}\tilde{S}_1)_r\partial_{rr}\tilde{S}_1+(\frac{2}{\mu} v\tilde{S}_1)_r\partial_{rr}\tilde{S}_1
				\bigg)\mathrm{d}r\\
				&+\int_{1}^{+\infty}\tau\rho\bigg(
				\frac{2}{\lambda}[(\partial_{r} \tilde{S}_2)^2]_t
				+\frac{2r}{\lambda}(\partial_{r} \tilde{S}_2\partial_{rr} \tilde{S}_2)_t
				-(\frac{2r}{\lambda}v\partial_{r} \tilde{S}_2)_r\partial_{rr} \tilde{S}_2
				\bigg)\mathrm{d}r
				-C E^{\frac{1}{2}}(t) \mathcal{D}(t)\\
				\geq&
				\frac{\mathrm{d}}{\mathrm{d} t}\int_{1}^{+\infty}\tau\rho\bigg(
				\frac{2}{3\mu}(\partial_{r}\tilde{S}_1)^2+\frac{6}{\mu r^2}(\tilde{S}_1)^2 +\frac{2r}{3\mu}\partial_{rr}\tilde{S}_1 \partial_{r}\tilde{S}_1-\frac{2}{\mu}\partial_{rr}\tilde{S}_1\tilde{S}_1-\frac{4}{\mu r}\partial_{r}\tilde{S}_1 \tilde{S}_1
				+\frac{2}{\lambda}(\partial_{r} \tilde{S}_2)^2\\
				&\qquad\qquad\qquad\qquad\qquad\qquad\qquad\qquad\qquad\qquad\qquad\qquad\qquad\quad+\frac{2r}{\lambda}\partial_{rr} \tilde{S}_2\partial_{r} \tilde{S}_2\bigg)\mathrm{d}r
				-C E^{\frac{1}{2}}(t) \mathcal{D}(t),
			\end{aligned}
		\end{equation}
		and
			\begin{align*}
				\int_{1}^{+\infty}H_2\mathrm{d}r
				=\int_{1}^{+\infty}\bigg(\frac{4}{3\mu}(\partial_{r}\tilde{S}_1)^2+\frac{12}{\mu r^2}(\tilde{S}_1)^2 +\frac{4r}{3\mu}\partial_{rr}\tilde{S}_1 \partial_{r}\tilde{S}_1-\frac{4}{\mu}\partial_{rr}\tilde{S}_1\tilde{S}_1-\frac{8}{\mu r}\partial_{r}\tilde{S}_1 \tilde{S}_1
				+\frac{4}{\lambda}(\partial_{r}  \tilde{S}_2)^2\\
				+\frac{4r}{\lambda}\partial_{rr} \tilde{S}_2\partial_{r} \tilde{S}_2\bigg)\mathrm{d}r.
			\end{align*}
		
		To address the term $\int_{1}^{+\infty}H_3\mathrm{d}r$, we focus on $H_3+G_3$,
		
		\begin{equation*}
			\begin{aligned}
				&\int_{1}^{+\infty}(H_3+G_3)\mathrm{d}r\\
				=&\int_{1}^{+\infty}\bigg[-\frac{2}{3}(r^2\partial_{rrr}v\partial_{rr}\tilde{S}_1+r^2\partial_{rr}v\partial_{rrr}\tilde{S}_1+2r\partial_{rr}v\partial_{rr}\tilde{S}_1)\\
				&\quad\quad\quad-(\frac{4r}{3}\partial_{rr}v\partial_{rr}\tilde{S}_1+\frac{4r}{3}\partial_{r}v\partial_{rrr}\tilde{S}_1+\frac{4}{3}\partial_{r}v\partial_{rr}\tilde{S}_1)+(\frac{4}{3}\partial_{r}v\partial_{rr}\tilde{S}_1+\frac{4}{3}v\partial_{rrr}\tilde{S}_1)\\
				&\quad\quad\quad-
				(\frac{4r}{3}\partial_{rrr}v\partial_{r}\tilde{S}_1+\frac{4r}{3}\partial_{rr}v\partial_{rr}\tilde{S}_1+\frac{4}{3}\partial_{rr}v\partial_{r}\tilde{S}_1)-\frac{8}{3}(\partial_{rr}v\partial_{r}\tilde{S}_1+\partial_{r}v\partial_{rr}\tilde{S}_1)\\
				&\quad\quad\quad+ (\frac{8}{3r}v\partial_{rr}\tilde{S}_1+\frac{8}{3r}\partial_{r}v\partial_{r}\tilde{S}_1-\frac{8}{3r^2}v\partial_{r}\tilde{S}_1)
				+4(\partial_{rrr}v\tilde{S}_1+\partial_{rr}v\partial_{r}\tilde{S}_1)\\
				&\quad\quad\quad+(\frac{8}{r}\partial_{rr}v\tilde{S}_1+\frac{8}{r}\partial_{r}v
				\partial_{r}\tilde{S}_1-\frac{8}{r^2}\partial_{r}v\tilde{S}_1)
				-( \frac{8}{r^2}\partial_{r}v\tilde{S}_1+\frac{8}{r^2}v\partial_{r}\tilde{S}_1-\frac{16}{r^3}v\tilde{S}_1)\\
				&\quad\quad\quad-(r^2\partial_{rrr}v\partial_{rr}\tilde{S}_2+r^2\partial_{rr}v\partial_{rrr}\tilde{S}_2+2r\partial_{rr}v\partial_{rr}\tilde{S}_2)\\
				&\quad\quad\quad-(2r\partial_{rrr}v\partial_{r}\tilde{S}_2+2r\partial_{rr}v\partial_{rr}\tilde{S}_2+2\partial_{rr}v\partial_{r}\tilde{S}_2)\\
				&\quad\quad\quad-(2r\partial_{rr}v\partial_{rr}\tilde{S}_2+2r\partial_{r}v\partial_{rrr}\tilde{S}_2+2\partial_{r}v\partial_{rr}\tilde{S}_2)
				-(4\partial_{rr}v\partial_{r}\tilde{S}_2+4\partial_{r}v\partial_{rr}\tilde{S}_2)\\
				&\quad\quad\quad+(2\partial_{r}v\partial_{rr}\tilde{S}_2+2v\partial_{rrr}\tilde{S}_2)+(\frac{4}{r}\partial_{r}v\partial_{r}\tilde{S}_2+\frac{4}{r}v\partial_{rr}\tilde{S}_2-\frac{4}{r^2}v\partial_{r}\tilde{S}_2)\bigg]\mathrm{d}r\\
				:=&-\int_{1}^{+\infty}(I_2+I_3)\mathrm{d}r
			\end{aligned}
		\end{equation*}
		where
		
		\begin{equation}\nonumber
			\begin{aligned}
				I_2&=\bigg(\frac{2}{3}\partial_{r}(r^2\partial_{rr}v\partial_{rr}\tilde{S}_1)
				+\frac{2}{3}\partial_{r}(2r\partial_{r}v\partial_{rr}\tilde{S}_1)
				-\frac{4}{3}\partial_{r}(v\partial_{rr}\tilde{S}_1)
				+\frac{2}{3}\partial_{r}(2r\partial_{rr}v\partial_{r}\tilde{S}_1)
				+\frac{2}{3}\partial_{r}(4\partial_{r}v\partial_{r}\tilde{S}_1)\\
				&\quad\quad-\frac{4}{3}\partial_{r} (\frac{2v}{r}\partial_{r}\tilde{S}_1)
				-\frac{2}{3}\partial_{r}(6\partial_{rr}v\tilde{S}_1)
				-\frac{2}{3}\partial_{r}(\frac{12}{r}\partial_{r}v\tilde{S}_1)
				+\frac{4}{3}\partial_{r}( \frac{6v}{r^2}\tilde{S}_1)\bigg),\\
				I_3&=\bigg(\partial_{r}(r^2\partial_{rr}v\partial_{rr}\tilde{S}_2)+\partial_{r}(2r\partial_{rr}v\partial_{r}\tilde{S}_2)++\partial_{r}(2r\partial_{r}v\partial_{rr}\tilde{S}_2)+\partial_{r}(4\partial_{r}v\partial_{r}\tilde{S}_2)-2\partial_{r}(v\partial_{rr}\tilde{S}_2)\\
				&\quad\quad-2\partial_{r}(\frac{2v}{r}\partial_{r}\tilde{S}_2)\bigg).
			\end{aligned}
		\end{equation}
		
		Furthermore, we get
			\begin{align*}
				&\int_{1}^{+\infty}H_4\mathrm{~d}r\\
				=&\int_{1}^{+\infty}
				\bigg(-\tau\epsilon\rho\frac{12}{\mu r^2}\tilde{S}_1\partial_{r}\tilde{S}_1
				+\tau\epsilon\rho\frac{2}{3\mu}\partial_{r}\tilde{S}_1\partial_{rr}\tilde{S}_1
				+\tau\epsilon\rho\frac{4}{\mu r}(\partial_{r}\tilde{S}_1)^2
				-\tau\epsilon\rho\frac{2r}{3\mu}(\partial_{rr}\tilde{S}_1)^2
				+\tau\epsilon\rho\frac{4}{\mu r}\tilde{S}_1\partial_{rr}\tilde{S}_1\\
				&-\tau\epsilon\rho\frac{2r}{3\mu}\partial_{r}\tilde{S}_1\partial_{rrr}\tilde{S}_1
				+\tau\epsilon\rho\frac{2}{\mu}\tilde{S}_1\partial_{rrr}\tilde{S}_1
				-\tau\epsilon\rho\frac{4}{\lambda}\partial_{r}\tilde{S}_2\partial_{rr}\tilde{S}_2
				-\tau\epsilon\rho\frac{2r}{\lambda}(\partial_{rr}\tilde{S}_2)^2
				-\tau\epsilon\rho\frac{2r}{\lambda}\partial_{r}\tilde{S}_2\partial_{rrr}\tilde{S}_2\bigg)\mathrm{d}r\\
				\geq& C\epsilon\int_{1}^{+\infty}
				-(\tilde{S}_1^2+(\partial_{r}\tilde{S}_1)^2
				+(\partial_{rr}\tilde{S}_1)^2+(\partial_{r}\tilde{S}_2)^2
				+(\partial_{rr}\tilde{S}_2)^2)\mathrm{d}r\mathrm{d}t
				+\frac{2\tau\epsilon}{3\mu}(\rho\partial_{r}\tilde{S}_1\partial_{rr}\tilde{S}_1)(t,1)\\
				&
				-\frac{2\tau\epsilon}{\mu}(\rho\tilde{S}_1\partial_{rr}\tilde{S}_1)(t,1)
				+\frac{2\tau\epsilon}{\lambda}(\rho\partial_{r}\tilde{S}_2\partial_{rr}\tilde{S}_2)(t,1)
				-C\left(E(0)+ E^{\frac{1}{2}}(t) \int_{0}^{t}\mathcal{D}(s)\mathrm{d}s\right) \displaybreak \\
				\geq&-C\epsilon\int_{1}^{+\infty}\left((\partial_{rr}\tilde{S}_1)^2+(\partial_{rr}\tilde{S}_2)^2\right)\mathrm{d}r
				-\frac{4\tau\epsilon}{3\mu}\eta\rho(t,1)(\partial_{rr}\tilde{S}_1)^2(t,1)
				-\frac{2\tau\epsilon}{\lambda}\eta\rho(t,1)(\partial_{rr}\tilde{S}_2)^2(t,1)\\
				&-\frac{2\tau\epsilon}{\mu} C(\eta)((\tilde{S}_1)^2+(\partial_{r}\tilde{S}_1)^2)(t,1)
				-\frac{2\tau\epsilon}{\lambda} C(\eta)((\partial_{r}\tilde{S}_2)^2)(t,1)
				-C\left(E(0)+ E^{\frac{1}{2}}(t)  \int_{0}^{t}\mathcal{D}(s)\mathrm{d}s
				\right).
			\end{align*}
		Here and after, $\eta$ is any positive constant to be chosen later.	
		
		By adding (\ref{64}) to (\ref{61}) and integrating over $(0,t)$, we derive that
		\begin{equation}\label{65}
			\begin{aligned}
				&\int_{1}^{+\infty}\bigg(\frac{r^2 P^{\prime}(\rho)}{2\rho} \left(\partial_{rr} \rho\right)^2+\frac{ P^{\prime}(\rho)}{\rho}2(\partial_{r}\rho)^2+\frac{P^{\prime}(\rho)}{\rho}2r\partial_{r}\rho\partial_{rr}\rho\\%\ρ相关
				&\quad\quad
				+\frac{r^2\rho}{2} (\partial_{rr}v)^2+2\rho(\partial_{r}v)^2+\frac{2\rho}{r^2}v^2+2r\rho\partial_{rr}v\partial_{r}v-2\rho\partial_{rr}v v-\frac{4\rho v}{r}\partial_{r}v\\%v相关
				&\quad\quad
				+\tau\rho(\frac{r^2}{6\mu}(\partial_{rr}\tilde{S}_1)^2+\frac{2}{3\mu}(\partial_{r}\tilde{S}_1)^2+\frac{6}{\mu r^2}(\tilde{S}_1)^2 +\frac{2r}{3\mu}\partial_{rr}\tilde{S}_1 \partial_{r}\tilde{S}_1-\frac{2}{\mu}\partial_{rr}\tilde{S}_1\tilde{S}_1-\frac{4}{\mu r}\partial_{r}\tilde{S}_1 \tilde{S}_1\\%S_1相关
				&\quad\quad
				+\frac{r^2}{2\lambda}(\partial_{rr} \tilde{S}_2)^2+\frac{2}{\lambda}(\partial_{r} \tilde{S}_2)^2+\frac{2r}{\lambda}\partial_{rr} \tilde{S}_2\partial_{r} \tilde{S}_2)\bigg)\mathrm{d}r
				\\%S_2相关
				&+\int_{0}^{t}\int_{1}^{+\infty}\bigg(\frac{r^2}{3\mu}(\partial_{rr}\tilde{S}_1)^2+\frac{4}{3\mu}(\partial_{r}\tilde{S}_1)^2+\frac{12}{\mu r^2}(\tilde{S}_1)^2 +\frac{4r}{3\mu}\partial_{rr}\tilde{S}_1 \partial_{r}\tilde{S}_1-\frac{4}{\mu}\partial_{rr}\tilde{S}_1\tilde{S}_1-\frac{8}{\mu r}\partial_{r}\tilde{S}_1 \tilde{S}_1\\
				&\quad\quad
				+\frac{r^2}{\lambda}(\partial_{rr}\tilde{S}_2)^2+\frac{4}{\lambda}(\partial_{r}  \tilde{S}_2)^2+\frac{4r}{\lambda}\partial_{rr} \tilde{S}_2\partial_{r} \tilde{S}_2\bigg)\mathrm{d}r\mathrm{d}t\\
				&+\int_{0}^{t}\left(\frac{\tau\epsilon}{3\mu}\rho(t,1)(\partial_{rr}\tilde{S}_1)^2(t,1)
				+\frac{3\tau\epsilon}{4\lambda}\rho(t,1)(\partial_{rr}\tilde{S}_2)^2(t,1)\right)\mathrm{d}t\\ 
				&\leq \int_{0}^{t}\int_{1}^{+\infty} \left(I_1+I_2+I_3\right)\mathrm{d}r\mathrm{d}t+C \left(E(0)+E^{\frac{3}{2}}(t)+E^{\frac{1}{2}}(t) \int_{0}^{t}\mathcal{D}(s)\mathrm{d}s
				\right).
			\end{aligned}
		\end{equation}
		
		Next, we address the first term on the right-hand side of equation (\ref{65}). By applying the momentum equation and equations $(\ref{22})_1$ and $(\ref{22})_4$, we have
		
			\begin{align*}
				&\int_{1}^{+\infty}(I_1+I_2+I_3)\mathrm{d}r\\
				=&\bigg(\frac{2}{3}r^2\partial_{rr}v\partial_{rr}\tilde{S}_1
				+\frac{4}{3}r\partial_{r}v\partial_{rr}\tilde{S}_1
				-\frac{4}{3}v\partial_{rr}\tilde{S}_1
				+\frac{4}{3}r\partial_{rr}v\partial_{r}\tilde{S}_1
				+\frac{8}{3}\partial_{r}v\partial_{r}\tilde{S}_1
				-\frac{8}{3}\frac{v}{r}\partial_{r}\tilde{S}_1\\
				&\qquad
				-4\partial_{rr}v\tilde{S}_1-\frac{8}{r}\partial_{r}v\tilde{S}_1
				+\frac{8}{r^2}v\tilde{S}_1\bigg)\bigg|_{1}^{+\infty}\\
				&+
				\bigg(r^2\partial_{rr}v\partial_{rr}\tilde{S}_2+2r\partial_{rr}v\partial_{r}\tilde{S}_2+2r\partial_{r}v\partial_{rr}\tilde{S}_2+4\partial_{r}v\partial_{r}\tilde{S}_2-2v\partial_{rr}\tilde{S}_2-\frac{4v}{r}\partial_{r}\tilde{S}_2\bigg)\bigg|_{1}^{+\infty}\\
				&-\bigg(r^2\partial_{rr}P\partial_{rr}v +2r\partial_{r}P\partial_{rr}v+2r\partial_{rr}P\partial_{r}v+
				4\partial_{r}v\partial_{r}P
				-2v\partial_{rr}P-\frac{4}{r}v\partial_{r}P\bigg)\bigg|_{1}^{+\infty}\\
				=&r^2\partial_{rr}v\left(\dfrac{2}{3}\partial_{rr} \tilde{S}_1+\dfrac{2}{r} \partial_r \tilde{S}_1-\dfrac{2}{r^2}\tilde{S}_1+\partial_{rr} \tilde{S}_2-\partial_{rr}P\right)\bigg|_{1}^{+\infty}\\
				&+2r\partial_{r}v\left(\dfrac{2}{3}\partial_{rr} \tilde{S}_1+\dfrac{2}{r} \partial_r \tilde{S}_1-\dfrac{2}{r^2}\tilde{S}_1+\partial_{rr} \tilde{S}_2-\partial_{rr}P\right)\bigg|_{1}^{+\infty}\\
				&-r\partial_{rr}v\left(\dfrac{2}{3}\partial_r \tilde{S}_1+\dfrac{2}{r} \tilde{S}_1+\partial_r \tilde{S}_2-\partial_{r}P\right)\bigg|_{1}^{+\infty}
				-2\partial_{r}v\left(\dfrac{2}{3}\partial_r \tilde{S}_1+\dfrac{2}{r} \tilde{S}_1+\partial_r \tilde{S}_2-\partial_{r}P\right)\bigg|_{1}^{+\infty}\\
				&-2v\left(\dfrac{2}{3}\partial_{rr} \tilde{S}_1+\dfrac{2}{r} \partial_r \tilde{S}_1-\dfrac{2}{r^2}\tilde{S}_1+\partial_{rr} \tilde{S}_2-\partial_{rr}P\right)\bigg|_{1}^{+\infty}
				+\frac{2}{r}v\left(\dfrac{2}{3}\partial_r \tilde{S}_1+\dfrac{2}{r} \tilde{S}_1+\partial_r \tilde{S}_2-\partial_{r}P\right)\bigg|_{1}^{+\infty}\\
				&+2r\frac{\partial_{r}P}{\rho}(\partial_{tr}\rho+2\partial_{r}\rho\partial_{r}v)\bigg|_{1}^{+\infty}
				+2r\partial_{r}\tilde{S}_2\left(\frac{\tau\rho}{\lambda}\left(\partial_{tr} \tilde{S}_2+(v-\epsilon)\partial_{rr} \tilde{S}_2\right)+\partial_{r} \tilde{S}_2\right)\bigg|_{1}^{+\infty}.
			\end{align*}

		Furthermore, using momentum equation, $(\ref{22})_4$ and the boundary condition (\ref{1.6}), we have 
		\begin{equation*}
			\begin{aligned}
				&\int_{0}^{t}\int_{1}^{+\infty}(I_1+I_2+I_3)\mathrm{d}r\mathrm{d}t\\
				=& \int_{0}^{t}\bigg[\left((r^2\partial_{rr}v+2r\partial_{r}v-2v)(\rho\partial_{tr}v+\rho v\partial_{rr}v) +(-r\partial_{rr}v-2\partial_{r}v+\frac{2v}{r})
				(\rho\partial_{t}v+\rho v \partial_{r}v)\right)\bigg|_{1}^{+\infty}\\
				&\qquad+2r\frac{\partial_{r}P}{\rho}(\partial_{tr}\rho+2\partial_{r}\rho\partial_{r}v)\bigg|_{1}^{+\infty}
				+2r\partial_{r}\tilde{S}_2\left(\frac{\tau\rho}{\lambda}\left(\partial_{tr} \tilde{S}_2+(v-\epsilon)\partial_{rr} \tilde{S}_2\right)+\partial_{r} \tilde{S}_2\right)\bigg|_{1}^{+\infty}\bigg]\mathrm{d}t\\
				=&\int_{0}^{t}
				\bigg[-(r^2\partial_{rr}v+2r\partial_{r}v)\rho\partial_{tr}v(t,1)
				-2r\frac{\partial_{r}P}{\rho}(\partial_{tr}\rho+2\partial_{r}\rho\partial_{r}v)(t,1)\\
				&\qquad-2r\partial_{r}\tilde{S}_2\left(\frac{\tau\rho}{\lambda}\left(\partial_{tr} \tilde{S}_2+(v-\epsilon)\partial_{rr} \tilde{S}_2\right)+\frac{1}{\lambda}\partial_{r} \tilde{S}_2\right)(t,1)\bigg]
				\mathrm{d}t
			\end{aligned}
		\end{equation*}
	Next, we estimate each term in the above equation.
		\begin{equation}\nonumber
			\begin{aligned}
				&\int_{0}^{t}-(r^2\partial_{rr}v+2r\partial_{r}v)\rho\partial_{tr}v(t,1)\mathrm{d}t
				=-\int_{0}^{t}(\rho\partial_{rr}v\partial_{tr}v)(t,1)\mathrm{d}t- (\rho(\partial_{r}v)^2)(t,1)\bigg|_0^t\\
				&
				\leq C\int_{0}^{t}|\partial_{rr}v\partial_{tr}v(t,1)|\mathrm{d}t+E(0)
				\leq C\int_{0}^{t}((\partial_{rr}v)^2(t,1)+(\partial_{tr}v)^2(t,1))\mathrm{d}t+E(0),
			\end{aligned}
		\end{equation}
		$$
		\begin{aligned}
			&\int_{0}^{t}-2r\frac{\partial_{r}P}{\rho}(\partial_{tr}\rho+2\partial_{r}\rho\partial_{r}v)(t,1)\mathrm{d}t\\
			=& -\frac{P^{\prime}(\rho)}{\rho}(\partial_{r}\rho)^2(t,1)\bigg|_0^t
			+\int_{0}^{t}[\frac{P^{\prime}(\rho)}{\rho}]_t(\partial_{r}\rho)^2(t,1)\mathrm{d}t-\int_{0}^{t}\frac{4\partial_{r}P}{\rho}\partial_{r}\rho\partial_{r}v(t,1)\mathrm{d}t\\
			\leq&-\frac{P^{\prime}(\rho)}{\rho}(\partial_{r}\rho)^2(t,1)
			+\int_{0}^{t}C\|\partial_{t}\rho(\partial_{r}\rho)^2\|_{L^{\infty}}\mathrm{d}t
			-\int_{0}^{t}C\|\partial_{r}v(\partial_{r}\rho)^2\|_{L^{\infty}}\mathrm{d}t\\
			\leq&\int_{0}^{t}C\|\partial_{t}\rho\|_{H^{1}}\|\partial_{r}\rho\|_{H^{1}}^2\mathrm{d}t
			-\int_{0}^{t}C\|\partial_{r}v\|_{H^{1}}\|\partial_{r}\rho\|_{H^{1}}^2\mathrm{d}t
			\leq C \left(E(0)+E^{\frac{1}{2}}(t) \int_{0}^{t}\mathcal{D}(s)\mathrm{d}s\right),
		\end{aligned}
		$$
		and
		$$
		\begin{aligned}
			&\int_{0}^{t}-2r\partial_{r}\tilde{S}_2\left(\frac{\tau\rho}{\lambda}\left(\partial_{tr} \tilde{S}_2+(v-\epsilon)\partial_{rr} \tilde{S}_2\right)+\frac{1}{\lambda}\partial_{r} \tilde{S}_2\right)(t,1)\mathrm{d}t\\
			=&-\int_{0}^{t}\frac{\tau}{\lambda}\rho(t,1)(\partial_{r}\tilde{S}_2)^2_t(t,1)\mathrm{d}t
			+\int_{0}^{t}\frac{2\epsilon\tau }{\lambda}\rho(t,1)\partial_{r}\tilde{S}_2(t,1)\partial_{rr}\tilde{S}_2(t,1)\mathrm{d}t
			-\int_{0}^{t}\frac{2}{\lambda}(\partial_{r}\tilde{S}_2)^2(t,1)\mathrm{d}t\\
			\leq&-\frac{\tau}{\lambda}\rho(t,1)(\partial_{r}\tilde{S}_2)^2(t,1)\bigg|_0^t
			+\int_{0}^{t}\frac{2\tau}{\lambda}\partial_{t}\rho(\partial_{r}\tilde{S}_2)^2(t,1)\mathrm{d}t
			+\int_{0}^{t}\frac{2\epsilon\tau }{\lambda}\rho(\eta(\partial_{rr}\tilde{S}_2)^2
			+C(\eta)(\partial_{r}\tilde{S}_2)^2)(t,1)\mathrm{d}t\\
			\leq& \int_{0}^{t}\frac{2\epsilon\tau }{\lambda}\eta\rho(t,1)(\partial_{rr}\tilde{S}_2)^2(t,1)
			+C \left(E(0)+E^{\frac{1}{2}}(t) \int_{0}^{t}\mathcal{D}(s)\mathrm{d}s\right).
		\end{aligned}
		$$
		Here and after, $\eta$ is any positive constant to be chosen later.
		
		Combining the above results and using \eqref{hu3.4}, we can obtain
		\begin{equation}\label{3.85}
			\begin{aligned}
				&\int_{1}^{+\infty}\bigg(\frac{r^2 P^{\prime}(\rho)}{2\rho} \left(\partial_{rr} \rho\right)^2+\frac{2 P^{\prime}(\rho)}{\rho}(\partial_{r}\rho)^2+\frac{2rP^{\prime}(\rho)}{\rho}\partial_{r}\rho\partial_{rr}\rho\\%\ρ相关
				&\qquad\qquad+\frac{r^2\rho}{2} (\partial_{rr}v)^2+2\rho(\partial_{r}v)^2+\frac{2\rho}{r^2}v^2+2r\rho\partial_{rr}v\partial_{r}v-2\rho\partial_{rr}v v-\frac{4\rho v}{r}\partial_{r}v\\%v相关
				&\qquad\qquad+\tau\rho(\frac{r^2}{6\mu}(\partial_{rr}\tilde{S}_1)^2+\frac{2}{3\mu}(\partial_{r}\tilde{S}_1)^2+\frac{6}{\mu r^2}\tilde{S}_1^2 +\frac{2r}{3\mu}\partial_{rr}\tilde{S}_1 \partial_{r}\tilde{S}_1-\frac{2}{\mu}\partial_{rr}\tilde{S}_1\tilde{S}_1\\%S_1相关
				&\qquad\qquad-\frac{4}{\mu r}\partial_{r}\tilde{S}_1 \tilde{S}_1 +\frac{r^2}{2\lambda}(\partial_{rr} \tilde{S}_2)^2+\frac{2}{\lambda}(\partial_{r} \tilde{S}_2)^2+\frac{2r}{\lambda}\partial_{rr} \tilde{S}_2\partial_{r} \tilde{S}_2)\bigg)\mathrm{d}r\\%S_2相关
				&+\int_0^t\int_{1}^{+\infty}\bigg(\frac{r^2}{3\mu}(\partial_{rr}\tilde{S}_1)^2+\frac{4}{3\mu}(\partial_{r}\tilde{S}_1)^2+\frac{12}{\mu r^2}(\tilde{S}_1)^2 +\frac{4r}{3\mu}\partial_{rr}\tilde{S}_1 \partial_{r}\tilde{S}_1-\frac{4}{\mu}\partial_{rr}\tilde{S}_1\tilde{S}_1\\
				&\qquad\qquad\qquad-\frac{8}{\mu r}\partial_{r}\tilde{S}_1 \tilde{S}_1
				+\frac{r^2}{\lambda}(\partial_{rr}\tilde{S}_2)^2+\frac{4}{\lambda}(\partial_{r}  \tilde{S}_2)^2 +\frac{4r}{\lambda}\partial_{rr} \tilde{S}_2\partial_{r} \tilde{S}_2\bigg)\mathrm{d}r\mathrm{d}t\\
				&+\int_{0}^{t}\left(\frac{3\tau\epsilon}{16\mu}(\partial_{rr}\tilde{S}_1)^2(t,1)
				+\frac{9\tau\epsilon}{32\lambda}(\partial_{rr}\tilde{S}_2)^2(t,1)\right)\mathrm{d}t\\
				&\leq C \left(E(0)+E^{\frac{3}{2}}(t)+E^{\frac{1}{2}}(t) \int_{0}^{t}\mathcal{D}(s)\mathrm{d}s
				+\int_{0}^{t}((\partial_{rr}v)^2(t,1)+(\partial_{tr}v)^2(t,1))\mathrm{d}t\right).
				\\ 
			\end{aligned}
		\end{equation}
		Further, we observe that 
		\begin{equation}\nonumber
			\begin{aligned}
				&\int_{1}^{+\infty}(\frac{r^2\rho}{2} (\partial_{rr}v)^2+2\rho(\partial_{r}v)^2+\frac{2\rho}{r^2}v^2+2r\rho\partial_{rr}v\partial_{r}v-2\rho\partial_{rr}v v-\frac{4\rho v}{r}\partial_{r}v)\mathrm{d}r\\
				=&\int_{1}^{+\infty}(\frac{r^2\rho}{2} (\partial_{rr}v)^2+4\rho(\partial_{r}v)^2
				+2r\rho\partial_{rr}v\partial_{r}v
				+2\partial_{r}\rho\partial_{r}v v+\frac{2}{r}\partial_{r}\rho v^2)\mathrm{d}r\\
				\geq&\int_{1}^{+\infty}(\frac{r^2\rho}{2} (\partial_{rr}v)^2+4\rho(\partial_{r}v)^2
				+2r\rho\partial_{rr}v\partial_{r}v)\mathrm{d}r-CE^{\frac{1}{2}}(t) \mathcal{D}(t),
			\end{aligned}
		\end{equation}
		where we used the boundary condition (\ref{1.6}). 
		Moreover, we know
		\begin{equation}\nonumber
			\begin{aligned}
				&\int_{1}^{+\infty}(\frac{r^2\rho}{2} (\partial_{rr}v)^2+4\rho(\partial_{r}v)^2
				+2r\rho\partial_{rr}v\partial_{r}v)\mathrm{d}r\\
				\geq&\int_{1}^{+\infty}(\frac{r^2\rho}{2} (\partial_{rr}v)^2+4\rho(\partial_{r}v)^2
				-\frac{3r^2\rho}{8}(\partial_{rr}v)^2-\frac{8}{3}\rho(\partial_{r}v)^2)\mathrm{d}r
				=\int_{1}^{+\infty}(\frac{r^2\rho}{8} (\partial_{rr}v)^2+\frac{4}{3}\rho(\partial_{r}v)^2)\mathrm{d}r.
			\end{aligned}
		\end{equation}
		Similarly, we have 
		\begin{equation}\nonumber
			\begin{aligned}
				&\int_{1}^{+\infty}(\frac{r^2}{6\mu}(\partial_{rr}\tilde{S}_1)^2+\frac{2}{3\mu}(\partial_{r}\tilde{S}_1)^2+\frac{6}{\mu r^2}(\tilde{S}_1)^2 +\frac{2r}{3\mu}\partial_{rr}\tilde{S}_1 \partial_{r}\tilde{S}_1-\frac{2}{\mu}\partial_{rr}\tilde{S}_1\tilde{S}_1-\frac{4}{\mu r}\partial_{r}\tilde{S}_1 \tilde{S}_1)\mathrm{d}r\\
				\geq&\int_{1}^{+\infty}(\frac{r^2}{24\mu}(\partial_{rr}\tilde{S}_1)^2+\frac{16}{9\mu}(\partial_{r}\tilde{S}_1)^2+\frac{4}{\mu r^2}(\tilde{S}_1)^2 
				)\mathrm{d}r-CE^{\frac{1}{2}}(t) \mathcal{D}(t).
			\end{aligned}
		\end{equation}
		
		Thus, equation (\ref{3.85}) can be simplified to 
		\begin{equation}\label{3.87}
			\begin{aligned}
				&\int_{1}^{+\infty}\bigg(\frac{ P^{\prime}(\rho)}{2\rho}\left(r\partial_{rr} \rho+r\partial_{r} \rho\right)^2
				+\frac{r^2\rho}{8} (\partial_{rr}v)^2+\frac{4}{3}\rho(\partial_{r}v)^2
				\\
				&\qquad\qquad+\tau\rho\left(\frac{r^2}{24\mu}(\partial_{rr}\tilde{S}_1)^2+\frac{16}{9\mu}(\partial_{r}\tilde{S}_1)^2+\frac{4}{\mu r^2}(\tilde{S}_1)^2\right)+\frac{\tau\rho}{2\lambda}\left(r\partial_{rr} \tilde{S}_2+2\partial_{r} \tilde{S}_2\right)^2\bigg)\mathrm{d}r\\%S_2相关
				&+\int_0^t\int_{1}^{+\infty}\bigg(\frac{r^2}{12\mu}(\partial_{rr}\tilde{S}_1)^2+\frac{32}{9\mu}(\partial_{r}\tilde{S}_1)^2+\frac{8}{\mu r^2}(\tilde{S}_1)^2
				+\frac{1}{\lambda}\left(r\partial_{rr} \tilde{S}_2+2\partial_{r} \tilde{S}_2\right)^2\bigg)\mathrm{d}r\mathrm{d}t\\
				&
				+\int_{0}^{t}\left(\frac{\tau\epsilon}{4\mu}(\partial_{rr}\tilde{S}_1)^2(t,1)
				+\frac{3\tau\epsilon}{8\lambda}(\partial_{rr}\tilde{S}_2)^2(t,1)\right)\mathrm{d}t\\
				&\leq C \left(E(0)+E^{\frac{3}{2}}(t)+E^{\frac{1}{2}}(t) \int_{0}^{t}\mathcal{D}(s)\mathrm{d}s
				+\int_{0}^{t}((\partial_{rr}v)^2(t,1)+(\partial_{tr}v)^2(t,1))\mathrm{d}t\right).
				\\ 
			\end{aligned}
		\end{equation}
		Then, we can get the desired result.
	\end{proof}    
	
	To address the boundary term on the right-hand side of \eqref{3.87}, we now establish the following key lemmas.
	
	First, we give estimates of $\partial_{tr} v$ on the boundary.
	\begin{lemma}\label{lem7_tx}
		There exists a constant $C$ such that for any $0\leq t \leq T$
		\begin{equation}\label{bdy_tx}
			\begin{aligned}
				\int_{0}^{t}(\partial_{tr} v)^2(t,1)\mathrm{d}t
				\leq C\left(E(0)+E^{\frac{3}{2}}(t)+E^{\frac{1}{2}}(t) \int_{0}^{t}\mathcal{D}(s)\mathrm{d}s\right).
			\end{aligned}
		\end{equation}
		
	\end{lemma}
	\begin{proof}
		Multiplying  $(\ref{47})_3$ by $\dfrac{r^2}{3\mu}\partial_{tr} v$  and integrating the result, we get
		\begin{align}\label{76}
				&\int_{0}^{t}\int_{1}^{+\infty}\frac{r^2\tau}{3\mu}\rho\partial_{tr} v\partial_{ttr} \tilde{S}_1\mathrm{d}r\mathrm{d}t
				+\int_{0}^{t}\int_{1}^{+\infty}\frac{r^2\tau}{3\mu}\rho (v-\epsilon)\partial_{tr} v\partial_{trr} \tilde{S}_1\mathrm{d}r\mathrm{d}t
				\nonumber \\
				&+\int_{0}^{t}\int_{1}^{+\infty}\frac{r^2}{3\mu}\partial_{tr} v\partial_{tr} \tilde{S}_1\mathrm{d}r\mathrm{d}t-\int_{0}^{t}\int_{1}^{+\infty}\frac{2}{3}\left(r^2\partial_{trr}v \partial_{tr}v-r(\partial_{tr}v)^2 +\partial_{t}v\partial_{tr}v\right)\mathrm{d}r\mathrm{d}t \nonumber \\
				\leq& C\left(E(0)+ E^{\frac{1}{2}}(t) \int_{0}^{t}\mathcal{D}(s)\mathrm{d}s\right)
				.
		\end{align}	
		We estimate each term of the above equation, respectively. Firstly, using the momentum equation $(\ref{15})_2$, we have
		\begin{align}\label{77}
			%\begin{aligned}
				&\int_{0}^{t}\int_{1}^{+\infty}\dfrac{r^2\tau}{3\mu}\rho\partial_{tr} v\partial_{ttr} \tilde{S}_1\mathrm{d}r\mathrm{d}t \nonumber \\
				=&
				\int_{1}^{+\infty}\dfrac{r^2\tau}{3\mu}\rho\partial_{tr} v \partial_{tr} \tilde{S}_1\big|_0^{t}\mathrm{d}r-
				\int_{0}^{t}\int_{1}^{+\infty}\dfrac{r^2\tau}{3\mu}\rho\partial_{ttr} v \partial_{tr} \tilde{S}_1\mathrm{d}r\mathrm{d}t-
				\int_{0}^{t}\int_{1}^{+\infty}\dfrac{r^2\tau}{3\mu}\partial_{t}\rho
				\partial_{tr} v \partial_{tr} \tilde{S}_1\mathrm{d}r\mathrm{d}t \nonumber \\
				\geq& \int_{1}^{+\infty}\dfrac{r^2\tau}{3\mu}\rho\partial_{tr} v\partial_{tr} \tilde{S}_1\mathrm{d}r-
				\int_{0}^{t}\int_{1}^{+\infty}\dfrac{r^2\tau}{2\mu}\rho\partial_{ttr} v(\rho\partial_{tt} v+\rho v\partial_{tr} v+\partial_{tr} P-\dfrac{2}{r} \partial_t \tilde{S}_1
				-\partial_{tr} \tilde{S}_2)\mathrm{d}r\mathrm{d}t \nonumber \\
				&-C\left(E(0)+ E^{\frac{1}{2}}(t) \int_{0}^{t}\mathcal{D}(s)\mathrm{d}s\right)
			\end{align}
		%\end{equation}
		where 
		\begin{equation}\nonumber
			\begin{aligned}
				\int_{1}^{+\infty}\dfrac{r^2\tau}{3\mu}\rho\partial_{tr} v\partial_{tr} \tilde{S}_1\mathrm{d}r\geq&-C\int_{1}^{+\infty}r^2\rho ((\partial_{tr} v)^2+\tau(\partial_{tr} \tilde{S}_1)^2)\mathrm{d}r\\
				\geq&-C\left(E(0)+ E^{\frac{3}{2}}(t)+ E^{\frac{1}{2}}(t) \int_{0}^{t}\mathcal{D}(s)\mathrm{d}s\right),
			\end{aligned}
		\end{equation}
		\begin{equation}\nonumber
			\begin{aligned}
				&-\int_{0}^{t}\int_{1}^{+\infty}\dfrac{r^2\tau}{2\mu}\rho^2\partial_{ttr} v\partial_{tt} v\mathrm{d}r\mathrm{d}t\\
				=&\int_{0}^{t}\int_{1}^{+\infty}\dfrac{\tau}{4\mu}(2r\rho^2+2r^2\rho
				\partial_{r}\rho)(\partial_{tt}v)^2\mathrm{d}r\mathrm{d}t
				\geq-C\left(E(0)+ E^{\frac{1}{2}}(t) \int_{0}^{t}\mathcal{D}(s)\mathrm{d}s\right),
			\end{aligned}
		\end{equation}
		\begin{align*}
		&-\int_{0}^{t}\int_{1}^{+\infty}\dfrac{r^2\tau}{2\mu}\rho^2v\partial_{ttr} v\partial_{tr} v\mathrm{d}r\mathrm{d}t\\
		=&-\int_{1}^{+\infty}\dfrac{\tau}{4\mu}r^2\rho^2v(\partial_{tr} v)^2\big|_0^t\mathrm{d}r 
		+\int_{0}^{t}\int_{1}^{+\infty}\dfrac{r^2\tau}{4\mu}\partial_{t}(\rho^2v)
		(\partial_{tr} v)^2\mathrm{d}r\mathrm{d}t\\
		\geq&-C\left(E(0)+ E^{\frac{3}{2}}(t)+ E^{\frac{1}{2}}(t) \int_{0}^{t}\mathcal{D}(s)\mathrm{d}s\right),
		\end{align*}
		and 
			\begin{align*}
				&-\int_{0}^{t}\int_{1}^{+\infty}\dfrac{\tau r^2}{2\mu}\rho\partial_{ttr} v\partial_{tr} P\mathrm{d}r\mathrm{d}t\\
				=&-\int_{1}^{+\infty}\dfrac{\tau r^2} {2\mu} \rho\partial_{tr} v\partial_{tr} P\big|_0^t\mathrm{d}r 
				+\int_{0}^{t}\int_{1}^{+\infty}\dfrac{\tau r^2}{2\mu}\rho\partial_{tr} v \partial_{ttr} P \mathrm{d}r \mathrm{d}t
				+\int_{0}^{t}\int_{1}^{+\infty}\dfrac{\tau r^2}{2\mu}\partial_{tr} v \partial_{t}\rho\partial_{tr} P \mathrm{d}r\mathrm{d}t\\
				\geq&-C\int_{1}^{+\infty} r^2((\partial_{tr} v)^2+(\partial_{tr} \rho)^2)\mathrm{d}r +\int_{0}^{t}\int_{1}^{+\infty}\dfrac{\tau r^2}{2\mu} \rho P^{\prime}(\rho) \partial_{ttr}\rho \partial_{tr} v\mathrm{d}r\mathrm{d}t\\&
				-C\left(E(0)+ E^{\frac{1}{2}}(t) \int_{0}^{t}\mathcal{D}(s)\mathrm{d}s\right)\\
				\geq&-\int_{0}^{t}\int_{1}^{+\infty}\dfrac{\tau r^2}{2\mu} \rho P^{\prime}(\rho)
				\partial_{tr}v(\partial_{trr}(\rho v)+\dfrac{2}{r} \partial_{tr}(\rho v)-\dfrac{2}{r^2} \partial_{t}(\rho v))\mathrm{d}r\mathrm{d}t\\
				&-C\left(E(0)+ E^{\frac{3}{2}}(t)+ E^{\frac{1}{2}}(t) \int_{0}^{t}\mathcal{D}(s)\mathrm{d}s\right)\\
				\geq&
				-\int_{0}^{t}\int_{1}^{+\infty}\dfrac{\tau r^2}{2\mu}\rho P^{\prime}(\rho)
				v\partial_{tr}v\partial_{trr}\rho\mathrm{d}r\mathrm{d}t
				-C\left(E(0)+ E^{\frac{3}{2}}(t)+ E^{\frac{1}{2}}(t) \int_{0}^{t}\mathcal{D}(s)\mathrm{d}s\right).
			\end{align*}
		
		Note that there is still a three-order term $\int_{0}^{t}\int_{1}^{+\infty}\frac{\tau r^2}{2\mu}\rho P^{\prime}(\rho) v\partial_{tr}v \partial_{trr}\rho \mathrm{d}r\mathrm{d}t$ in the above inequality, which can be canceled out by the convective term, as shown in \eqref{103}. 
		
	Next, we handle the remaining two terms in inequality \eqref{77}. Applying Lemma \ref{lem2} and Lemmas \ref{lem3}-\ref{lem5}, we have
		
		\begin{equation}\nonumber
			\begin{aligned}
				&\int_{0}^{t}\int_{1}^{+\infty}\dfrac{\tau r}{\mu}\rho \partial_{ttr}v \partial_{t}\tilde{S}_1\mathrm{d}r\mathrm{d}t\\
				=&
				\int_{1}^{+\infty}\dfrac{\tau r}{\mu}\rho \partial_{tr}v \partial_{t}\tilde{S}_1\big|_0^{t}\mathrm{d}r
				-\int_{0}^{t}\int_{1}^{+\infty}\dfrac{\tau r}{\mu}\rho \partial_{tr}v \partial_{tt}\tilde{S}_1\mathrm{d}r\mathrm{d}t
				-\int_{0}^{t}\int_{1}^{+\infty}\dfrac{\tau r}{\mu}\partial_{t}\rho \partial_{tr}v \partial_{t}\tilde{S}_1\mathrm{d}r\mathrm{d}t\\
				\geq&
				-C\int_{1}^{+\infty}r^2\rho((\partial_{tr}v)^2+(\partial_{t}\tilde{S}_1)^2)\mathrm{d}r
				-C\int_{0}^{t}\int_{1}^{+\infty}r^2\rho((\partial_{tr}v)^2+(\partial_{tt}\tilde{S}_1)^2)\mathrm{d}r\mathrm{d}t\\&
				-C\left(E(0)+ E^{\frac{1}{2}}(t) \int_{0}^{t}\mathcal{D}(s)\mathrm{d}s\right)\\
				\geq&-C\left(E(0)+ E^{\frac{3}{2}}(t)+ E^{\frac{1}{2}}(t) \int_{0}^{t}\mathcal{D}(s)\mathrm{d}s\right).
			\end{aligned}
		\end{equation}
		In view of equation $(\ref{47})_4$ and Lemmas \ref{lem3.4}, \ref{lem4} and \ref{lem5}, we have	
		\begin{equation}\nonumber
			\begin{aligned}
				&\int_{0}^{t}\int_{1}^{+\infty}\dfrac{\tau r^2}{2\mu}\rho \partial_{ttr}v \partial_{tr}\tilde{S}_2\mathrm{d}r\mathrm{d}t\\
				=&
				\int_{1}^{+\infty}\dfrac{\tau r^2}{2\mu}\rho \partial_{tr}v \partial_{tr}\tilde{S}_2\big|_0^t\mathrm{d}r
				-\int_{0}^{t}\int_{1}^{+\infty}\dfrac{\tau r^2}{2\mu}\rho \partial_{tr}v \partial_{ttr}\tilde{S}_2\mathrm{d}r\mathrm{d} t
				-\int_{0}^{t}\int_{1}^{+\infty}\dfrac{\tau r^2}{2\mu}\partial_{t}\rho \partial_{tr}v \partial_{t}\tilde{S}_2\mathrm{d}r\mathrm{d} t\\
				\geq &
				-C\int_{1}^{+\infty}r^2\rho ((\partial_{tr}v)^2+ (\partial_{tr}\tilde{S}_2)^2)\mathrm{d}r
				-C(E(0)+E^{\frac{1}{2}}(t) \int_{0}^{t}\mathcal{D}(s)\mathrm{d}s)\\
				&-\int_{0}^{t}\int_{1}^{+\infty}\dfrac{r^2}{2\mu}\partial_{tr}v
				\left(\lambda(\partial_{trr} v+\dfrac{2}{r}\partial_{tr} v
				-\dfrac{2}{r^2} \partial_{t}v)-\partial_{tr} \tilde{S}_2
				-\tau \rho (v-\epsilon)\rho\partial_{trr}\tilde{S}_2\right)\mathrm{d}r\mathrm{d} t
				\\
				\geq &
				\int_{0}^{t}\frac{\lambda}{4\mu} (\partial_{tr}v)^2(t,1)\mathrm{d} t
				-\int_{0}^{t}\int_{1}^{+\infty}\dfrac{\lambda r}{2\mu} (\partial_{tr}v)^2\mathrm{d}r\mathrm{d} t
				+\int_{0}^{t}\int_{1}^{+\infty}\dfrac{\lambda}{\mu}  \partial_{t}v \partial_{tr}v \mathrm{d}r\mathrm{d} t
				\\
				&+\int_{0}^{t}\int_{1}^{+\infty}\dfrac{r^2}{\mu} \partial_{tr}v\partial_{tr} \tilde{S}_2 \mathrm{d}r\mathrm{d} t
				+\int_{0}^{t}\int_{1}^{+\infty}\dfrac{r^2\tau}{2\mu} \rho(v-\epsilon) \partial_{tr}v\partial_{trr} \tilde{S}_2 \mathrm{d}r\mathrm{d} t\\
				&-C\left(E(0)+ E^{\frac{3}{2}}(t)+ E^{\frac{1}{2}}(t) \int_{0}^{t}\mathcal{D}(s)\mathrm{d}s\right)\\
				\geq &
				\int_{0}^{t}\int_{1}^{+\infty}\dfrac{r^2\tau}{2\mu} \rho(v-\epsilon) \partial_{tr}v\partial_{trr} \tilde{S}_2 \mathrm{d}r\mathrm{d}t
				-C\left(E(0)+ E^{\frac{3}{2}}(t)+ E^{\frac{1}{2}}(t) \int_{0}^{t}\mathcal{D}(s)\mathrm{d}s\right)
			\end{aligned}
		\end{equation}
		
		Finally, we have
		\begin{equation}
			\begin{aligned}
				&\int_{0}^{t}\int_{1}^{+\infty}\dfrac{r^2\tau}{3\mu}\rho\partial_{tr} v\partial_{ttr} \tilde{S}_1\mathrm{d}r\mathrm{d}t\\
				\geq&
				\int_{0}^{t}\frac{\lambda}{4\mu} (\partial_{tr}v)^2(t,1)\mathrm{d} t
				+\int_{0}^{t}\int_{1}^{+\infty}\dfrac{r^2\tau}{2\mu} \rho(v-\epsilon) \partial_{tr}v\partial_{trr} \tilde{S}_2 \mathrm{d}r\mathrm{d}t\\
				&-\int_{0}^{t}\int_{1}^{+\infty}\dfrac{\tau r^2}{2\mu}\rho P^{\prime}(\rho)
				v\partial_{tr}v\partial_{trr}\rho\mathrm{d}r\mathrm{d}t
				-C\left(E(0)+ E^{\frac{3}{2}}(t)+ E^{\frac{1}{2}}(t) \int_{0}^{t}\mathcal{D}(s)\mathrm{d}s\right).	
			\end{aligned}
		\end{equation}	
		
		For the second term on the left-hand side of equation (\ref{76}), by applying equation $(\ref{47})_2$, we get
		\begin{align}\label{103}
				&\int_{0}^{t}\int_{1}^{+\infty}\dfrac{r^2\tau}{3\mu}\rho (v-\epsilon)\partial_{tr} v\partial_{trr} \tilde{S}_1\mathrm{d}r\mathrm{d}t \nonumber \\
				=&\int_{0}^{t}\int_{1}^{+\infty}\dfrac{r^2\tau}{2\mu}\rho (v-\epsilon)\partial_{tr} v(\rho\partial_{ttr} v+\rho v\partial_{trr} v+\partial_{trr} P-\dfrac{2}{r} \partial_{tr} \tilde{S}_1+\dfrac{2}{r^2} \partial_{t} \tilde{S}_1
				-\partial_{trr} \tilde{S}_2-g_2)
				\mathrm{~d}r\mathrm{d}t \nonumber\\
				\geq& 
				\int_{0}^{t}\int_{1}^{+\infty}\dfrac{r^2\tau}{2\mu}\rho P^{\prime}(\rho)(v-\epsilon) \partial_{tr}v\partial_{trr}\rho \mathrm{d}r\mathrm{d}t
				-\int_{0}^{t}\int_{1}^{+\infty}\dfrac{r^2\tau}{2\mu}\rho (v-\epsilon) \partial_{tr}v\partial_{trr} \tilde{S}_2 \mathrm{d}r\mathrm{d}t \nonumber \\ 
				&
				-C\left(E(0)+ E^{\frac{3}{2}}(t)+ E^{\frac{1}{2}}(t) \int_{0}^{t}\mathcal{D}(s)\mathrm{d}s\right).
		\end{align}
		
		For the third term on the left-hand side of equation (\ref{76}), we get
		\begin{equation}
			\begin{aligned}
				\int_{0}^{t}\int_{1}^{+\infty}\dfrac{r^2}{3\mu}\partial_{tr} v\partial_{tr} \tilde{S}_1\mathrm{d}r\mathrm{d}t
				&\geq -C\int_{0}^{t}\int_{1}^{+\infty} r^2((\partial_{tr} v)^2
				+(\partial_{tr} \tilde{S}_1)^2) \mathrm{d}r\mathrm{d}t\\
				&\geq -C\left(E(0)+ E^{\frac{3}{2}}(t)+ E^{\frac{1}{2}}(t) \int_{0}^{t}\mathcal{D}(s)\mathrm{d}s\right).
			\end{aligned}
		\end{equation}
		
		For the last term on the left-hand side of equation (\ref{76}), we get
		\begin{equation}
			\begin{aligned}
				&-\int_{0}^{t}\int_{1}^{+\infty}\frac{2}{3}\left(r^2\partial_{trr}v \partial_{tr}v-r(\partial_{tr}v)^2 +\partial_{t}v\partial_{tr}v\right)\mathrm{d}r\mathrm{d}t\\
				&\geq \int_{0}^{t}\frac{1}{3}(\partial_{tr} v)^2(t,1)\mathrm{d}t- C\left(E(0)+ E^{\frac{3}{2}}(t)+ E^{\frac{1}{2}}(t) \int_{0}^{t}\mathcal{D}(s)\mathrm{d}s\right).
			\end{aligned}
		\end{equation}
		
		Combining the above results, we have
		\begin{equation}\nonumber
			\begin{aligned}
				\int_{0}^{t}\frac{1}{3}(\partial_{tr} v)^2(t,1)\mathrm{d}t
				\leq
				\int_{0}^{t}\int_{1}^{+\infty}\dfrac{r^2\tau\epsilon}{2\mu}\rho P^{\prime}(\rho)\partial_{tr}v\partial_{trr}\rho \mathrm{d}r\mathrm{d}t
				+C\left(E(0)+ E^{\frac{3}{2}}(t)+ E^{\frac{1}{2}}(t) \int_{0}^{t}\mathcal{D}(s)\mathrm{d}s\right).
			\end{aligned}
		\end{equation}
		
		%\begin{equation}
		%	\begin{aligned}
			%		&\int_{0}^{t}\int_{1}^{+\infty}\dfrac{r^2\tau\epsilon}{2\mu}\rho P^{\prime}(\rho)\partial_{tr}v\partial_{trr}\rho \mathrm{d}r\mathrm{d}t\\
			%		=&\int_{0}^{t}\int_{1}^{+\infty}-\dfrac{r^2\tau\epsilon}{2\mu} P^{\prime}(\rho)\partial_{trr}\rho 
			%		(\partial_{tt} \rho+v\partial_{tr}\rho+\partial_{t}\rho\partial_{r}v+\partial_{r}\rho\partial_{t}v+\frac{2}{r}\partial_{t}(\rho v))\mathrm{d}r\mathrm{d}t\\
			%		\leq&
			%		\\
			%		& \leq C\left(E(0)+E^{\frac{3}{2}}(t) +E^{\frac{1}{2}}(t) \int_{0}^{t}\mathcal{D}(s)\mathrm{d}s\right)
			%		,
			%	\end{aligned}
		%\end{equation}
		In view of equation $(\ref{47})_1$ and Lemmas \ref{lem3.4}, \ref{lem4} and \ref{lem5}, we derive
		\begin{equation}\nonumber
			\begin{aligned}
				&\int_{0}^{t}\int_{1}^{+\infty}\dfrac{r^2\tau\epsilon}{2\mu}\rho P^{\prime}(\rho)\partial_{tr}v\partial_{trr}\rho \mathrm{d}r\mathrm{d}t\\
				=&-\int_{0}^{t}\dfrac{r^2\tau\epsilon}{2\mu}\rho P^{\prime}(\rho)\partial_{tr}v\partial_{tr}\rho(t,1)\mathrm{d}t-
				\int_{0}^{t}\int_{1}^{+\infty}\dfrac{\tau\epsilon}{2\mu}[r^2\rho P^{\prime}(\rho)\partial_{tr}v]_r\partial_{tr}\rho \mathrm{d}r\mathrm{d}t\\
				\leq&-\int_{0}^{t}\dfrac{\tau\epsilon}{2\mu}\rho P^{\prime}(\rho)\partial_{tr}v\partial_{tr}\rho(t,1)\mathrm{d}t
				+\int_{0}^{t}\int_{1}^{+\infty}\dfrac{\tau\epsilon}{2\mu}r^2P^{\prime}(\rho)\partial_{tr}\rho(\partial_{ttr} \rho+v\partial_{trr}\rho+\frac{2}{r}\rho\partial_{tr}v-\frac{2}{r^2}\rho \partial_{t}v)\mathrm{d}r\mathrm{d}t\\
				&-
				\int_{0}^{t}\int_{1}^{+\infty}\dfrac{\tau\epsilon}{2\mu}2r\rho P^{\prime}(\rho)\partial_{tr}v\partial_{tr}\rho \mathrm{d}r\mathrm{d}t
				+ C\left(E(0)+E^{\frac{1}{2}}(t) \int_{0}^{t}\mathcal{D}(s)\mathrm{d}s\right)
				\\
				\leq&\dfrac{C\tau\epsilon}{2\mu}\int_{0}^{t}\left((\partial_{tr}v)^2+(\partial_{tr}\rho)^2\right)(t,1) \mathrm{d}t
				+\dfrac{C\tau\epsilon}{\mu}\int_{0}^{t}\int_{1}^{+\infty}\left((\partial_{t}v)^2+(\partial_{tr}v)^2+(\partial_{tr}\rho)^2\right)
				\mathrm{d}r\mathrm{d}t
				\\
				&+ C\left(E(0)+E^{\frac{3}{2}}(t) +E^{\frac{1}{2}}(t) \int_{0}^{t}\mathcal{D}(s)\mathrm{d}s\right)\\
				\leq&\dfrac{C\tau\epsilon}{2\mu}\int_{0}^{t}(\partial_{tr}v)^2(t,1) \mathrm{d}t
				+ C\left(E(0)+E^{\frac{3}{2}}(t) +E^{\frac{1}{2}}(t) \int_{0}^{t}\mathcal{D}(s)\mathrm{d}s\right),
			\end{aligned}
		\end{equation}
		where we use equation $(\ref{15})_2$, leading to
		\begin{equation}\nonumber
			\begin{aligned}	
				\dfrac{C\tau\epsilon}{2\mu}\int_{0}^{t}(\partial_{tr}\rho)^2(t,1) \mathrm{d}t
				&\leq \dfrac{C\tau\epsilon}{2\mu}\int_{0}^{t}((\partial_{tr}\tilde{S}_1)^2+(\partial_{tr}\tilde{S}_2)^2+(\partial_{t}\tilde{S}_1)^2)(t,1)\mathrm{d}t\\
				&\leq C\left(E(0)+E^{\frac{3}{2}}(t) +E^{\frac{1}{2}}(t) \int_{0}^{t}\mathcal{D}(s)\mathrm{d}s\right).
			\end{aligned}
		\end{equation}
		Thus, we can obtain
		\begin{equation}\label{89}
			\begin{aligned}
				\int_{0}^{t}(\partial_{tr} v)^2(t,1)\mathrm{d}t
				\leq C\left(E(0)+E^{\frac{3}{2}}(t)+E^{\frac{1}{2}}(t) \int_{0}^{t}\mathcal{D}(s)\mathrm{d}s\right).
			\end{aligned}
		\end{equation}
	This finish the proof of Lemma \ref{lem7_tx}.
	\end{proof}                                                                                                                                                                                                       
	\begin{lemma}\label{lem7_xx}
		There exists a constant $C$ such that for any $0\leq t \leq T$
		\begin{align}\label{bdy}
				&\int_{0}^{t}(\partial_{rr} v)^2(t,1)\mathrm{~d}t \nonumber\\
				\leq& \frac{\tau C}{2\mu}(\eta+\epsilon)\int_{1}^{+\infty}r^2\rho(\partial_{rr} v)^2\mathrm{d}r
				+\eta\int_{0}^{t}\int_{1}^{+\infty}\frac{r^2}{2\mu}\left((\partial_{rr}\tilde{S}_1)^2+(\partial_{rr}\tilde{S}_2)^2\right)\mathrm{d}r\mathrm{d}t \nonumber\\
				&
				+\frac{C\tau\epsilon}{\mu}\int_{0}^{t}\eta\left(
				(\partial_{rr} \tilde{S}_1)^2+(\partial_{rr} \tilde{S}_2)^2\right)(t,1)\mathrm{d}t
				+C(\eta)\epsilon\int_{0}^{t}\int_{1}^{+\infty}r^2((\partial_{rr}\tilde{S}_1)^2+(\partial_{rr}\tilde{S}_2)^2)\mathrm{d}r\mathrm{d}t \nonumber\\
				&
				+C\left(E(0)+ E^{\frac{3}{2}}(t)+ E^{\frac{1}{2}}(t) \int_{0}^{t}\mathcal{D}(s)\mathrm{d}s\right)
			\end{align}
		where $\eta$ is any small constant to be chosen later.
	\end{lemma}
	\begin{proof}
		Multiplying the equation $(\ref{60})_3$ by $\frac{r^2}{3\mu}\partial_{rr} v$
		and integrating over $(0,t)\times[1,+\infty)$, we get
		\begin{align}\label{70}
				&\int_{0}^{t}\int_{1}^{+\infty}\frac{r^2\tau}{3\mu}\rho\partial_{rr} v\partial_{trr} \tilde{S}_1\mathrm{d}r\mathrm{d}t +\int_{0}^{t}\int_{1}^{+\infty}\dfrac{r^2\tau}{3\mu}\rho (v-\epsilon) \partial_{rr} v\partial_{rrr} \tilde{S}_1\mathrm{d}r\mathrm{d}t\nonumber \\
				& +\int_{0}^{t}\int_{1}^{+\infty}\frac{r^2}{3\mu}\partial_{rr} v\partial_{rr} \tilde{S}_1\mathrm{d}r\mathrm{d}t 
				-\int_{0}^{t}\int_{1}^{+\infty}\frac{2}{3}\left(r^2\partial_{rrr}v \partial_{rr}v-r(\partial_{rr} v)^2+2\partial_{r}v\partial_{rr}v-\frac{2}{r}v\partial_{rr} v\right)\mathrm{d}r\mathrm{d}t \nonumber\\
				&\leq C\left(E(0)+ E^{\frac{1}{2}}(t) \int_{0}^{t}\mathcal{D}(s)\mathrm{d}s\right).
		\end{align}	
		We estimate each term of the above equation, respectively. Firstly, by using equation $(\ref{22})_2$, we have
		\begin{align}\label{71}
				&\int_{0}^{t}\int_{1}^{+\infty}\dfrac{r^2\tau}{3\mu}\rho\partial_{rr} v\partial_{trr} \tilde{S}_1\mathrm{d}r\mathrm{d}t \nonumber \\
				=&
				\int_{1}^{+\infty}\dfrac{r^2\tau}{3\mu}\rho\partial_{rr} v\partial_{rr} \tilde{S}_1\big|_0^t\mathrm{d}r-
				\int_{0}^{t}\int_{1}^{+\infty}\dfrac{r^2\tau}{3\mu}\rho\partial_{trr} v\partial_{rr} \tilde{S}_1\mathrm{d}r\mathrm{d}t-
				\int_{0}^{t}\int_{1}^{+\infty}\dfrac{r^2\tau}{3\mu}\partial_{t}\rho\partial_{rr} v\partial_{rr} \tilde{S}_1\mathrm{d}r\mathrm{d}t \nonumber \\
				\geq& \int_{1}^{+\infty}\dfrac{r^2\tau}{3\mu}\rho\partial_{rr} v\partial_{rr} \tilde{S}_1\mathrm{d}r-
				\int_{0}^{t}\int_{1}^{+\infty}\dfrac{r^2\tau}{2\mu}\rho\partial_{trr} v(\rho\partial_{tr} v+\rho v\partial_{rr} v+\partial_{rr} P-\dfrac{2}{r} \partial_r \tilde{S}_1
				+\dfrac{2}{r^2}\tilde{S}_1 \nonumber \\
				&\qquad\qquad\qquad\qquad\qquad\qquad\qquad\qquad-\partial_{rr} \tilde{S}_2)\mathrm{d}r\mathrm{d}t-C\left(E(0)+ E^{\frac{1}{2}}(t) \int_{0}^{t}\mathcal{D}(s)\mathrm{d}s\right),
			\end{align}
		where 
		\begin{equation}\nonumber
			\begin{aligned}
				&-\int_{0}^{t}\int_{1}^{+\infty}\dfrac{r^2\tau}{2\mu}\rho^2\partial_{trr} v\partial_{tr} v\mathrm{d}r\mathrm{d}t
				-\int_{0}^{t}\int_{1}^{+\infty}\dfrac{r^2\tau}{2\mu}\rho^2v\partial_{trr} v\partial_{rr} v\mathrm{d}r\mathrm{d}t\\
				\geq&\int_{0}^{t}\frac{\tau}{4\mu}\rho^2(\partial_{tr}v)^2(t,1)\mathrm{d}t
				+\int_{0}^{t}\int_{1}^{+\infty}\dfrac{\tau}{4\mu}(2r\rho^2+2r^2\rho\partial_{r}\rho)(\partial_{tr}v)^2\mathrm{d}r\mathrm{d}t
				-C\left(E(0)+ E^{\frac{1}{2}}(t) \int_{0}^{t}\mathcal{D}(s)\mathrm{d}s\right)\\
				\geq& -C\left(E(0)+ E^{\frac{3}{2}}(t)+ E^{\frac{1}{2}}(t) \int_{0}^{t}\mathcal{D}(s)\mathrm{d}s\right),
			\end{aligned}
		\end{equation}
		and 
		\begin{equation}\nonumber
			\begin{aligned}
				&-\int_{0}^{t}\int_{1}^{+\infty}\dfrac{\tau r^2}{2\mu}\rho\partial_{trr} v\partial_{rr} P\mathrm{d}r\mathrm{d}t\\
				\geq&\int_{1}^{+\infty}\frac{\tau r^2}{2\mu}\rho\partial_{rr} v\partial_{rr} P\mathrm{d}r +\int_{0}^{t}\int_{1}^{+\infty}\dfrac{\tau r^2}{2\mu}\rho\partial_{rr} v P^{\prime}(\rho)\partial_{trr} \rho \mathrm{d}r\mathrm{d}t
				-C\left(E(0)+ E^{\frac{1}{2}}(t) \int_{0}^{t}\mathcal{D}(s)\mathrm{d}s\right).
			\end{aligned}
		\end{equation}	
		By using equation $(\ref{60})_1$, Lemmas \ref{lem3.4} and \ref{lem5}, we have
	
			\begin{align*}
				&\int_{0}^{t}\int_{1}^{+\infty}\dfrac{\tau r^2}{2\mu}\rho\partial_{rr} v P^{\prime}(\rho)\partial_{trr} \rho \mathrm{d}r\mathrm{d}t\\
				=&-\int_{0}^{t}\int_{1}^{+\infty}\frac{\tau r^2}{2\mu}\rho\partial_{rr} v P^{\prime}(\rho)(\partial_{rrr}(\rho v)+\frac{2}{r} \partial_{rr}(\rho v)-\frac{4}{r^2} \partial_{r}(\rho v)+\frac{4}{r^3}\rho v)\mathrm{d}r\mathrm{d}t\\
				\geq&
				\int_{0}^{t}\dfrac{\tau r^2}{4\mu}P^{\prime}(\rho)\rho^2(\partial_{rr} v)^2(t,1)\mathrm{d}t
				-\int_{0}^{t}\int_{1}^{+\infty}\dfrac{\tau r}{2\mu}
				\rho^2 P^{\prime}(\rho)(\partial_{rr} v)^2\mathrm{d}r\mathrm{d}t\\
				&-\int_{0}^{t}\int_{1}^{+\infty}\dfrac{\tau r^2}{2\mu}
				\rho P^{\prime}(\rho)v\partial_{rr} v\partial_{rrr}\rho\mathrm{d}r\mathrm{d}t
				+\int_{0}^{t}\int_{1}^{+\infty}\dfrac{2\tau}{\mu}
				\rho^2 P^{\prime}(\rho)\partial_{rr} v\partial_{r} v\mathrm{d}t\\
				&
				-\int_{0}^{t}\int_{1}^{+\infty}\dfrac{2\tau}{\mu r}
				\rho^2 P^{\prime}(\rho)v\partial_{rr}v\mathrm{d}r\mathrm{d}t
				-C\left(E(0)+E^{\frac{1}{2}}(t) \int_{0}^{t}\mathcal{D}(s)\mathrm{d}s\right)
				\displaybreak \\
				\geq&
				-\int_{0}^{t}\int_{1}^{+\infty}\dfrac{\tau r^2}{2\mu}
				\rho P^{\prime}(\rho)v\partial_{rr} v\partial_{rrr}\rho\mathrm{d}r\mathrm{d}t
				-C\int_{0}^{t}\int_{1}^{+\infty}r^2((\partial_{rr} v)^2+(\partial_{r} v)^2+v^2)\mathrm{d}r\mathrm{d}t\\
				&\qquad\qquad\qquad\qquad\qquad\qquad\qquad\qquad\qquad
				-C\left(E(0)+ E^{\frac{3}{2}}(t)+ E^{\frac{1}{2}}(t) \int_{0}^{t}\mathcal{D}(s)\mathrm{d}s\right)\\
				\geq&-\int_{0}^{t}\int_{1}^{+\infty}\dfrac{\tau r^2}{2\mu}
				\rho P^{\prime}(\rho)v\partial_{rr} v\partial_{rrr}\rho\mathrm{d}r\mathrm{d}t
				-C\epsilon\int_{0}^{t}\int_{1}^{+\infty}r^2((\partial_{rr}\tilde{S}_1)^2+(\partial_{rr}\tilde{S}_2)^2)\mathrm{d}r\mathrm{d}t\\
				&\qquad\qquad\qquad\qquad\qquad\qquad\qquad\qquad\qquad
				-C\left(E(0)+ E^{\frac{3}{2}}(t)+ E^{\frac{1}{2}}(t) \int_{0}^{t}\mathcal{D}(s)\mathrm{d}s\right)
			\end{align*}
		since $P^{\prime}(\rho)>0$. 
		Next, we continue to handle the remaining terms in (\ref{71}). By applying Lemmas \ref{lem4} and  $\ref{lem5}$, we have	 
		\begin{equation}\nonumber
			\begin{aligned}
				&\int_{0}^{t}\int_{1}^{+\infty}\frac{r\tau}{\mu}\rho\partial_{trr}v
				\partial_{r} \tilde{S}_1\mathrm{d}r\mathrm{d}t\\
				=&\int_{1}^{+\infty} \frac{r\tau}{\mu}\rho\partial_{rr}v
				\partial_{r} \tilde{S}_1\big|_0^t\mathrm{d}r
				-\int_{0}^{t}\int_{1}^{+\infty} \frac{r\tau}{\mu}
				\rho\partial_{tr} \tilde{S}_1\partial_{rr}v\mathrm{d}r\mathrm{d}t
				-\int_{0}^{t}\int_{1}^{+\infty} \frac{r\tau}{\mu}
				\partial_{t}\rho\partial_{r} \tilde{S}_1 \partial_{rr}v\mathrm{d}r\mathrm{d}t\\
				\geq&\int_{1}^{+\infty} \frac{r\tau}{\mu}\rho\partial_{rr}v
				\partial_{r} \tilde{S}_1\mathrm{d}r
				-C\int_{0}^{t}\int_{1}^{+\infty} r^2((\partial_{rr}v)^2+(\partial_{tr} \tilde{S}_1)^2)\mathrm{d}r\mathrm{d}t
				-C\left(E(0)+ E^{\frac{1}{2}}(t) \int_{0}^{t}\mathcal{D}(s)\mathrm{d}s\right)\\
				\geq&\int_{1}^{+\infty} \frac{r\tau}{\mu}\rho\partial_{rr}v
				\partial_{r} \tilde{S}_1\mathrm{d}r
				-C\epsilon\int_{0}^{t}\int_{1}^{+\infty}r^2((\partial_{rr} \tilde{S}_1)^2
				+(\partial_{rr} \tilde{S}_2)^2)\mathrm{d}r\mathrm{d}t\\
				&\qquad\qquad\qquad\qquad\qquad\qquad\qquad\qquad\qquad\qquad\qquad\quad-C\left(E(0)+ E^{\frac{3}{2}}(t)+ E^{\frac{1}{2}}(t) \int_{0}^{t}\mathcal{D}(s)\mathrm{d}s\right).
			\end{aligned}
		\end{equation}
		By applying Lemmas \ref{lem2} and  $\ref{lem5}$, we have	
			\begin{align*}
				&-\int_{0}^{t}\int_{1}^{+\infty}\frac{\tau}{\mu}\rho\partial_{trr}v
				\tilde{S}_1\mathrm{d}r\mathrm{d}t\\
				=&-\int_{1}^{+\infty} \frac{\tau}{\mu}\rho\partial_{rr}v \tilde{S}_1\big|_0^t\mathrm{d}r
				+\int_{0}^{t}\int_{1}^{+\infty} \frac{\tau}{\mu}
				\rho \partial_{t}\tilde{S}_1\partial_{rr}v\mathrm{d}r\mathrm{d}t +\int_{0}^{t}\int_{1}^{+\infty} \frac{\tau}{\mu}
				\partial_{t}\rho \partial_{rr}v\tilde{S}_1\mathrm{d}r\mathrm{d}t\\
				\geq&-\int_{1}^{+\infty} \frac{\tau}{\mu}\rho\partial_{rr}v \tilde{S}_1\mathrm{d}r
				-\int_{0}^{t}\int_{1}^{+\infty} \frac{\tau\rho}{2\mu}((\partial_{rr}v)^2
				+(\partial_{t}\tilde{S}_1)^2) \mathrm{d}r
				+\int_{0}^{t}\int_{1}^{+\infty} \frac{\tau}{\mu}
				\partial_{t}\rho \partial_{rr}v\tilde{S}_1\mathrm{d}r\mathrm{d}t\\
				\geq&-\int_{1}^{+\infty} \frac{\tau}{\mu}\rho\partial_{rr}v \tilde{S}_1\mathrm{d}r
				-C\epsilon\int_{0}^{t}\int_{1}^{+\infty}r^2((\partial_{rr} \tilde{S}_1)^2
				+(\partial_{rr} \tilde{S}_2)^2)\mathrm{d}r\mathrm{d}t\\
				&\qquad\qquad\qquad\qquad\qquad\qquad\qquad\qquad\qquad\qquad\qquad-C\left(E(0)+ E^{\frac{3}{2}}(t)+ E^{\frac{1}{2}}(t) \int_{0}^{t}\mathcal{D}(s)\mathrm{d}s\right).
			\end{align*}
	
		In view of equation $(\ref{60})_4$ and Lemmas \ref{lem3.4} and \ref{lem5}, we have
			\begin{align*}
				&\int_{0}^{t}\int_{1}^{+\infty}\frac{r^2\tau}{2\mu}\rho\partial_{trr}v
				\partial_{rr} \tilde{S}_2\mathrm{d}r\mathrm{d}t\\
				=&\int_{1}^{+\infty} \frac{r^2\tau}{2\mu}\rho\partial_{rr}v
				\partial_{rr} \tilde{S}_2\big|_0^t\mathrm{d}r
				-\int_{0}^{t}\int_{1}^{+\infty} \frac{r^2\tau}{2\mu}\rho\partial_{rr}v
				\partial_{trr} \tilde{S}_2\mathrm{d}r\mathrm{d}t
				-\int_{0}^{t}\int_{1}^{+\infty} \frac{r^2\tau}{2\mu}\partial_{t}\rho\partial_{rr}v
				\partial_{rr} \tilde{S}_2\mathrm{d}r\mathrm{d}t\\
				\geq&	\int_{1}^{+\infty} \frac{r^2\tau}{2\mu}\rho\partial_{rr}v
				\partial_{rr} \tilde{S}_2\mathrm{d}r
				-C\left(E(0)+ E^{\frac{1}{2}}(t) \int_{0}^{t}\mathcal{D}(s)\mathrm{d}s\right)\\
				&-\int_{0}^{t}\int_{1}^{+\infty}\frac{r^2}{2\mu}\partial_{rr}v
				(\lambda(\partial_{rrr} v+\dfrac{2}{r}\partial_{rr} v-\dfrac{4}{r^2} \partial_{r}v+\dfrac{4}{r^3}v)-\tau  \rho (v-\epsilon) \partial_{rrr}\tilde{S}_2 
				-\partial_{rr} \tilde{S}_2)\mathrm{d}r\mathrm{d}t \displaybreak \\
				\geq&\int_{1}^{+\infty}\frac{r^2\tau}{2\mu}\rho\partial_{rr}v\partial_{rr} \tilde{S}_2\mathrm{d}r
				+\int_{0}^{t}\int_{1}^{+\infty}\dfrac{r^2\tau}{2\mu}\rho  (v-\epsilon)  \partial_{rr}v\partial_{rrr} \tilde{S}_2 \mathrm{d}r\mathrm{d}t
				-\eta\int_{0}^{t}\int_{1}^{+\infty}\frac{r^2}{2\mu}(\partial_{rr} \tilde{S}_2)^2\mathrm{d}r\mathrm{d}t\\
				&
				-C(\eta)\epsilon\int_{0}^{t}\int_{1}^{+\infty}r^2((\partial_{rr} \tilde{S}_1)^2
				+(\partial_{rr} \tilde{S}_2)^2)\mathrm{d}r\mathrm{d}t
				-C\left(E(0)+ E^{\frac{3}{2}}(t)+ E^{\frac{1}{2}}(t) \int_{0}^{t}\mathcal{D}(s)\mathrm{d}s\right).
			\end{align*}
		
		For the second term on the left-hand side of (\ref{70}), by applying $(\ref{60})_2$, we get
		\begin{align}\label{88}
				&\int_{0}^{t}\int_{1}^{+\infty}\dfrac{r^2\tau}{3\mu}\rho (v-\epsilon)\partial_{rr} v\partial_{rrr} \tilde{S}_1\mathrm{d}r\mathrm{d}t \nonumber\\
				=&\int_{0}^{t}\int_{1}^{+\infty}\dfrac{r^2\tau}{2\mu}\rho (v-\epsilon)\partial_{rr} v\big(\rho\partial_{trr} v+\rho v\partial_{rrr} v+\partial_{rrr} P-\dfrac{2}{r} \partial_{rr} \tilde{S}_1+\dfrac{4}{r^2} \partial_{r} \tilde{S}_1
				+\dfrac{4}{r^3}\tilde{S}_1-\partial_{rrr} \tilde{S}_2 \nonumber\\
				&\qquad\qquad\qquad\qquad\qquad\qquad\qquad\qquad\qquad\qquad\qquad\qquad\qquad\qquad\qquad\qquad\qquad\qquad-l_2\big)
				\mathrm{d}r\mathrm{d}t \nonumber\\
				\geq&\int_{1}^{+\infty}-\dfrac{\tau\epsilon}{4\mu}r^2\rho^2(\partial_{rr} v)^2\mathrm{d}r
				-C\epsilon\int_{0}^{t}\int_{1}^{+\infty}r^2\left((\partial_{rr} \tilde{S}_1)^2+(\partial_{rr} \tilde{S}_2)^2\right)\mathrm{d}r\mathrm{d}t
				\nonumber\\
				&+\int_{0}^{t}\int_{1}^{+\infty}\dfrac{\tau r^2}{2\mu}\rho v \partial_{rr}v P^{\prime}(\rho) \partial_{rrr}\rho\mathrm{d}r\mathrm{d}t
				-\int_{0}^{t}\int_{1}^{+\infty}\dfrac{r^2\tau}{2\mu}\rho (v-\epsilon) \partial_{rr}v\partial_{rrr} \tilde{S}_2 \mathrm{d}r\mathrm{d}t \nonumber\\
				&-\int_{0}^{t}\int_{1}^{+\infty}\dfrac{r^2\tau\epsilon}{2\mu}\rho P^{\prime}(\rho)\partial_{rr}v\partial_{rrr}\rho \mathrm{d}r\mathrm{d}t
				-C\left(E(0)+ E^{\frac{1}{2}}(t) \int_{0}^{t}\mathcal{D}(s)\mathrm{d}s\right)
			\end{align}
		where we use Lemmas \ref{lem1}, \ref{lem3} and \ref{lem5}, leading to
		\begin{equation}\nonumber
			\begin{aligned}
				&-\int_{0}^{t}\int_{1}^{+\infty}\frac{\tau r}{\mu}\rho(v-\epsilon) \partial_{rr} \tilde{S}_1\partial_{rr}v\mathrm{d}r\mathrm{d}t\\
				\geq&-\frac{\tau\epsilon}{\mu} \int_{0}^{t}\int_{1}^{+\infty}Cr^2 \left((\partial_{rr}v)^2+(\partial_{rr} \tilde{S}_1)^2\right)\mathrm{d}r\mathrm{d}t
				-C\left(E(0)+ E^{\frac{1}{2}}(t) \int_{0}^{t}\mathcal{D}(s)\mathrm{d}s\right)\\
				\geq&-C\epsilon\int_{0}^{t}\int_{1}^{+\infty}r^2\left((\partial_{rr} \tilde{S}_1)^2+(\partial_{rr} \tilde{S}_2)^2\right)\mathrm{d}r\mathrm{d}t
				-C\left(E(0)+ E^{\frac{3}{2}}(t)+ E^{\frac{1}{2}}(t) \int_{0}^{t}\mathcal{D}(s)\mathrm{d}s\right),
			\end{aligned}
		\end{equation}
		and
		\begin{equation}\nonumber
			\begin{aligned}
				&-\int_{0}^{t}\int_{1}^{+\infty}\frac{2\epsilon\tau}{\mu}\rho\partial_{r} \tilde{S}_1\partial_{rr}v\mathrm{d}r\mathrm{d}t
				-\int_{0}^{t}\int_{1}^{+\infty}\dfrac{2\tau\epsilon}{\mu r}\rho \partial_{rr}v\tilde{S}_1\mathrm{d}r\mathrm{d}t\\
				\geq&-C\epsilon\int_{0}^{t}\int_{1}^{+\infty}r^2\left((\partial_{rr} \tilde{S}_1)^2+(\partial_{rr} \tilde{S}_2)^2\right)\mathrm{d}r\mathrm{d}t
				-C\left(E(0)+ E^{\frac{3}{2}}(t)+ E^{\frac{1}{2}}(t) \int_{0}^{t}\mathcal{D}(s)\mathrm{d}s\right).
			\end{aligned}
		\end{equation}
		
		For the third term on the left-hand side of equation (\ref{70}), we have
		\begin{equation}
			\begin{aligned}
				&\int_{0}^{t}\int_{1}^{+\infty}\dfrac{r^2}{3\mu}\partial_{rr} v\partial_{rr} \tilde{S}_1\mathrm{d}r\mathrm{d}t
				\geq \int_{0}^{t}\int_{1}^{+\infty}\left(-C(\eta) \dfrac{r^2}{3\mu}(\partial_{rr} v)^2
				- \eta \dfrac{r^2}{3\mu}(\partial_{rr} \tilde{S}_1)^2\right) \mathrm{d}r\mathrm{d}t\\
				\geq& -C(\eta)\epsilon\int_{0}^{t}\int_{1}^{+\infty}r^2\left((\partial_{rr} \tilde{S}_1)^2+(\partial_{rr} \tilde{S}_2)^2\right)\mathrm{d}r\mathrm{d}t
				-\eta \int_{0}^{t}\int_{1}^{+\infty}\dfrac{r^2}{3\mu}(\partial_{rr} \tilde{S}_1)^2\mathrm{d}r\mathrm{d}t
			\end{aligned}
		\end{equation}
		where $\eta$ is sufficiently small.
		For the last term on the left-hand side of equation (\ref{70}), we have
			\begin{align*}
				&-\int_{0}^{t}\int_{1}^{+\infty}\frac{2}{3}\left(r^2\partial_{rrr}v \partial_{rr}v-r(\partial_{rr} v)^2+2\partial_{r}v\partial_{rr}v-\frac{2}{r}v\partial_{rr} v\right)\mathrm{d}r\mathrm{d}t\\
				\geq&\int_{0}^{t}\frac{1}{3}(\partial_{rr} v)^2(t,1) \mathrm{d}t
				-C\bigg(\epsilon\int_{0}^{t}\int_{1}^{+\infty}r^2\left((\partial_{rr} \tilde{S}_1)^2+(\partial_{rr} \tilde{S}_2)^2\right)\mathrm{d}r\mathrm{d}t\\
				&\qquad\qquad\qquad\qquad\qquad\qquad\quad+
				E(0)+ E^{\frac{3}{2}}(t)+ E^{\frac{1}{2}}(t) \int_{0}^{t}\mathcal{D}(s)\mathrm{d}s\bigg).
			\end{align*}

		Combining the above result, one can obtain
			\begin{align*}
				&\int_{0}^{t}\frac{1}{3}
				(\partial_{rr}v)^2(t,1)\mathrm{d}t\\
				\leq& \int_{1}^{+\infty}\left(
				\frac{\tau r^2}{2\mu}\rho\partial_{rr} v (\partial_{rr} P
				-\frac{2}{3}\partial_{rr} \tilde{S}_1
				-\frac{2}{r}\partial_{r}\tilde{S}_1+\frac{2}{r^2}\tilde{S}_1-\partial_{rr}\tilde{S}_2 ) 
				+\frac{\epsilon\tau r^2}{4\mu}\rho^2(\partial_{rr} v)^2 \right)\mathrm{d}r\\
				&+\eta\int_{0}^{t}\int_{1}^{+\infty}
				\frac{r^2}{2\mu}\left((\partial_{rr} \tilde{S}_1)^2+(\partial_{rr} \tilde{S}_2)^2\right)
				\mathrm{d}r\mathrm{d}t
				+\int_{0}^{t}\int_{1}^{+\infty}\dfrac{r^2\tau\epsilon}{2\mu}\rho P^{\prime}(\rho)\partial_{rr}v\partial_{rrr}\rho \mathrm{d}r\mathrm{d}t\\
				&+C(\eta)\epsilon\int_{0}^{t}\int_{1}^{+\infty}r^2\left((\partial_{rr} \tilde{S}_1)^2+(\partial_{rr} \tilde{S}_2)^2\right)\mathrm{d}r\mathrm{d}t\\
				\leq&\frac{\tau C}{2\mu}\int_{1}^{+\infty} r^2\rho \left(\eta(\partial_{rr}v)^2+C(\eta)(\partial_{tr}v)^2
				+ \epsilon (\partial_{rr}v)^2 \right)\mathrm{d}r
				+\eta\int_{0}^{t}\int_{1}^{+\infty}
				\frac{r^2}{2\mu}\left((\partial_{rr} \tilde{S}_1)^2+(\partial_{rr} \tilde{S}_2)^2\right)
				\mathrm{d}r\mathrm{d}t\\
				&+\int_{0}^{t}\int_{1}^{+\infty}\dfrac{r^2\tau\epsilon}{2\mu}\rho P^{\prime}(\rho)\partial_{rr}v\partial_{rrr}\rho \mathrm{d}r\mathrm{d}t
				+C(\eta)\epsilon\int_{0}^{t}\int_{1}^{+\infty}r^2\left((\partial_{rr} \tilde{S}_1)^2+(\partial_{rr} \tilde{S}_2)^2\right)\mathrm{d}r\mathrm{d}t\\
				&
				+C\left(E(0)+ E^{\frac{3}{2}}(t)+ E^{\frac{1}{2}}(t) \int_{0}^{t}\mathcal{D}(s)\mathrm{d}s\right).
			\end{align*}
		In view of equation $(\ref{60})_1$ and $(\ref{60})_2$ and Lemmas \ref{lem5}, we obtain
		\begin{equation}\nonumber
			\begin{aligned}
				&\int_{0}^{t}\int_{1}^{+\infty}\dfrac{r^2\tau\epsilon}{2\mu}\rho P^{\prime}(\rho)\partial_{rr}v\partial_{rrr}\rho \mathrm{d}r\mathrm{d}t\\
				=&-\int_{0}^{t}\dfrac{r^2\tau\epsilon}{2\mu}\rho P^{\prime}(\rho)\partial_{rr}v\partial_{rr}\rho(t,1)\mathrm{d}t-
				\int_{0}^{t}\int_{1}^{+\infty}\dfrac{\tau\epsilon}{2\mu}[r^2\rho P^{\prime}(\rho)\partial_{rr}v]_r\partial_{rr}\rho \mathrm{d}r\mathrm{d}t\\
				\leq&-\int_{0}^{t}\dfrac{r^2\tau\epsilon}{2\mu}\rho P^{\prime}(\rho)\partial_{rr}v\partial_{rr}\rho(t,1)\mathrm{d}t
				+\int_{0}^{t}\int_{1}^{+\infty}\dfrac{r^2\tau\epsilon}{2\mu}P^{\prime}(\rho)(\partial_{trr} \rho+v\partial_{rrr}\rho+\frac{2}{r}\rho\partial_{rr}v+\frac{4}{r^3}\rho v)\partial_{rr}\rho\mathrm{d}r\mathrm{d}t\\
				&-
				\int_{0}^{t}\int_{1}^{+\infty}\dfrac{\tau\epsilon}{2\mu}2r\rho P^{\prime}(\rho)\partial_{rr}v\partial_{rr}\rho \mathrm{d}r\mathrm{d}t
				+ CE^{\frac{1}{2}}(t) \int_{0}^{t}\mathcal{D}(s)\mathrm{d}s
				\\
				\leq& \frac{C\tau\epsilon}{2\mu}\int_{0}^{t}\left(C(\eta)(\partial_{rr}v)^2(t,1)+\eta(\partial_{rr}\rho)^2(t,1)\right)\mathrm{d}t
				-\int_{1}^{+\infty}\dfrac{\tau\epsilon}{4\mu}P^{\prime}(\rho)(\partial_{rr}\rho)^2\big|_0^t\mathrm{d}r\\
				&
				+ C\left(E(0)+E^{\frac{3}{2}}(t) +E^{\frac{1}{2}}(t) \int_{0}^{t}\mathcal{D}(s)\mathrm{d}s\right)\\
				\leq& \frac{C\tau\epsilon}{2\mu}\int_{0}^{t}C(\eta)(\partial_{rr}v)^2(t,1)\mathrm{d}t
				+\frac{C\tau\epsilon}{\mu}\int_{0}^{t}\eta\left(
				(\partial_{tr}v)^2+(\partial_{rr} \tilde{S}_1)^2+(\partial_{rr} \tilde{S}_2)^2+(\partial_{r} \tilde{S}_1)^2+(\tilde{S}_1)^2\right)(t,1)\mathrm{d}t\\
				&-\int_{1}^{+\infty}\dfrac{\tau\epsilon}{4\mu}P^{\prime}(\rho)(\partial_{rr}\rho)^2\mathrm{d}r
				+ C\left(E(0)+E^{\frac{3}{2}}(t) +E^{\frac{1}{2}}(t) \int_{0}^{t}\mathcal{D}(s)\mathrm{d}s\right)\\
				\leq&
				\frac{C\tau\epsilon}{2\mu}\int_{0}^{t}C(\eta)(\partial_{rr}v)^2(t,1)\mathrm{d}t
				+\frac{C\tau\epsilon}{\mu}\int_{0}^{t}\eta\left(
				(\partial_{rr} \tilde{S}_1)^2+(\partial_{rr} \tilde{S}_2)^2\right)(t,1)\mathrm{d}t\\
				&
				+ C\left(E(0)+E^{\frac{3}{2}}(t) +E^{\frac{1}{2}}(t) \int_{0}^{t}\mathcal{D}(s)\mathrm{d}s\right)
			\end{aligned}
		\end{equation}
		since $P^{\prime}(\rho)>0$.
	
		Thus, 
		by using Lemmas $\ref{lem1}, \ref{lem2-1}$ and \ref{lem5}, we finally have
		
		\begin{equation}
		\begin{aligned}\label{82}
				&\int_{0}^{t}(\partial_{rr} v)^2(t,1)\mathrm{~d}t \\
				\leq& 
				\frac{\tau C}{2\mu}(\eta+\epsilon)\int_{1}^{+\infty} r^2\rho(\partial_{rr}v)^2\mathrm{d}r +\eta\int_{0}^{t}\int_{1}^{+\infty}\frac{r^2}{2\mu}\left((\partial_{rr}\tilde{S}_1)^2+(\partial_{rr}\tilde{S}_2)^2\right)\mathrm{d}r\mathrm{d}t \\&
				+\frac{C\tau\epsilon}{2\mu}\int_{0}^{t}C(\eta)(\partial_{rr}v)^2(t,1)\mathrm{d}t
				+\frac{C\tau\epsilon}{\mu}\int_{0}^{t}\eta\left(
				(\partial_{rr} \tilde{S}_1)^2+(\partial_{rr} \tilde{S}_2)^2\right)(t,1)\mathrm{d}t \\&
				+C(\eta)\epsilon\int_{0}^{t}\int_{1}^{+\infty}\left((\partial_{rr} \tilde{S}_1)^2+(\partial_{rr} \tilde{S}_2)^2\right)\mathrm{d}r\mathrm{d}t
				+C\left(E(0)+ E^{\frac{3}{2}}(t)+ E^{\frac{1}{2}}(t) \int_{0}^{t}\mathcal{D}(s)\mathrm{d}s\right)
		\end{aligned}
		\end{equation}
	where \(\tau\epsilon\) is sufficiently small. Thus, the desired result can be obtained immediately.
	\end{proof}
	
	Next, according to Lemmas \ref{lem8},\ref{lem7_tx} and \ref{lem7_xx}, we can obtain the following corollary.
	\begin{corollary}\label{cor1}
		There exists a constant C such that for any $0\leq t \leq T$
		\begin{equation}
			\begin{aligned}
				&\int_{1}^{+\infty}\bigg(\left(r\partial_{rr} \rho+r\partial_{r} \rho\right)^2
				+r^2 (\partial_{rr}v)^2+(\partial_{r}v)^2
				+\tau\left(r^2(\partial_{rr}\tilde{S}_1)^2+(\partial_{r}\tilde{S}_1)^2+\frac{1}{ r^2}(\tilde{S}_1)^2\right)\\
				&\qquad\qquad
				+\tau\left(r\partial_{rr} \tilde{S}_2+2\partial_{r} \tilde{S}_2\right)^2\bigg)\mathrm{d}r
				+\int_{0}^{t}\left((\partial_{rr}v)^2(t,1)+(\partial_{tr}v)^2(t,1)
				\right)\mathrm{d}t\\
				&+\int_0^t\int_{1}^{+\infty}\bigg((r^2(\partial_{rr}\tilde{S}_1)^2+(\partial_{r}\tilde{S}_1)^2+\frac{1}{ r^2}(\tilde{S}_1)^2)
				+\left(r\partial_{rr} \tilde{S}_2+2\partial_{r} \tilde{S}_2\right)^2\bigg)\mathrm{d}r\mathrm{d}t\\
				&\leq C \left(E(0)+E^{\frac{3}{2}}(t)+E^{\frac{1}{2}}(t) \int_{0}^{t}\mathcal{D}(s)\mathrm{d}s\right).	
			\end{aligned}
		\end{equation}
	\end{corollary}
	
	\begin{proof}
		Indeed, multiplying  (\ref{bdy}) by $C+1$ and adding it to (\ref{sed-rr}).
		Choose a sufficiently small $\eta$ such that $\frac{(C+1)}{2\mu}\eta<\frac{1}{2}$,
		and then fix $\eta$, we have 
		\begin{equation}
			\begin{aligned}
				&\int_{1}^{+\infty}\bigg(\left(r\partial_{rr} \rho+r\partial_{r} \rho\right)^2
				+\frac{1}{2}r^2 (\partial_{rr}v)^2+(\partial_{r}v)^2
				+\tau\left(r^2(\partial_{rr}\tilde{S}_1)^2+(\partial_{r}\tilde{S}_1)^2+\frac{1}{ r^2}(\tilde{S}_1)^2\right)\\
				&\qquad\qquad
				+\tau\left(r\partial_{rr} \tilde{S}_2+2\partial_{r} \tilde{S}_2\right)^2\bigg)\mathrm{d}r
				+\int_{0}^{t}\left((\partial_{rr}v)^2(t,1)+(\partial_{tr}v)^2(t,1)
				\right)\mathrm{d}t\\
				&+\int_0^t\int_{1}^{+\infty}\bigg((\frac{1}{2}r^2(\partial_{rr}\tilde{S}_1)^2+(\partial_{r}\tilde{S}_1)^2+\frac{1}{ r^2}(\tilde{S}_1)^2)
				+\frac{1}{2}\left(r\partial_{rr} \tilde{S}_2+2\partial_{r} \tilde{S}_2\right)^2\bigg)\mathrm{d}r\mathrm{d}t\\
				\leq &
				\frac{\tau C}{\mu}\epsilon\int_{1}^{+\infty}r^2(\partial_{rr}v)^2\mathrm{d}r
				+C(\eta)\epsilon\int_{0}^{t}\int_{1}^{+\infty}r^2((\partial_{rr}\tilde{S}_1)^2+(\partial_{rr}\tilde{S}_2)^2)\mathrm{d}r\mathrm{d}t\\
				&+C \left(E(0)+E^{\frac{3}{2}}(t)+E^{\frac{1}{2}}(t) \int_{0}^{t}\mathcal{D}(s)\mathrm{d}s\right).	
			\end{aligned}
		\end{equation}
		For a fixed $\eta$, by choosing a sufficiently small $\epsilon$, and applying Lemma \ref{lem7_tx}, we can obtain the desired result.
	\end{proof}
	On the other hand, using $(\ref{22})_2$, Lemma \ref{lem8} and Corollary \ref{cor1}, we can easily get the second order
	dissipation of $\rho$.
	\begin{corollary}
		There exists some constant C such that
		\begin{equation}
			\int_{0}^{t}\int_{1}^{+\infty}r^2(\partial_{rr}\rho)^2\mathrm{d}r\mathrm{d}t
			\leq C \left(E(0)+E^{\frac{3}{2}}(t)+E^{\frac{1}{2}}(t) \int_{0}^{t}\mathcal{D}(s)\mathrm{d}s\right).
		\end{equation}
	\end{corollary}
	Combining Lemmas Lemma \ref{lem3}-\ref{lem8} and Corollary \ref{cor1},  we get
	\begin{lemma}\label{lem10}
		There exists some constant C such that
		\begin{equation}
			\begin{aligned}
				&\left\|rD(\rho, v, \sqrt{\tau} \tilde{S}_1, \sqrt{\tau} \tilde{S}_2)\right\|_{H^{1}}^2
				+\tau^2\left\|r\partial_{t}^2(\rho, v, \sqrt{\tau} \tilde{S}_1, \sqrt{\tau} \tilde{S}_2)\right\|_{L^{2}}^2\\
				&+\int_{0}^{t}\bigg(\left\|r D^2(\rho, v)\right\|_{L^2}^2
				+\left\|rD(\tilde{S}_1, \tilde{S}_2)\right\|_{H^{1}}^2
				+\tau^2\left\|r\partial_{t}^2(\tilde{S}_1,\tilde{S}_2)\right\|_{L^{2}}^2\bigg)\mathrm{~d} t\\&
				\leq C \left(E(0)+E^{\frac{3}{2}}(t)+E^{\frac{1}{2}}(t) \int_{0}^{t}\mathcal{D}(s)\mathrm{d}s\right).
			\end{aligned}
		\end{equation}
	\end{lemma}
	Therefore, combining Lemma \ref{lem3.5} and \ref{lem10}, the Proposition 3.1 follows immediately.
	
	\section{Passing to the Limit and proof of Main Theorems}
	
	[\textbf{Proof of Theorem \ref{thm1}}]
	According to Proposition \ref{prop1}, the local solution $(\rho^\epsilon,v^\epsilon,\tilde{S}_1^{\epsilon},\tilde{S}_2^{\epsilon})$ can
	be extend to $[0,+\infty)$ by usual continuation methods. Thus, for fixed $\epsilon$, we get a global solution
	$(\rho^\epsilon,v^\epsilon,\tilde{S}_1^{\epsilon},\tilde{S}_2^{\epsilon})$ to system (\ref{9})-(\ref{11}) satisfying
	\begin{equation}\label{4.1}
		\begin{aligned}
			&\sup _{0 \leq s \leq t}\bigg( \sum_{k=0}^1\left\|r\partial_{t}^k(\rho-1, v, \sqrt{\tau} \tilde{S}_1, \sqrt{\tau} \tilde{S}_2)\right\|_{H^{2-k}}^2
			+\tau^2\left\|r\partial_{t}^2(\rho-1, v, \sqrt{\tau} \tilde{S}_1, \sqrt{\tau} \tilde{S}_2)\right\|_{L^{2}}^2\bigg)\\
			&+\int_{1}^{+\infty}\bigg(\sum_{\alpha=1}^2\left\|r D^\alpha(\rho, v)\right\|_{L^2}^2
			+\sum_{k=0}^1\left\|r\partial_{t}^k(\tilde{S}_1, \tilde{S}_2)\right\|_{H^{2-k}}^2
			+\tau^2\left\|r\partial_{t}^2(\tilde{S}_1,\tilde{S}_2)\right\|_{L^{2}}^2\bigg)\mathrm{~d} t
			\leq C E(0)
		\end{aligned}
	\end{equation}
	where C is a constant independent of $\tau$ and $\epsilon$. Therefore, for fixed $\tau$, the uniform bounds of $(\rho^\epsilon-1,v^\epsilon, \sqrt{\tau}\tilde{S}_1^{\epsilon}, \sqrt{\tau}\tilde{S}_2^{\epsilon})$ in $L^{\infty}\left([0,\infty), H^2(\Omega)\right)$
	implies that there exists $(\rho^\epsilon-1,v^\epsilon, \sqrt{\tau}\tilde{S}_1^{\epsilon}, \sqrt{\tau}\tilde{S}_2^{\epsilon})
	\in  L^{\infty}\left([0,\infty), H^2(\Omega)\right)$ such that
	
	$$
	(\rho^\epsilon,v^\epsilon, \sqrt{\tau}\tilde{S}_1^{\epsilon}, \sqrt{\tau}\tilde{S}_2^{\epsilon}) \rightharpoonup
	(\rho,v, \sqrt{\tau}\tilde{S}_1, \sqrt{\tau}\tilde{S}_2) \quad \text { weakly - } * \quad \text { in } \quad L^{\infty}\left(\mathbb{R}^{+} ; H^2(\Omega)\right).
	$$
	
	Further, the uniform boundedness of $\partial_{t}^k(\rho-1, v, \sqrt{\tau} \tilde{S}_1, \sqrt{\tau} \tilde{S}_2)$, $k=1,2$ in
	$L^{\infty}\left([0, T], H^{2-k}\right) \cap L^2\left([0, T], H^{2-k}\right)$ for any $T>0$ ensures, via established compactness theorems, that the sequence $(\rho^\epsilon,v^\epsilon, \sqrt{\tau}\tilde{S}_1^{\epsilon}, \sqrt{\tau}\tilde{S}_2^{\epsilon})$ is relative compact in
	$C^0\left([0, T], H^{2-\delta_0}\right) \cap C^1\left([0, T], H^{1-\delta_0}\right)$ for any $\delta_0>0$. 
	As a consequence, as $\delta\rightarrow 0$ and up to subsequences,
	$$
	(\rho^\epsilon,v^\epsilon, \tilde{S}_1^{\epsilon}, \tilde{S}_2^{\epsilon}) \rightarrow(\rho,v, \tilde{S}_1, \tilde{S}_2), \quad \text { strongly } \quad \text { in } \quad C^0\left([0, T], H^{2-\delta_0}\right) \cap C^1\left([0, T], H^{1-\delta_0}\right).
	$$
	
	This strong convergence suffices to pass to the limit in (\ref{9}) and (\ref{4.1}), thereby confirming that 
	$(\rho,v, \tilde{S}_1, \tilde{S}_2)$ is a classical solution to system $(\ref{1.4})-(\ref{1.6})$ satisfying (\ref{1.8}). The uniqueness follows directly from the regularity
	$$
	(\rho, v, \tilde{S}_1, \tilde{S}_2) \in C\left([0, \infty), H^{2-\delta_0}(\Omega)\right) \cap L^{\infty}\left((0, \infty), H^2(\Omega)\right) \cap L^2\left((0, \infty), H^1(\Omega)\right),
	$$
	which completes the proof of Theorem \ref{thm1}.

	[\textbf{Proof of Theorem \ref{thm2}}]
	Let $(\rho^{\tau},v^{\tau},\tilde{S}_1^{\tau},\tilde{S}_2^{\tau})$ be the global solutions obtained in Theorem \ref{thm1} satisfying (\ref{1.8}). 
	
	Then there exist limit functions
	$
	(\rho^{0},v^{0}) \in L^{\infty}\left(\mathbb{R}^{+} ; H^2(\Omega)\right), (\tilde{S}_1^{0},\tilde{S}_2^{0}) \in L^2\left(\mathbb{R}^{+} ; H^2(\Omega)\right),
	$
	such that along a subsequence as $\tau\rightarrow 0$:
	\begin{equation}
		\begin{aligned}
			& \left(\rho^{\tau},v^{\tau}\right) \rightharpoonup^*\left(\rho^{0},v^{0}\right) \quad \text { weakly-* in } L^{\infty}\left([0,\infty), H^2(\Omega)\right), \\
			& \left(\tilde{S}_1^{\tau},\tilde{S}_2^{\tau}\right) \rightharpoonup\left(\tilde{S}_1^{0},\tilde{S}_2^{0}\right) \quad \text { weakly in } L^2\left([0,\infty), H^2(\Omega)\right) .
		\end{aligned}
	\end{equation}
	The uniform boundedness of $\partial_{t}^k(\rho^{\tau},v^{\tau})$ in $L^2\left([0,\infty), H^{2-k}(\Omega)\right)$ for $k=1,2$ ensures, via classical compactness arguments, that the sequence $(\rho^{\tau},v^{\tau})$ is relatively compact in
	$C^0\left([0, T], H^{2-\delta_0}\right) \cap C^1\left([0, T], H^{1-\delta_0}\right)$ for any $\delta_0>0$ and $T>0$. Consequently, along further subsequences as $\tau$$\rightarrow$0, 
	$$
	(\rho^\tau,v^\tau) \rightarrow(\rho^0,v^0), \quad \text { strongly } \quad \text { in } \quad C^0\left([0, T], H^{2-\delta_0}(\Omega)\right) \cap C^1\left([0, T], H^{1-\delta_0}(\Omega)\right).
	$$
	
	From the uniform estimate (\ref{1.8}) and the constitutive relations $(\ref{1.4})_4-(\ref{1.4})_5$, we derive  
	$$
	\sup_{t \geq 0} \|(\tilde{S}_1^\tau, \tilde{S}_2^\tau)\|_{H^1} \leq C E(0).
	$$
	
	Since $\partial_{t}(\tilde{S}_1^{\tau},\tilde{S}_2^{\tau})$ are bounded in $L^2([0,T], H^1(\Omega))$,
	the sequence $(\tilde{S}_1^{\tau},\tilde{S}_2^{\tau})$ is relatively compact in $C^0\left([0, T], H^{1-\delta}(\Omega)\right)$
	for any $T > 0$. Therefore, as $\tau$$\rightarrow$0 and up to subsequences,
	$$
	\left(\tilde{S}_1^{\tau},\tilde{S}_2^{\tau}\right) \rightarrow\left(\tilde{S}_1^{0},\tilde{S}_2^{0}\right) \quad \text { strongly in } C^0\left([0, T], H^{1-\delta}(\Omega)\right).
	$$
	
	The uniform bound on $\sqrt{\tau}(\tilde{S}_1^{\tau},\tilde{S}_2^{\tau})$ implies:
	\[
	\tau(\tilde{S}_1^{\tau},\tilde{S}_2^{\tau}) \to (0, 0) \quad \text{in } L^\infty([0, \infty), H^2(\Omega)),
	\]
	and consequently in the distributional sense:
	\[
	\tau (\partial_t\tilde{S}_1^\tau, \partial_t \tilde{S}_2^\tau) \to (0, 0) \quad \text{in } D'((0, \infty) \times \Omega)
	\quad \text{as} \quad \tau\rightarrow0.
	\]
	Passing to the limit in equations $(\ref{1.4})_3$ and $(\ref{1.4})_4$, we have
	\begin{equation}\label{4.3}
		\tilde{S}_1^0=2\mu(\partial_{r}v^0-\frac{v^0}{r}), \quad \tilde{S}_2^0=\lambda(\partial_{r}v^0+\frac{2v^0}{r}). \quad \text{a.e. in } (0, \infty) \times \Omega.
	\end{equation}

	Finally, passing to the limit in $(\ref{1.4})_1-(\ref{1.4})_2$ and using (\ref{4.3}), the limit functions \((\rho^0, v^0)\) satisfy the classical compressible Navier-Stokes system. This completes the proof of Theorem \ref{thm2}.

	%\bibliography{ref}

\end{document}